\documentclass[xcolor=dvipsnames,svgnames,table,reqno,11pt]{amsart}

\usepackage[left=1.5in, right=1.5in, bottom=1.25in, height=8in]
{geometry}

\input xy
\xyoption{all}
\usepackage{accents}
\usepackage{epsfig}
\usepackage{xcolor}
\usepackage{amsthm}
\usepackage{amssymb}
\usepackage{bm} 
\usepackage{amsmath}
\usepackage{amscd}
\usepackage{amsopn}
\usepackage{graphicx}

\usepackage{tikz}
\usetikzlibrary{shapes,shadows,arrows}
\usepackage{tikz-cd}

\usepackage{xspace}

\usepackage{hhline}
\usepackage{easybmat}
\usepackage{caption}   
\usepackage{relsize}

\usepackage{url}
\usepackage{enumitem, hyperref}\hypersetup{colorlinks}

\usepackage{stmaryrd}

\usepackage[T1]{fontenc}

\colorlet{purpleB70}{blue!70!red}

\colorlet{orangeR65}{red!65!yellow}

\definecolor{red2}{HTML}{d41173}

\definecolor{neongreen}{HTML}{1bf702}

\definecolor{radicalred}{HTML}{FF355E}

\definecolor{denim}{HTML}{1560BD}

\definecolor{darkcyan}{rgb}{0.0, 0.55, 0.55}

\definecolor{cilek}{HTML}{FF43A4}

\definecolor{mor}{HTML}{9F00C5}

\definecolor{phlox}{rgb}{0.87, 0.0, 1.0}

\definecolor{fluorescentpink}{HTML}{FF1493}

\definecolor{napiergreen}{rgb}{0.16, 0.5, 0.0}

\definecolor{kellygreen}{rgb}{0.3, 0.73, 0.09}

\definecolor{parisgreen}{HTML}{ 50C878 }

\definecolor{palatinateblue}{rgb}{0.15, 0.23, 0.89}

\definecolor{ceruleanblue}{rgb}{0.16, 0.32, 0.75}

\definecolor{brandeisblue}{rgb}{0.0, 0.44, 1.0}

\definecolor{KLMblue}{HTML}{0FC0FC}

\definecolor{cinnamon}{rgb}{0.82, 0.41, 0.12}

\definecolor{darkorange}{rgb}{1.0, 0.55, 0.0}

\definecolor{darktangerine}{rgb}{1.0, 0.66, 0.07}

\definecolor{deepcarrotorange}{rgb}{0.91, 0.41, 0.17}

\definecolor{internationalorange}{HTML}{FF4F00}

\definecolor{persimmon}{HTML}{EC5800}

\definecolor{pumpkin}{HTML}{FF7518}

\definecolor{darkred}{rgb}{1,0,0} 
\definecolor{darkgreen}{rgb}{0,0.7,0}
\definecolor{darkblue}{rgb}{0,0,1}

\hypersetup{colorlinks,
linkcolor=darkblue,
filecolor=darkgreen,
urlcolor=darkred,
citecolor=darkgreen}

\makeatletter
\def\reflb#1#2{\begingroup
    #2%
    \def\@currentlabel{#2}%
    \phantomsection\label{#1}\endgroup
}
\makeatother

\numberwithin{equation}{section}
\newtheorem{Theorem}{Theorem}
\numberwithin{Theorem}{section}

\newtheorem{TheoremX}{Theorem}

\newtheorem{CorollaryX}[TheoremX]{Corollary}

\newtheorem   {Lemma}[Theorem]{Lemma}

\newtheorem*  {Conjecture}{Conjecture}
\newtheorem   {Proposition}[Theorem]{Proposition}
\newtheorem   {Corollary}[Theorem]{Corollary}
\theoremstyle {definition}
\newtheorem   {Definition}[Theorem]{Definition}
\theoremstyle {remark}
\newtheorem   {Remark}[Theorem]{Remark}
\newtheorem   {Example}[Theorem]{Example}

\def    \eps    {\epsilon}

\newcommand{\CA}{{\mathcal A}}

\newcommand{\DD}{{\mathcal D}}

\newcommand{\CI}{{\mathcal I}}

\newcommand{\CS}{{\mathcal S}}

\newcommand{\supp}{\operatorname{supp}}

\newcommand{\id}{{\mathit id}}

\newcommand{\const}{{\mathit const}}

\newcommand{\charr}{{\mathit char}\,}

\newcommand{\spl}{{\frak {sp}}}

\newcommand{\gl}{{\frak {gl}}}

\newcommand{\fs}{{\mathfrak s}}

\newcommand{\tU}{\tilde{U}}

\newcommand{\hx}{\hat{x}}
\newcommand{\cx}{\check{x}}

\newcommand{\tPhi}{\tilde{\Phi}}
\newcommand{\tH}{\tilde{H}}

\newcommand{\A}{{\mathcal A}}

\newcommand{\CB}{{\mathcal B}}

\newcommand{\PP}{{\mathcal P}}

\def    \F      {{\mathbb F}}

\def    \C      {{\mathbb C}}
\def    \R      {{\mathbb R}}

\def    \Z      {{\mathbb Z}}
\def    \N      {{\mathbb N}}
\def    \Q      {{\mathbb Q}}

\def    \T      {{\mathbb T}}
\def    \CP     {{\mathbb C}{\mathbb P}}

\def    \12     {{\frac{1}{2}}}

\def    \p      {\partial}

\def    \im     {\operatorname{im}}

\def    \SH     {\operatorname{SH}}
\def    \COH    {\operatorname{CH}}
\def    \Sp     {\operatorname{Sp}}
\def    \tSp    {\operatorname{\widetilde{Sp}}}
\def    \tSpn    {\widetilde{\operatorname{Sp}}{}^*}

\def    \HF     {\operatorname{HF}}

\def    \H      {\operatorname{H}}

\def    \CF      {\operatorname{CF}}

\def    \vDelta  {\vec{\Delta}}

\def    \sgn   {\mathit{sgn}}
\def    \sign   {\mathit{sign}}

\def    \hmu   {\operatorname{\hat{\mu}}}

\def    \coker {\operatorname{coker}}
\def    \TSp     {\widetilde{\operatorname{Sp}}}

\newcommand \beg   {\operatorname{Beg}}
\newcommand \en   {\operatorname{End}}
\newcommand \chieq   {\chi^{\scriptscriptstyle{eq}}}

\newcommand \Cbar {C_{\scriptscriptstyle{bar}}}
\newcommand \Kbar {K_{\scriptscriptstyle{bar}}}
\newcommand \Phideg {\Phi_{\scriptscriptstyle{deg}}}

\newcommand   \WW {\widehat{W}}
\newcommand   \UU {\widehat{U}}
\newcommand{\ttheta}{\tilde{\theta}}

\newcommand{\iddots}{\reflectbox{$\ddots$}}

\begin{document}


\setlength{\smallskipamount}{6pt}
\setlength{\medskipamount}{10pt}
\setlength{\bigskipamount}{16pt}





\title [Closed Orbits of Dynamically Convex Reeb Flows]{Closed Orbits
  of Dynamically Convex Reeb Flows: Towards the HZ- and Multiplicity
  Conjectures}

\author[Erman \c C\. inel\. i]{Erman \c C\. inel\. i}
\author[Viktor Ginzburg]{Viktor L. Ginzburg}
\author[Ba\c sak G\"urel]{Ba\c sak Z. G\"urel}

\dedicatory{\normalsize{In memory of Edi Zehnder}}

\address{E\c C: ETH Z\"urich, R\"amistrasse 101, 8092 Z\"urich,
  Switzerland}
\email{erman.cineli@math.eth.ch}

\address{VG: Department of Mathematics, UC Santa Cruz, Santa Cruz, CA
  95064, USA} \email{ginzburg@ucsc.edu}

\address{BG: Department of Mathematics, University of Central Florida,
  Orlando, FL 32816, USA} \email{basak.gurel@ucf.edu}



\date{\today} 

\thanks{The work is partially supported by the NSF grants DMS-2304207
  (BG) and DMS-2304206 (VG), the Simons Foundation grants 855299 (BG)
  and MP-TSM-00002529 (VG), and the ERC Starting Grant 851701 via a
  postdoctoral fellowship (E\c{C})}

\begin{abstract}
  We study the multiplicity problem for prime closed orbits of
  dynamically convex Reeb flows on the boundary of a star-shaped
  domain in $\R^{2n}$. The first of our two main results asserts that
  such a flow has at least $n$ prime closed Reeb orbits, improving the
  previously known lower bound by a factor of two and settling a
  long-standing open question.  The second main theorem is that when,
  in addition, the domain is centrally symmetric and the Reeb flow is
  non-degenerate, the flow has either exactly $n$ or infinitely many
  prime closed orbits.  This is a higher-dimensional contact variant
  of Franks' celebrated $2$-or-infinity theorem and, viewed from the
  symplectic dynamics perspective, settles a particular case of the
  contact Hofer--Zehnder conjecture.  The proofs are based on several
  auxiliary results of independent interest on the structure of the
  filtered symplectic homology and the properties of closed orbits.
\end{abstract}

\maketitle

\vspace{-0.2in}


\tableofcontents

\section{Introduction and main results}
\label{sec:intro+results}

\subsection{Introduction}
\label{sec:intro}
We study the multiplicity problem for closed characteristics, aka
prime closed orbits, of the Reeb flow on the boundary of a star-shaped
domain in $\R^{2n}$. The first of our two main results, Theorem
\ref{thm:mult}, asserts that such a flow has at least $n$ prime closed
Reeb orbits whenever the flow is dynamically convex (and the boundary
is smooth), improving the previously known lower bound by a factor of
two. The lower bound $n$ is obviously sharp.

The second main theorem, Theorem \ref{thm:HZ}, is that when, in
addition, the domain is centrally symmetric and the Reeb flow is
non-degenerate, the flow has either exactly $n$ or infinitely many
prime closed orbits.  This is a higher-dimensional contact variant of
the celebrated Franks' theorem on periodic points of area preserving
diffeomorphisms of $S^2$. From the symplectic dynamics perspective,
the theorem settles a contact version of the Hofer--Zehnder (HZ)
conjecture discussed below.

The proofs are based on several auxiliary results on the structure of
the filtered symplectic homology and the properties of closed orbits,
which are of independent interest.  For instance, we show that this
homology is one-dimensional for any background field and any positive
action threshold whenever the number of prime closed orbits is finite
and the flow is dynamically convex.  Below we provide more context for
this work.

Broadly understood, the study of closed orbits of Reeb flows on
contact type hypersurfaces $M$ in $\R^{2n}$, e.g., on convex
hypersurfaces, goes back at least to Lyapunov's times if not
earlier. Here we exclusively focus on the case where the hypersurface
is the boundary of a star-shaped domain or, equivalently, on Reeb
flows on the standard contact sphere $S^{2n-1}$, thus fixing the
underlying contact topology.  The following conjecture encompasses
much of what we know or can expect to be true about prime closed
orbits of such flows.

\begin{Conjecture}[The $n$-or-$\infty$ conjecture]
  Let $M$ be the boundary of a star-shaped domain in $\R^{2n}$, which
  we assume to be smooth. Then

  \begin{itemize}

  \item[\reflb{C1}{\rm{(C-M)}}] {\rm Multiplicity:} the Reeb flow on
    $M$ has at least $n$ prime closed orbits;
      
  \item[\reflb{C2}{\rm{(C-HZ)}}] {\rm HZ-conjecture:} the Reeb flow on
    $M$ has exactly $n$ prime closed orbits whenever the flow is a
    Reeb pseudo-rotation, i.e., the number of prime closed orbits is
    finite.

  \end{itemize}
  \end{Conjecture}

  Here we deliberately broke down the question into two parts,
  \ref{C1} and \ref{C2}, for our understanding of them and their
  status are very different. Overall, very little is known about
  either part of the conjecture in such a generality. For instance,
  without extra conditions on $M$, it is not even known that there
  must be more than one prime orbit or more than two if
  the flow is non-degenerate. (See, however, \cite{AGKM}.)  The
  existence of at least one orbit was proved in \cite{Ra} (see also
  \cite{We1}) and served as the basis of the Weinstein conjecture,
  \cite{We2}, later established in \cite{Vi:WC} for all contact type
  hypersurfaces in $\R^{2n}$. We expect the $n$-or-$\infty$ conjecture
  to be true for a very broad class of Reeb flows on $M$, but possibly
  not all. Once some additional restrictions on $M$
    are imposed, usually along the lines of convexity or dynamical
    convexity or symmetry, much more is known.

  In various settings the conjecture is far from new. For
  instance, Part \ref{C1}, often referred to as the multiplicity
  problem, is stated in, say, \cite{ALM} in full generality exactly as
  above. For convex hypersurfaces in $\R^{2n}$, it is included in
  Ekeland's monograph, \cite[p.\ 198]{Ek}, and apparently it was a
  well-known open question already then. Thus, since convex
  hypersurfaces are dynamically convex (see \cite{HWZ:convex}), our
  first result establishes Ekeland's variant of the conjecture and, in
  fact, is more general.  (Dynamical convexity is the lower bound
  $n+1$, the same as for ellipsoids, for the Conley--Zehnder index of
  closed orbits. This condition is in general less restrictive than
  convexity; see \cite{CE1, CE2, CGHi, DGRZ} and also \cite{AM, GM}.)
  For non-degenerate flows, Ekeland's conjecture was proved in
  \cite{LZ} (see also \cite{Lo}) for convex hypersurfaces and in
  \cite{GK} in the dynamically convex case. Recently, again in the
  non-degenerate setting, the dynamical convexity condition has been
  significantly relaxed to a very minor restriction on the index; see
  \cite{DLLW1, DLLW2, GGMa} and references therein.

  Without the non-degeneracy requirement, the previously known lower
  bound on the number of prime closed orbits is only
  $\lceil n/2 \rceil+1$, i.e., roughly $n/2$, and the full form of
  dynamical convexity is used in the proofs; see \cite{DL, GG:LS} and
  also \cite{LZ, Wa2} for earlier results. As we have already
  mentioned, in Theorem \ref{thm:mult} we improve this lower bound to
  $n$, i.e., by roughly a factor of two. (When $n=3$ and $4$, the
  lower bounds $3$ and, respectively, $4$ have been established in
  \cite{Wa1, WHL}.)

  Imposing the symmetry condition somewhat simplifies the question. In
  this case, the existence of $n$ prime closed orbits on convex
  hypersurfaces was proved in \cite{LLZ} without the non-degeneracy
  requirement and then in \cite{GM} under a less restrictive dynamical
  convexity type condition; see also \cite{ALM} for other relevant
  results. Our second main result, Theorem \ref{thm:HZ}, establishes a
  particular case of Part \ref{C2} and asserts that a Reeb flow with
  finitely many prime closed orbits must have exactly $n$ such orbits,
  provided that the flow is non-degenerate and $M$ is centrally
  symmetric. The lower bound $n$ is known to hold due to either
  non-degeneracy or, independently, symmetry, and the new point is the
  upper bound.

  To put this theorem and Part \ref{C2} in perspective, it is
  illuminating to first discuss similar results and questions in the
  setting of Hamiltonian diffeomorphisms and Hamiltonian Floer theory.

  Recall that by the celebrated theorem of Franks a Hamiltonian
  diffeomorphism of $S^2$ with more than two fixed points must have
  infinitely many periodic orbits; see \cite{Fr1, Fr2, LeC}. A
  generalization of Franks' theorem to higher dimensions was
  conjectured by Hofer and Zehnder in \cite[p.\ 263]{HZ}, asserting
  that a Hamiltonian diffeomorphism of $\CP^n$ with more than $n+1$
  fixed points must have infinitely many periodic orbits.  (Hence, the
  name ``HZ-conjecture''.) Here, by Arnold's conjecture, $n+1$ is the
  minimal possible number of fixed points, \cite{Fl2, Fo, FW}. For
  instance, $n+1=2$ for $\CP^1=S^2$. The HZ-conjecture was proved in
  \cite{Sh:HZ} under a minor additional non-degeneracy type condition
  but for a broader class of manifolds than complex projective spaces;
  see also, e.g., \cite{Al1, Al2, Al3, ABF, AL, BX, Ba1, Ba2, Ba3,
    CGG:HZ, GG:hyperbolic, GG:nc, Gu:nc, Su:nc} for other relevant
  results, conjectures and more context.

  Part \ref{C2} is a contact analogue of the HZ-conjecture.  To the
  best of our knowledge, this conjecture was originally stated in
  \cite{Gu:pr}. In that paper the first results along the lines of an
  upper bound on the number of prime closed orbits were proved for
  lacunary Reeb flows on $S^{2n-1}$. (See also \cite{AM:recent} for a
  generalization of this upper bound to other prequantization
  bundles.) A more recent result closely related to the contact
  HZ-conjecture is that a non-degenerate dynamically convex Reeb flow
  with a hyperbolic closed orbit must have infinitely many prime
  closed orbits, \cite{CGGM:Inv}. (We emphasize that that result is
  logically independent from what is proved in this paper: neither of
  the results imply the other. The approaches to the question are also
  substantially different.) Other than these works, we are not aware
  of any prior results supporting the conjecture when $2n-1\geq
  5$. Theorem \ref{thm:HZ} proves the first sufficiently general case
  of the conjecture, albeit under additional conditions.

  From a slightly different perspective, the $n$-or-$\infty$
  conjecture is a question about Reeb pseudo-rotations. In the context
  of symplectic dynamics, the two competing definitions of a
  pseudo-rotation are that this is a Hamiltonian system (a Hamiltonian
  diffeomorphism or a Reeb flow) with finitely many periodic orbits
  or, alternatively, with the minimal possible number of periodic
  orbits (perhaps, counted with multiplicity); see \cite{GG:PR, Sh:HZ,
    CGGM:Inv}. Hypothetically, the two definitions are equivalent and
  Theorem \ref{thm:HZ} establishes this fact for a certain class of
  Reeb flows, just as \cite{Sh:HZ} did for Hamiltonian diffeomorphisms
  of $\CP^n$ and some other manifolds; see also
    \cite{AL,BX,CGG:HZ}. (Interestingly, this is no longer true for
  symplectomorphisms, \cite{ABF}.) Overall, perhaps expectedly so,
  there is a considerable similarity between the properties of
  Hamiltonian and Reeb pseudo-rotations; see, e.g., \cite{AK, LRS}
  vs.\ \cite{Ka} or \cite{GG:PR} vs.\ \cite{CGGM:Inv}.
 
  Reeb and Hamiltonian pseudo-rotations are of interest because they
  may have quite complicated dynamics beyond periodic orbits; for
  instance, they can be ergodic, \cite{AK, Ka, LRS}. However, in all
  known examples the periodic orbit data of a pseudo-rotation (the
  linearized return maps and the actions) is the same as a suitably
  chosen $S^1$-action. (This is consistent with the conjecture that
  all pseudo-rotations are obtained by the Anosov--Katok conjugation
  method from \cite{AK, Ka}.) In Corollary \ref{cor:ellipsoids} we
  show that, under minor non-resonance conditions, the fixed point
  data of a dynamically convex Reeb pseudo-rotation matches that of an
  ellipsoid; cf.\ \cite {GG:RvsPR}.  It is also worth pointing out
  that all known examples of Reeb pseudo-rotations on $S^{2n-1}$ are
  dynamically convex and non-degenerate and actually have the same
  fixed point data as irrational ellipsoids.

  Finally, the case of dimension $2n-1=3$ deserves a separate
  discussion. According to the so called 2-or-infinity conjecture, a
  Reeb flow on a closed 3-manifold has either exactly two or
  infinitely many prime closed orbits. This conjecture is almost
  proved in general, and lens spaces are the only manifolds admitting
  Reeb pseudo-rotations; see \cite{CGH, CGHP, CGHHL,
    CDR}. Furthermore, the conjecture has been completely proved for
  $S^3$. (See also \cite{GHHM} for a different proof of \ref{C1}
  for $S^3$.)  Hence, both \ref{C1} and \ref{C2} parts of our
  $n$-or-$\infty$ conjecture hold in this case. However, one cannot
  expect this clear-cut dichotomy to hold in higher dimensions, even
  for tight or fillable contact manifolds. Indeed, while it is
  reasonable to expect that a majority of such contact manifolds do
  not admit Reeb pseudo-rotations in all dimensions just as in
  dimension three, it is not clear how to rigorously state this
  conjecture and a description of manifolds that do seems beyond our
  reach. Moreover, the lower bound on the number of orbits clearly
  depends on the manifold when $n>2$.

\begin{Remark}[Smoothness]
  Returning to the statement of Ekeland's conjecture in \cite[p.\
  198]{Ek}, it is worth pointing out that there the hypersurface $M$
  is assumed to be only $C^1$-smooth. To the best of our
  understanding, our methods require $M$ to be at least $C^2$. We
  would expect Part \ref{C1} of the $n$-or-$\infty$ conjecture to be
  still true in dimension 3 in the $C^1$-case and accessible by
  symplectic topological methods, and possibly false in
  higher-dimensions; cf.\ \cite{BHS, LRSV}.
\end{Remark}

The key to the proofs of the two main theorems is Theorem \ref{thm:SH}
asserting that for a dynamically convex Reeb pseudo-rotation the
(non-equivariant) filtered symplectic homology is one-dimensional for
every action threshold and any ground field. The proof of that theorem
relies on three main ingredients and the novelty is how they are
  put together. The first ingredient, Theorem \ref{thm:IR-A}, is a key
  refinement of the index recurrence theorem, \cite[Thm.\
  5.2]{GG:LS}. (The latter is closely related to and can be derived
  from the common jump theorems from \cite{DLW, Lo, LZ}.)  This is a
combinatorial result about, as the name suggests, recurrent behavior
of the Conley--Zehnder index type invariants under iterations. This
theorem and its consequences are repeatedly used throughout the
paper. The other two ingredients are not new: one is an \emph{a
  priori} upper bound on the boundary depth of the symplectic homology
persistence module (Theorem \ref{thm:vanishing}, originally from
\cite{GS}) and the other is the Smith inequality for symplectic
homology (Theorem \ref{thm:Smith}, originally from \cite{Se,
  ShZ}). Here, due to the nature of the Smith inequality which is used
to obtain a lower bound on the dimension, we first need to work over
fields of positive characteristic and then extend the result to all
fields by the universal coefficient theorem. Once Theorem \ref{thm:SH}
is proved, we can set the ground field to be $\Q$ or any field of zero
characteristic.

We use Theorem \ref{thm:SH} to show that every homologically visible
(over $\Q$) closed orbit has one-dimensional equivariant local
symplectic homology. This allows to assign such a closed orbit a
degree and treat the orbit as if it were non-degenerate. We show that
the degrees of such orbits form the increasing sequence
$n+1, n+3, \ldots$ and the degrees of bars are $n, n+2, \ldots$. All
$\Q$-visible closed orbits have distinct period (action) and the
same ratio of the action and the mean index. (See Corollary
\ref{cor:SH} and Theorem \ref{thm:Orbits}.)  In other words, the
(non-equivariant) symplectic homology persistence module and the orbit
degrees behave as they would for an irrational ellipsoid.

Theorem \ref{thm:HZ} is a direct consequence of Corollary \ref{cor:SH}
and standard results about periodic orbits on centrally symmetric
hypersurfaces. The proof of Theorem \ref{thm:mult} uses again the
index recurrence theorem, Theorem \ref{thm:Orbits} and also several
facts about the behavior of the local homology under iterations. The
last two components are crucially new ingredients of the proof
compared to previous multiplicity results.

The paper is organized as follows. In Section \ref{sec:results} we
precisely state our main results. Sections \ref{sec:SH} and
\ref{sec:IR} cover preliminary material and auxiliary results, some
standard and some not, in a reasonably general setting. In Sections
\ref{sec:Liouville} and \ref{sec:persistence} we spell out our
conventions and notation and then review the notion of a persistence
module in the context suitable for our purposes. Equivariant and
non-equivariant symplectic homology, filtered and local, and their
properties needed for the proofs are discussed in Sections
\ref{sec:SH-pers}, \ref{sec:LSH} and \ref{sec:beg-end}. Most of the
facts we recall there are well-known but there are also some new
results. The main theme of Section \ref{sec:IR} is the Conley--Zehnder
index. There we discuss several Conley--Zehnder index type
invariants and their properties, briefly touch upon dynamical
convexity, and then state a new version of the index recurrence
theorem and some of its consequences. This is the combinatorial
backbone of the proof. In Section \ref{sec:pfs} we prove the main
results of the paper. Finally, the index recurrence theorem is proved
in Section \ref{sec:IRT-pf}.

\subsection{Main results}
\label{sec:results}

Before stating the results, we need to briefly review some
terminology, referring the reader to Sections \ref{sec:SH} and
\ref{sec:CZ+DC} for more details.

Let $M$ be the boundary of a star-shaped domain $W\subset
\R^{2n}$. Throughout the paper, we assume $M$ to be smooth. We say that
a closed Reeb orbit $y$ in $M$ is the $k$th \emph{iterate} of a closed
orbit $x$ and write $y=x^k$ if $y(t)=x(kt)$ up to the choice of
initial conditions on $x$ and $y$. (In what follows we will always
identify the closed orbits $t\mapsto x(t)$ and $t\mapsto x(t+\theta)$
for all $\theta\in\R$.) A closed orbit $x$ is \emph{prime} if it is
not an iterate of any other closed orbit.

A closed Reeb orbit $x$ is said to be \emph{non-degenerate} if $1$ is
not an eigenvalue of its Poincar\'e return map. The Reeb flow on $M$
is non-degenerate if all closed orbits, prime and iterated, are
non-degenerate.  The \emph{parity} of a non-degenerate orbit is the
parity of its Conley--Zehnder index $\mu$; see Section
\ref{sec:CZ-orbits}.

When $x$ is non-degenerate, we call it \emph{non-alternating} if the
number, counted with algebraic multiplicity, of real negative
eigenvalues in $(-1,0)$ of its Poincar\'e return map is
even. Otherwise, we call $x$ \emph{alternating}. A non-degenerate even
iterate of an alternating closed orbit is called \emph{bad}; all other
non-degenerate closed orbits are called \emph{good}. When $M$ is
centrally symmetric, we call $x$ \emph{symmetric} if $-x(t)=x(t+T/2)$,
where $T$ is the period of $x$, i.e., $x$ and $-x$ are the same orbit
up to the choice of initial condition.

Fix a ground field $\F$. (The choice of $\F$ is essential in several
instances in this paper.)  A closed Reeb orbit $x$ is said to be
\emph{$\F$-visible} if its \emph{local symplectic homology}
$\SH(x;\F)$ over $\F$ is non-zero; see Section \ref{sec:LSH}. For
instance, a non-degenerate closed orbit $x$ is $\Q$-visible if and
only if $x$ is good. We denote the set of $\F$-visible closed orbits
by $\PP(\F)$ or simply $\PP$ when $\F$ is clear from the context.

Finally, recall that the Reeb flow on
  $M^{2n-1}$ is said to be \emph{dynamically convex} if
$\mu_-(x)\geq n+1$ for all closed Reeb orbits $x$, where $\mu_-$ is
the lower semi-continuous extension of the Conley--Zehnder index
$\mu$; see Section \ref{sec:CZ+DC}.

The following four theorems and their corollaries are the key
results of the paper.

\begin{TheoremX}[Multiplicity]
  \label{thm:mult}
  Assume that the Reeb flow on the boundary $M^{2n-1}\subset \R^{2n}$
  of a star-shaped domain is dynamically convex and has finitely many
  prime closed orbits, i.e., the flow is a Reeb pseudo-rotation.

\begin{itemize}
\item[\reflb{Mult}{\rm{(M1)}}] Then the flow has at least $n$ prime
  closed orbits. 

\item[\reflb{Mult-nd}{\rm{(M2)}}] If, in addition, the flow is
  non-degenerate, it has exactly $n$ non-alternating prime closed
  orbits, and all orbits have parity $n+1$.

\end{itemize}
As a consequence, the Reeb flow on the boundary of a star-shaped
domain has at least $n$ prime closed orbits whenever the flow is
dynamically convex.
\end{TheoremX}

This theorem is proved in Section \ref{sec:pf-main1}, where in fact we
establish a more precise result (Theorem \ref{thm:mult-refined})
identifying $n$ prime orbits meeting a certain visibility
criterion. The key new point of Theorem \ref{thm:mult} is, of course,
\ref{Mult}, establishing Ekeland's conjecture and improving previously
known multiplicity lower bounds by roughly a factor of two. As we have
already mentioned, the existence of at least $n$ non-alternating
orbits for non-degenerate dynamically convex Reeb flows is well-known;
see \cite{GK,LZ}. The new part in \ref{Mult-nd}, essential for the
proof of Theorem \ref{thm:HZ}, is that in this case the number of
non-alternating closed prime orbits is exactly $n$ and the remaining
prime orbits, if they exist, must all be alternating.

Our second main result establishes a particular case of the contact
HZ-conjecture.

\begin{TheoremX}[HZ-conjecture]
  \label{thm:HZ}
  Assume that the Reeb flow on the boundary $M^{2n-1}\subset \R^{2n}$
  of a centrally symmetric star-shaped domain has finitely many prime
  closed orbits and is dynamically convex and non-degenerate.  Then
  the flow has exactly $n$ prime closed orbits. These orbits are
  symmetric and non-alternating.
\end{TheoremX}

This theorem, also proved in Section \ref{sec:pf-main1}, is in fact an
easy consequence of \ref{Mult-nd} and Corollary \ref{cor:SH}
below. Recall in this connection that \cite[Thm.\ A] {CGGM:Inv}
asserts, in particular, that a dynamically convex Reeb flow on $M$
with a hyperbolic closed orbit has infinitely many prime closed
orbits. Interestingly, that theorem does not follow from Theorem
\ref{thm:mult} even in the non-degenerate case. Nor is it a
consequence of Theorem \ref{thm:HZ} when in addition $M$ is symmetric.

\begin{Remark}
  In all results of this paper, the requirement that the flow has only
  finitely many prime closed orbits can be replaced by a
  hypothetically less restrictive condition that the number of prime
  \emph{eventually homologically visible} orbits is finite. Here we
  call $x$ eventually homologically visible if $\SH\big(x^k;\F)\neq 0$
  for some $k$ and $\F$. Furthermore, when $\charr\F>0$ in Theorem
  \ref{thm:SH} below it suffices to assume that the number of prime
  \emph{eventually $\F$-visible} orbits is finite, where $x$ is
  eventually $\F$-visible if $\SH\big(x^k;\F)\neq 0$ for some $k$ and
  that particular ground field $\F$.
\end{Remark}

Turning to our more technical results, fix a background field
$\F$. When $\F$ has positive characteristic, $\charr\F>0$, the reader
may take $\F=\Z_p$ and when $\charr\F=0$ it would be sufficient,
without loss of generality, to set $\F=\Q$. Denote by $\SH^t(W;\F)$,
where $t\in\R$ is the action or equivalently period threshold, the
filtered symplectic homology of $W$ with coefficients in $\F$. For
each $t$, this is a $\Z$-graded vector space over $\F$. Throughout the
paper it is convenient to treat the family of vector spaces
$\SH^t(W;\F)$ as a graded persistence module, which we denote by
$\SH(W;\F)$; see Sections \ref{sec:persistence} and \ref{sec:SH-pers}.
The next result, proved in Section \ref{sec:pf-main3}, is central to
the proofs of Theorems \ref{thm:mult} and \ref{thm:HZ}.

\begin{TheoremX}
  \label{thm:SH}
  Fix a ground field $\F$. Assume that the Reeb flow on the boundary
  $M^{2n-1}$ of a star-shaped domain $W \subset \R^{2n}$ is a
  dynamically convex pseudo-rotation, i.e., it has finitely many prime
  closed orbits and is dynamically convex. Then, for every $t>0$,
  \begin{equation}
    \label{eq:dim-key}
    \dim \SH^t(W;\F)=1.
  \end{equation}
In other words, every $t>0$ belongs to exactly one bar in the barcode
of the persistence module $\SH(W;\F)$. 
\end{TheoremX}  

The \emph{support} of a $\Z$-graded vector space is the set of degrees
in which the vector space is non-zero. Denote by $\COH(x;\F)$ the
\emph{equivariant local symplectic homology} of $x$ with coefficients
in $\F$; see Section \ref{sec:LSH}.  An easy consequence of
Theorem \ref{thm:SH} is the following corollary also proved in Section
\ref{sec:pf-main3}.

\begin{CorollaryX}
  \label{cor:SH}
Let $W$ be as in Theorem \ref{thm:SH}. Then
\begin{itemize}
\item all $\F$-visible closed orbits $x$ have distinct periods
  and $\dim_\F\SH(x;\F)=2$;
\item the homology $\SH(x;\Q)$ is supported in two
  consecutive degrees for $x\in\PP(\Q)$, i.e.,
  \begin{equation}
    \label{eq:deg(x)}
  \supp\SH(x;\Q)=:\{h,h+1\}
\end{equation}
for some $h\in\Z$, and $\COH_i(x;\Q)=\Q$ when $i=h$ and zero
otherwise.
\end{itemize}
\end{CorollaryX}

From now on, we can restrict our attention to ground fields of zero
characteristic, e.g., $\F=\Q$.  In what follows, we call $h$ from
\eqref{eq:deg(x)} \emph{the degree $\deg(x)$ of $x$}. In other words,
$h=\deg(x)\in \Z$ is the only integer such that $\COH_h(x;\Q)\neq 0$
(and hence is $\Q$.)  For instance, $\deg(x)=\mu(x)$ when $x$ is
non-degenerate and good. We emphasize that this notion is specific to
the setting of Theorem \ref{thm:SH} and Corollary \ref{cor:SH}, i.e.,
$\deg(x)$ is defined only when the flow is dynamically convex and has
finitely many prime closed Reeb orbits and $x$ is $\Q$-visible.
Recall that $\PP=\PP(\Q)$ stands for the set of all, not necessarily
prime, $\Q$-visible closed orbits. Denote by $\CA(x)$ the action (aka
the period) of $x$ and by $\hmu(x)$ the mean index of $x$; see
Sections \ref{sec:Liouville} and \ref{sec:CZ+DC}.

Our next theorem, proved in Section \ref{sec:pf-main2}, provides yet
more detailed information on closed Reeb orbits in $M$ and the
persistence module $\SH(W;\Q)$ and also serves as a bridge between
Theorems \ref{thm:mult} and \ref{thm:SH}.

\begin{TheoremX}
  \label{thm:Orbits}
  Assume that $W$ is as in Theorem \ref{thm:SH}.  Then $\PP$ and the
  barcode $\CB$ of the persistence module $\SH(W;\Q)$ have the
  following properties:

  \begin{itemize}

  \item[\reflb{O-index}{\rm{(O1)}}] The map $\deg\colon \PP\to \Z$ is
    a bijection onto $n-1+2\N=\{n+1,n+3,\ldots\}$ and orbits with
    larger degree have larger period. Likewise, $\deg\colon \CB\to\Z$
    is a bijection onto $n-2+2\N=\{n,n+2,\ldots\}$ and the bars with
    larger degree have larger period (at, say, the beginning point).

  \item[\reflb{O-clusters}{\rm{(O2)}}]
      All closed orbits $x\in \PP$ have the same ratio
    $\CA(x)/\hmu(x)$.
   
\end{itemize}
\end{TheoremX}  
 
This theorem asserts, in particular, that the graded barcode of
$\SH(W;\Q)$ looks exactly like that of an irrational ellipsoid: the
bars form an ascending infinite staircase with one riser of height one
(the degree) for each closed orbit and starting with a tread of degree
$n$ beginning at the interior of $W$. We will show in the next section
that the similarity with ellipsoids goes farther.  Note also that as a
consequence of \ref{O-index} the flow is automatically lacunary in the
terminology from \cite{AM:recent}. It is also worth noting that the
ratios $\CA(x)/\hmu(x)$ playing an important role in these questions
were originally, to the best of our knowledge, considered in
\cite{Ek,EH}.

\begin{Remark} We conjecture that \eqref{eq:dim-key} is equivalent to
  that the Reeb flow is dynamically convex and has only a finite
  number of prime orbits at least in the non-degenerate case. It not
  hard to see that then \eqref{eq:dim-key} implies dynamical
  convexity. However, proving finiteness appears to be much more
  difficult. Without non-degeneracy one would also have to account for
  invisible orbits.
\end{Remark} 

\begin{Remark}
  A consequence of Theorem \ref{thm:Orbits} and the index recurrence
  theorem (Theorem \ref{thm:IR-A}) is that a dynamically convex Reeb
  flow has exactly $n$ prime closed orbits whenever all closed orbits,
  prime and iterated, are $\Q$-visible; see Corollary
  \ref{cor:mult+HZ}.  Corollary \ref{cor:SH} and Theorem
  \ref{thm:Orbits} are revisited and further refined in
  \cite{CGG:Reeb-HZ2}.
\end{Remark}

\subsection{Application: ellipsoid comparison}
\label{sec:ellipsoids}
The main results of the paper have an interesting application in the
context of the question of the extent to which symplectic topology of
$W$ or of its interior is tied to the dynamics on $M=\p W$. Assume
that $W\subset \R^{2n}$ is star-shaped and the flow on $M=\p W$ is
non-degenerate, dynamically convex and has finitely many prime closed
orbits. We further require that the number of such orbits is exactly
$n$ and these orbits are non-alternating. For instance, this is the
case when $W$ is symmetric as in Theorem \ref{thm:HZ}. In this
section, in the spirit of \cite{GG:RvsPR}, we show that then, under a
very minor non-resonance condition, the linearized flows along the
prime closed Reeb orbits on $M=\p W$ are completely determined up to
conjugation by the persistence module $\SH(W;\F)$ and, moreover, these
linearized flows are exactly the same as those for an ellipsoid with
the same action spectrum as $M$.

To be more specific, denote by
$\vDelta:=\{\Delta_1<\cdots< \Delta_n\}\in\R^n_+$ the ordered set of
the mean indices of these orbits. (By Theorem \ref{thm:Orbits}, all
closed orbits have distinct mean indices. In what follows, we will
always order prime orbits so that their mean indices form an
increasing sequence.)  Consider the solid ellipsoid
$W_{\vDelta} \subset \C^n=\R^{2n}$ given by
$$
\sum_j \frac{|z_j|^2}{\Delta_j}\leq 1
$$
and its boundary $M_{\vDelta}=\p W_{\vDelta}$. The prime closed orbits
on $M_{\vDelta}$ also have mean indices $\Delta_1<\cdots< \Delta_n$
and up to scaling $M$ and $M_{\vDelta}$ have the same action
spectra. Indeed, let us scale $M$ so that $\hmu(x)=\CA(x)$ for all
closed orbits $x$ on $M$. This is possible by \ref{O-clusters}. Then
for both $M$ and $M_{\vDelta}$ the action spectrum is simply the union
of sequences $\Delta_j\N$.

\begin{CorollaryX}
  \label{cor:ellipsoids}
  Let $W\subset \R^{2n}$ be a star-shaped domain. Assume that the flow
  on $M=\p W$ is non-degenerate, dynamically convex, has exactly $n$
  prime closed orbits. 
  In
  addition, let us require the following non-resonance condition: for
  every $j$ all ratios $\pm\Delta_j/\Delta_i$, $i\neq j$, are distinct
  modulo $\Z$.  Then the linearized flows along the corresponding
  closed orbits in $M$ and $M_{\vDelta}$ are equal up to conjugation
  as elements of $\tSp\big(2(n-1)\big)$.
\end{CorollaryX}

This corollary, proved in Section \ref{sec:ellipsoids-pf}, asserts
that within the admittedly very narrow class of hypersurfaces $M$
meeting the conditions of the corollary, the action spectrum
completely determines the linearized flows along prime closed orbits
and, of course, the actions. In particular, all closed Reeb orbits in
$M$ are elliptic. At the same time, dynamics of the Reeb flow on $M$,
beyond closed orbits, can be radically different from that on
$M_{\vDelta}$ and quite non-trivial. For instance, the Reeb flow on
$M$ can be ergodic even in the setting of Corollary
\ref{cor:ellipsoids}; see \cite{Ka}. Clearly, then the closed domains
$W$ and $W_{\vDelta}$ are not symplectomorphic and such dynamics
features are not detected by their symplectic homology persistence
modules, which are isomorphic. We do not know if the interiors of $W$
and $W_{\vDelta}$ are symplectomorphic or not. (The answer may depend
on Liouville/Diophantine properties of $\vDelta$.)

\medskip\noindent\subsection*{Acknowledgments} The authors are
grateful to Leonardo Macarini, Marco Mazzucchelli and Felix Schlenk
for useful remarks and suggestions. The authors would also like to
thank the referees for thoroughly scrutinizing the paper.

\section{Symplectic homology}
\label{sec:SH}

\subsection{Conventions and notation}
\label{sec:Liouville}
Let $\alpha$ be the contact form on the boundary $M=\p W$ of a
Liouville domain $W^{2n\geq 4}$. We will also use the same notation
$\alpha$ for a primitive of the symplectic form $\omega$ on $W$. The
grading of Floer or symplectic homology is essential for our purposes
and for the sake of simplicity we require that $c_1(TW)=0$. (This
condition can be relaxed.) As usual, denote by $\WW$ the symplectic
completion of $W$, i.e.,
$$
\WW=W\cup_M M\times [1,\,\infty)
$$
with the symplectic form $\omega=d\alpha$ extended from $W$ to
$M\times [1,\infty)$ as
$$
\omega := d(r\alpha),
$$
where $r$ is the coordinate on $[1,\,\infty)$. Sometimes it is
convenient to extend the function $r$ to a collar of $M=\p W$ in
$W$. Thus we can think of $\WW$ as the union of $W$ and
$M\times [1-\eps,\,\infty)$ for small $\eps>0$ with
$M\times [1-\eps,\, 1]$ lying in $W$ and the symplectic form given by
the same formula.

The Reeb vector field $R$ of $\alpha$ on $M$ is determined by the
conditions $i_Rd\alpha=0$ and $\alpha(R)=1$, and the flow of $R$ is
called the Reeb flow. The period or action of a closed orbit $x$ of
the flow, denoted by $\CA(x)$, is the integral of $\alpha$ over
$x$. The \emph{action spectrum} $\CS(\alpha)$ is the collection of
periods of all closed Reeb orbits.

\begin{Example}
The main example we are interested in is where $W$ is a star-shaped
domain in $\R^{2n}$ with smooth boundary. Then we have a natural
identification $\WW=\R^{2n}$ as symplectic manifolds.
\end{Example}

As we have already mentioned in Section \ref{sec:results}, a closed
Reeb orbit $x$ is \emph{non-degenerate} if all eigenvalues of the
Poincar\'e return map are different from 1, and $x$ is \emph{strongly
  non-degenerate} if all iterates $x^k$, $k\in \N$, are
non-degenerate.  The Reeb flow or the form $\alpha$ is said to be
non-degenerate if all closed orbits are non-degenerate. For a closed
orbit $x$, we denote by $x^k$, $k\in\N$, its $k$th iterate. This is
simply the orbit $x$ traversed $k$ times. Clearly,
$\CA\big( x^k\big)=k\CA(x)$. An orbit is said to be \emph{prime} if it
is not an iterate of any other orbit. Otherwise, we will call an orbit
\emph{iterated}.

When $x$ is non-degenerate, we denote by $\mu(x)$ the Conley--Zehnder
index of $x$. Without the non-degeneracy condition, we have the lower
and upper semicontinuous extensions $\mu_\pm(x)$ of the
Conley--Zehnder index and the mean index $\hmu(x)$ associated to $x$.
Referring the reader to Section \ref{sec:CZ-LA} for the definitions
and a discussion, we only note here that these indices depend on an
additional piece of data (not unique in general) which determines the
choice of a symplectic trivialization of the contact structure along
$x$; see Section \ref{sec:CZ-orbits}.

Recall also that a non-degenerate closed orbit $x$ is said to be
\emph{non-alternating} if the number (counted with algebraic
multiplicity) of eigenvalues in $(-1, 0)$ of the return map is even.
Otherwise, we call $x$ \emph{alternating}. When $x$ is strongly
non-degenerate, these conditions are equivalent to that all indices
$\mu\big(x^k\big)$, $k\in\N$, have the same parity or that the parity
alternates. (A word is due on terminology: Non-alternating orbits
  are sometimes referred to as negative hyperbolic, but we find this
  term misleading when $\dim M>3$. The frequently used terms ``even/odd''
  are more intuitive; however they clash with the parity of an
  orbit.) While formally speaking these definitions require no
  conditions on $x$, we will apply them only when $x$ is prime and
  strongly non-degenerate. The notion of (non-)alternating closed
orbit is not particularly useful when $x^k$ is allowed to be
degenerate; see Example \ref{exam:iterates2}. The \emph{parity} of a
closed non-degenerate orbit $x$ is simply the parity of $\mu(x)$. A
non-degenerate closed orbit $x$ is said to be \emph{bad} if it is an
even iterate of an alternating prime orbit, and \emph{good} otherwise.

Throughout the paper we will impose the following additional condition
on the Reeb flow of $\alpha$ on $M$:
\begin{itemize}
\item[\reflb{BA}{\rm{(BA)}}] {\bf Background assumption:} \emph{All
    closed orbits are isolated, i.e., for every $t\in\R$ there are only
    finitely many closed orbits with period less than $t$.}
\end{itemize}
This condition is obviously satisfied in two cases we are interested
in here: when $\alpha$ is non-degenerate and when the Reeb flow has
only finitely many prime closed orbits. The role of this condition is
rather expository than conceptual, and most of the discussion in the
next two sections carries over to the general case with only technical
changes. For instance, the symplectic homology is still a persistence
module (see Section \ref{sec:persistence}) when $\CS$ is only closed
and nowhere dense and $V^s$ might be infinite-dimensional, once
\ref{PM2} in the definition of a persistence module is suitably
modified, \cite {CGGM:Entropy}.

We implicitly adopt the conventions from \cite{CGGM:Inv, CGGM:Entropy,
  GG:LS} on the almost complex structures, Hamiltonians and the Floer
equation. For the sake of brevity, since these conventions are never
explicitly used in the paper, we leave them out. Here we only mention
that with these conventions, the Conley--Zehnder index of a
non-degenerate minimum with small Hessian in dimension $2n$ is $n$.

The characteristic of the ground field $\F$ is essential in some parts
of the paper. When the characteristic is zero, we can and usually will
take $\F=\Q$. We use the notation $\F_p$ for the positive
characteristic field $\Z/p\Z$, where $p$ is prime, or any other field
of characteristic $p$.

\subsection{Graded persistence modules}
\label{sec:persistence}
While one can easily avoid using persistence modules in the proofs of
our main results, the language comes handy. In this section we briefly
recall relevant definitions in the form and generality suitable for
our purposes. We refer the reader to \cite{PRSZ} for a general
introduction to persistence modules, their applications in geometry
and analysis and further references.

Fix a field $\F$, which we will suppress in the notation. A
\emph{(graded) persistence module} $(V,\pi)$ is a family of
$\Z$-graded vector spaces $V^s$ over $\F$ parametrized by $s\in \R$
together with a functorial family $\pi$ of structure maps. These are
$\Z$-graded linear maps $\pi_{st}\colon V^s\to V^t$, where $s\leq t$,
and functoriality is understood as that $\pi_{sr}=\pi_{tr}\pi_{st}$
whenever $s\leq t\leq r$ and $\pi_{ss}=\id$. In what follows, we often
suppress $\pi$ in the notation and simply refer to $(V,\pi)$ as $V$.
In such a general form the concept is not particularly useful and one
usually imposes additional conditions on $(V,\pi)$. These conditions
vary depending on the context. Below we spell out the setting most
suitable for our purposes.

Namely, we require that there is a closed, discrete, bounded from
below subset $\CS\subset \R$, which we call the \emph{spectrum} of
$V$, and the following four conditions are met:
\begin{itemize}

\item[\reflb{PM1}{\rm{(PM1)}}] The persistence module $V$ is
  \emph{locally constant} outside $\CS$, i.e., $\pi_{st}$ is an
  isomorphism whenever $s\leq t$ are in the same connected component
  of $\R\setminus \CS$.

\item[\reflb{PM2}{\rm{(PM2)}}] In every degree $m$, the dimension of
  $V^s$ is finite for every $s$: $\dim V^s_m<\infty$.
 
\item[\reflb{PM3}{\rm{(PM3)}}] \emph{Left-semicontinuity}: For all
  $t\in\R$,
  \begin{equation}
    \label{eq:semi-cont}
    V^t=\varinjlim_{s<t} V^s.
  \end{equation}

\item[\reflb{PM4}{\rm{(PM4)}}] \emph{Lower bound}: $V^s=0$ when $s<s_0$
  for some $s_0\in\R$. (Throughout the paper we will assume that
  $s_0=0$.)

\end{itemize}

In the setting we are interested in, $V^s$ is naturally defined only
for $s\not\in\CS$; then the definition is extended to all $s\in\R$ by
\eqref{eq:semi-cont}. By \ref{PM4}, we can always assume that
$s_0\leq \inf \CS$, which essentially forces $\CS$ to be bounded from
below. We emphasize, however, that $\CS$ is not assumed to be bounded
from above and this is actually never the case in our setting.  It
will be convenient to set
$$
V^\infty=\varinjlim_{s\to\infty} V^s.
$$

A basic example motivating requirements \ref{PM1}--\ref{PM4} is that
of the sublevel homology of a smooth function.

\begin{Example}[Homology of sublevels] Let $M$ be a smooth manifold
  and $f\colon M\to \R$ be a proper smooth function bounded from below
  and having a discrete set of critical values. Set
  $V^s:=H_*\big(\{f<s\};\F\big)$ with the structure maps induced by
  inclusions. It is not hard to see that conditions
  \ref{PM1}--\ref{PM4} are met with $\CS$ being the set of critical
  values of $f$. 
\end{Example}

By definition, $\supp V^s$, the \emph{support} of $V^s$, is the set of
degrees $m$ such that $V_m^s\neq 0$.

Recall furthermore that a \emph{(graded) interval persistence module}
$\F_{(a,\,b]}$, where $-\infty<a<b\leq \infty$, is defined by setting
$$
V^s_m:=\begin{cases}
  \F & \text{ when } s\in (a,\,b],\\
  0 & \text{ when } s \not\in (a,\,b],
\end{cases}
$$
in one degree $m$ and $V^s_l=0$ in all other degrees $l\neq
m$. Furthermore, in degree $m$ we have $\pi_{st}=\id$ if
$a<s\leq t\leq b$ and $\pi_{st}=0$ otherwise. The interval modules are
exactly the simple persistence modules, i.e., the persistence modules
that are not a (non-trivial) direct sum of other persistence modules.

The normal form or structure theorem asserts that every persistence
module meeting the above conditions can be decomposed as a direct sum
of a countable collection (i.e., a countable multiset) of interval
persistence modules. Moreover, this decomposition is unique up to
re-ordering of the sum. We refer the reader to \cite{PRSZ} for a proof
of this theorem and further references, and also, e.g., to \cite{CZCG,
  ZC} for previous or related results.  The multiset $\CB(V)$ of
intervals entering this decomposition is referred to as the
\emph{barcode} of $V$ and the intervals occurring in $\CB(V)$ as
\emph{bars}. In addition, every bar is assigned a degree which is by
definition the degree of the component of $V^s$ it comes from.

Finally, for a graded persistence module, just as for graded vector
spaces, we define the grading shifted persistence module as
$$
(V[n])_m:=V_{m+n}.
$$
Thus, when $n\geq 0$, the grading of $V$ gets shifted by $n$ to the
left.

\subsection{Filtered symplectic homology}
\label{sec:SH-pers}
The first versions of symplectic homology were introduced in
\cite{CFH, Vi} and since then the theory has been further
developed. We refer the reader to, e.g., \cite{Abu, Bo, BO0, BO,
  BO:Gysin, Se:biased} for the definitions and basic constructions,
which we mainly omit or just informally sketch for the sake of
brevity, focusing on lesser known facts. Here, treating equivariant
and non-equivariant symplectic homology as persistence modules, we
closely follow \cite{CGGM:Entropy, GG:LS}.

Let the Liouville domain $W$ be as in Section \ref{sec:Liouville}.
Fixing a field $\F$, we denote the non-equivariant filtered symplectic
homology over $\F$ for the action interval $I\subset \R$ by
$\SH^I(W;\F)$. This is a vector space over $\F$, graded by the
Conley--Zehnder index. Furthermore, we set
$$
\SH^t(W;\F):=\SH^{(-\infty,t)}(W;\F).
$$
This family of vector spaces together with the natural ``inclusion''
maps form a graded persistence module in the sense of Section
\ref{sec:persistence} with $\CS=\CS(\alpha)\cup\{0\}$. (It is
essential at this point that $W$ meets the background requirement
\ref{BA}. Otherwise, we would need to modify the definition of a
persistence module; see \cite{CGGM:Entropy}.) We denote the resulting
persistence module by $\SH(W;\F)$. The global symplectic homology is
then $\SH^\infty(W;\F)$.

Somewhat informally, $\SH^t(W;\F)$ can be defined via the following
construction. As is pointed out in Section \ref{sec:persistence}, it
suffices to define $\SH^t(W;\F)$ when $t\not\in\CS$. Assume first that
the Reeb flow is non-degenerate. Let $H$ be an autonomous Hamiltonian
on $\WW$, which is a convex, strictly increasing function $h(r)$ of
$r$ on $M\times [1,\infty)$ and $0$ on $W$. The function $h$ is
required to be linear at infinity, i.e., of the form $tr-c$ with
$c>0$.  Furthermore, let $\tH$ be a small Morse perturbation of $H$
supported near $W$. Then $\tH$ is a Morse--Bott Hamiltonian in the
sense of \cite{Bo, BO0}, and $\SH^t(W;\F)$ is simply the Floer
homology of $H$ or equivalently $\tH$; cf.\ \cite[Thm.\ 3.5 and Cor.\
3.6]{CGGM:Entropy}. This is the homology of a complex with two
generators $\hx$ and $\cx$ of degrees, $\mu(x)+1$ and $\mu(x)$
respectively, for every closed Reeb of orbit $x$ with $\CA(x)<t$ and
one generator for every critical point of $\tH$; see, e.g., \cite{Bo,
  BO0, CGGM:Inv}. In fact, one can trim down this complex further and
make the number of generators of the latter type equal to
$\dim\H(W;\F)$. In any event, when $W$ is star-shaped we can have only
one such generator, which we denote by $[W]$.  When the Reeb flow is
degenerate, we first take a small non-degenerate perturbation of the
contact form and then apply this construction. For an interval
$I=[t',t)$ with endpoints outside $\CS$, $\SH^{I}(W;\alpha)$ is
  defined as the homology of the quotient complex.

As a consequence of this construction and \ref{BA}, for every $t$ the
support
$$
\supp \SH^t(W;\F):=\{ m\in\Z\mid \SH^t_m(W;\F)\neq 0\}
$$
is a finite set. Furthermore, when $\delta>0$ is sufficiently small,
e.g., $\delta<\inf\CS(\alpha)$,
\begin{equation}
\label{eq:SH-small-delta}
\SH^\delta(W;\F):=\H(W;\F)[-n];
\end{equation}
see, e.g., \cite{BO, Vi}.  It readily follows from the above
description of the symplectic homology as the Floer homology of $\tH$
that the Euler characteristic of the graded vector space $\SH^t(W;\F)$
is independent of $t>0$. To be more precise, for any field $\F$ and
every $t>0$, we have
\begin{equation}
  \label{eq:Euler}
  \sum_m(-1)^m\dim \SH^t_m(W;\F)=(-1)^n\chi(W).
\end{equation}
Less directly, one can also use \eqref{eq:chi0}, \eqref{eq:sum} and
\eqref{eq:long-exact} to prove \eqref{eq:Euler}.

The construction of symplectic homology carries over to $\Z$ or any
other coefficient ring (see, e.g., \cite{Abu}), and hence we have the
universal coefficient theorem. In particular, for a fixed $t$, in
every degree $m$ we have
\begin{equation}
  \label{eq:universal}
  \dim\SH^t_m(W;\Q)=\dim\SH^t_m(W;\F_p),
\end{equation}
when the characteristic $p$ is large enough (depending on
$t$). Indeed, the torsion part of the group $\SH^t_m(W;\Z)$ is finite
and \eqref{eq:universal} holds whenever $\SH^t_m(W;\Z)$ has no
elements of order $p$.

In this paper we will also use two less standard properties of
symplectic homology.  The first of these is \cite[Prop.\ 3.5]{GS},
which is essentially due to K. Irie, combined with \cite[Thm.\
3.8]{GS}.

\begin{Theorem}[Boundary depth upper bound, \cite{GS}]
  \label{thm:vanishing}
  Assume that $\SH^\infty(W;\F)=0$. Then all bars in the barcode of
  $\SH(W; \F)$ are bounded by some constant $\Cbar$ independent of the
  location of the bar. When $W$ is displaceable in $\WW$, the
  displacement energy can be taken as $\Cbar$, which is then
  independent of $\F$.
 \end{Theorem}

 \begin{Remark} Recall in this connection that $\SH^{\infty}(W;\F)$
   vanishes, for instance, whenever $W$ is displaceable in $\WW$. For
   example, $\SH^{\infty}(W)=0$ when $W$ is a star-shaped domain in
   $\WW=\R^{2n}$. In this case, vanishing of $\SH^{\infty}(W)$ is
   established in \cite{Vi}. The general case is proved in \cite{CFO}
   via Rabinowitz--Floer homology and direct proofs are given in
   \cite{Su} and~\cite{GS}; see \cite{CGG:Reeb-HZ2} for a further
     discussion.  Strictly speaking, Theorem \ref{thm:vanishing} is
   proved in \cite{GS} for $\F=\Q$, but the argument carries over
   word-for-word to any unital ring.
\end{Remark}

The second fact we will need is a variant of the Smith inequality or
Borel's localization for filtered Floer homology.

\begin{Theorem}[Smith Inequality,  \cite{Se, ShZ}]
  \label{thm:Smith}
  For every Liouville domain $W$ and any interval $I$ with end-points
  outside $\CS(\alpha)$,
\begin{equation}
  \label{eq:Smith0}
\dim\SH^{pI}(W;\F_p)\geq\dim\SH^I(W;\F_p)
\end{equation}
and, in particular,
\begin{equation}
  \label{eq:Smith01}
\dim\SH^{pt}(W;\F_p)\geq\dim
\SH^t(W;\F_p)
\end{equation}
for all $t>0$.
\end{Theorem}

These inequalities were originally proved for the Floer homology of
any Hamiltonian $H$ on $\WW$ of the form $H=tr-c$ outside a compact
set with $t\not\in\CS(\alpha)$, \cite{Se, ShZ}. The version stated
here readily follows from those results combined with the fact that
for any fixed interval $I$ with end-points outside $\CS(\alpha)$ and a
fixed prime $p$, we have $\SH^{I}(W;\F_p)=\HF^I(H;\F_p)$ for a
suitable choice of $H$; see \cite[Sect.\ 3]{CGGM:Inv}. For instance,
for $I=(-\infty,t)$ we can take any $H$ such that $H=tr-c$ at
infinity. Then, $\SH^{I}(W;\F_p)=\HF(H;\F_p)$ by \cite[Cor.\
3.7]{CGGM:Entropy}, which implies~\eqref{eq:Smith01}.

Filtered symplectic homology has an $S^1$-equivariant counterpart,
which we denote by $\COH^I(W;\F)$. (In fact, the construction carries
over to any ring and, in particular, $\Z$.) Again we refer the reader
to, e.g., \cite{BO:Gysin, Vi} for definitions; here we closely follow
\cite{GG:LS}. As above, let us set
$$
\COH^t(W;\F):=\COH^{(-\infty,t)}(W;\F),
$$
denoting the resulting persistence module by $\COH(W;\F)$, and also
$$
\COH^+(W;\F):=\COH^{[\delta,\infty)}(W;\F),
$$
where $\delta>0$ is sufficiently small. (Namely,
$\delta<\inf\CS(\alpha)$.) Then, when $W$ is a star-shaped domain,
\begin{equation}
  \label{eq:CH}
  \dim\COH_m^+(W;\F)=
  \begin{cases}
    1\textrm{ when } m=n+1+2i,\, i\geq 0,\\
    0\textrm{ otherwise.}
  \end{cases}
\end{equation}

When the Reeb flow is non-degenerate, for any interval
$I\subset (0,\infty)$ with end-points outside $\CS(\alpha)$, the space
$\COH^I(W;\F)$ is the homology of a certain complex generated by
closed Reeb orbits with action in $I$ and graded by the
Conley--Zehnder index; see Remark \ref{rmk:COH-complex} and
\cite[Sec.\ 2.5]{GG:LS}. (If $\alpha$ is degenerate and the end-points
of $I$ are outside $\CS(\alpha)$, we can first take a small
non-degenerate perturbation of $\alpha$ and then apply this
construction.)

\begin{Remark}
  \label{rmk:coeff}
  It is worth keeping in mind that this description does not carry
  over to a field of positive characteristic $p \ge 2$. The reason is
  that $x^p$ can contribute infinitely many generators to the complex
  over $\F_p$. For instance, if $x^p$ is non-degenerate and the only
  orbit with action in $I$ and $x$ is prime, $\COH^I(W;\F_p)$ is
  isomorphic to the homology of the infinite-dimensional lens space
  $S^{2\infty-1}/\Z_p$ over $\F_p$. Overall, the properties of
    the filtered homology $\COH^t(W;\F)$ depend significantly on
    $\charr\F$; see \cite{CGG:Reeb-HZ2}.
\end{Remark}

Finally, recall from \cite{BO:Gysin} that the equivariant and
non-equivariant symplectic homology of $W$ fit together into the Gysin
exact sequence
\[
\cdots \to \SH_m^I(W;\F)\to \COH^I_m(W;\F)
\stackrel{\DD}{\longrightarrow}  \COH^I_{m-2}(W;\F)\to
\SH^I_{m-1}(W;\F)\to 
\cdots .
\]
While the construction of symplectic homology in \cite{BO:Gysin} uses
$\Q$ as the coefficient field and a semi-infinite interval as $I$, it
carries over word-for-word to $\Z$, and hence any ring, and an
arbitrary interval $I$; see, e.g., \cite[Sect.\ 2]{GG:LS} for a
detailed discussion of $\DD$ and \cite{Abu} for a non-equivariant
construction over $\Z$.

\subsection{Local symplectic homology}
\label{sec:LSH}
Throughout this section all closed Reeb orbits, either prime and
iterated, are required to be isolated even when this is not explicitly
stated.  The construction of symplectic homology, equivariant or
non-equivariant, has a local counterpart. In this section we review
this construction and the properties of local symplectic homology and
then discuss its relation with filtered symplectic homology. While
some of the definitions and results are standard and well-known to the
experts, there are also several new observations.

\subsubsection{Basic definitions and facts}
\label{sec:local-def}
Let $x$ be an isolated closed Reeb orbit, not necessarily prime. Fix
an isolating neighborhood $U$ of $x(\R)$ in $M$. (By definition, $U$
is a neighborhood of $x$ containing no other closed orbit of period
close to $\CA(x)$.) Consider the product
$\UU=U\times (1-\eps,1+\eps)\subset \WW$ and let $H=h(r)$ be a
Hamiltonian such that $h'(1)=\CA(x)$ and $h''(1)>0$. Then $x$ gives
rise to an isolated set $S$ of 1-periodic orbits of $H$ with initial
conditions on $x(\R)$. By definition, the symplectic homology
$\SH(x;\F)$ is the local Floer homology $\HF(H;\F)$ of $H$ at $S$, and
$\COH(x;\F)$ is the $S^1$-equivariant local Floer homology.

Specifically, let $F$ be a non-degenerate 1-periodic in time
perturbation of $H$ and $J$ a 1-periodic in time almost complex
structure. When $F$ is sufficiently $C^2$-close to $H$, all 1-periodic
orbits of the flow of $F$ that $x$ splits into and Floer cylinders
between them are contained in $\UU$. (This is a variant of Gromov's
compactness, which follows, for instance, from the estimates in
\cite[Sec.\ 1.5]{Sa}.) For a generic choice of $J$ the regularity
conditions are satisfied. Hence we obtain a complex and $\SH(x;\F)$ is
the homology of this complex. The equivariant version is defined in a
similar fashion with appropriate modifications; see \cite[Sec.\
2.3]{GG:LS}.

This type of homology is the Floer-theoretic counterpart of critical
modules from \cite{GrMe1, GrMe2} and its definition goes back to
\cite{Fl1, Fl2}. The (non-equivariant) variant for isolated periodic
orbits of Hamiltonian diffeomorphisms is studied in \cite{Gi:CC,
  GG:gap}.  We refer the reader to \cite{Fe, Fe:thesis, GG:LS, McL}
for details in the Reeb case.

Similarly to the filtered homology, $\SH(x;\F)$ and $\COH(x;\Q)$ can
be more explicitly described as follows. Consider a small
non-degenerate perturbation of $\alpha$. Then $x$ splits into a finite
collection of non-degenerate orbits $x_i$. Essentially by definition,
$\SH(x;\F)$ is the homology of a certain complex with generators
$\cx_i$ and $\hx_i$ of degrees $\mu(x_i)$ and, respectively,
$\mu(x_i)+1$ and the differential coming from the Floer--Morse--Bott
construction as in \cite{Bo, BO0} and also, e.g.,
\cite{CGGM:Inv}. Likewise, $\COH(x;\Q)$ is the homology of a complex
over $\Q$ generated by the orbits $x_i$ and graded by the
Conley--Zehnder index; see \cite[Sec.\ 3]{GG:LS}. This is no longer
literally true when $\Q$ is replaced by a field of positive
characteristic; see Remark \ref{rmk:coeff}.

The absolute grading of the local symplectic homology is determined by
fixing a symplectic trivialization of the contact structure along
$x$. (Note that such a trivialization gives rise to trivializations
along all closed orbits $x_i$ that $x$ splits into.) Below we will
always assume that this trivialization agrees (up to homotopy) with
the trivialization used in the construction of the filtered symplectic
homology; see Section \ref{sec:CZ-orbits}.

\begin{Remark}
  \label{rmk:local}
  As is clear from these definitions, $\SH(x;\F)$ and $\COH(x;\F)$ are
  defined even for germs of contact structures near $x(\R)$ as long as
  $x$ is an isolated closed Reeb orbit for the germ. The absolute
  grading is again determined by the choice of a trivialization of the
  contact structure along $x$.
\end{Remark}

The support of $\SH(x;\F)$, denoted by $\supp\SH(x;\F)$, is the
collection of degrees $m$ such that $\SH_m(x;\F)\neq 0$. The support
$\supp\COH(x;\F)$ is defined in a similar fashion. It readily follows
from the definitions that
\begin{equation}
  \label{eq:support1}
  \supp \SH(x;\F)\subseteq [\mu_-(x),\mu_+(x)+1]
  \subseteq [\hmu(x)-n+1,\hmu(x)+n]
\end{equation}
and
\begin{equation}
  \label{eq:support2}
  \supp \COH(x;\Q)\subseteq [\mu_-(x),\mu_+(x)]
  \subseteq [\hmu(x)-n+1,\hmu(x)+n-1];
\end{equation}
see \cite[Prop.\ 2.20]{GG:LS}. Here $\hmu$ is the mean index 
and $\mu_\pm$ are the lower and upper semi-continuous extensions of
the Conley--Zehnder index; see Section \ref{sec:CZ-LA}.

For instance, when $x$ is non-degenerate, $\SH_m(x;\F)$ is $\F$ in
degrees $\mu(x)$ and $\mu(x)+1$ and zero in all other degrees. When in
addition $x$ is good, $\COH_m(x;\Q)=\Q$ for $m=\mu(x)$ and zero
otherwise. If $x$ is bad, $\COH(x;\Q)=0$.

Furthermore, recall from \cite{GG:LS} that equivariant and
non-equivariant local symplectic homology groups fit into the Gysin
exact sequence
$$
\cdots \to \SH_m(x;\F)\to \COH_m(x;\F) \stackrel{\DD}{\longrightarrow}
\COH_{m-2}(x;\F)\to \SH_{m-1}(x;\F)\to \cdots .
$$
This sequence is just a local variant of the Gysin sequence from \cite
{BO:Gysin}. Moreover, by \cite[Cor.\ 3.7]{GG:LS}, the shift operator
$\DD$ vanishes in local symplectic homology when $\F=\Q$, and the
Gysin sequence splits into short exact sequences. (This is no longer
true when $\charr \F > 0$ unless $\charr \F$ and the order of
iteration of $x$ are relatively prime.)  As a consequence,
\begin{equation}
  \label{eq:SH-CH}
\SH(x;\Q)=\COH(x;\Q)\oplus \COH(x;\Q)[-1] .
\end{equation}

As was mentioned above, a closed orbit $x$ is said to be
\emph{$\F$-visible} if $\SH(x;\F)\neq 0$. Likewise, $x$ is
\emph{equivariantly $\F$-visible} if $\COH(x;\F)\neq 0$. An
$\F$-visible closed orbit is automatically equivariantly $\F$-visible
by the Gysin sequence. The converse is also true when $\F=\Q$ due to
\eqref{eq:SH-CH}; see \cite{CGG:Reeb-HZ2} for a further discussion.

It is clear from this construction that the Euler characteristic of
$\SH(x;\F)$ is always zero for any field $\F$ and any closed Reeb
orbit $x$, i.e.,
\begin{equation}
  \label{eq:chi0}
\chi(x;\F):=\sum_m (-1)^m\dim\SH_m(x;\F)=0.
\end{equation}
Indeed, as stated in the beginning of this section $\SH(x;\F)$ is the
homology of a complex with generators coming in pairs $\cx_i$ and
$\hx_i$ where the degree of $\cx_i$ is $\mu(x_i)$ and the degree of
$\hx_i$ is $\mu(x_i)+1$. Therefore,
$$
\dim\SH(x;\F) \textrm{ is always even,}
$$
and so 
\begin{equation}
  \label{eq:dimgeq2}
  \dim\SH(x;\F)\geq 2 \textrm{ when $x$ is $\F$-visible.}
\end{equation}

The \emph{equivariant Euler characteristic} of an orbit $x$ is defined
as
\begin{equation}
  \label{eq:chieq}
  \chieq(x)=\sum_m (-1)^m\dim\COH_m(x;\Q).
\end{equation}
For instance, when $x$ is non-degenerate, $\chieq(x)$ is either
$(-1)^{\mu(x)}$ or $0$ depending on whether $x$ is good or bad. We
note that here we restrict our background field $\F$ to $\Q$ because
otherwise $\supp\COH(x;\F)$ can be infinite and then the right-hand
side of \eqref{eq:chieq} is not defined.

It readily follows from the description of $\COH(x;\Q)$ in Section
\ref{sec:local-def} that, whenever $x$ is prime, $\chieq(x)$ is equal
to the Hopf index of the Poincar\'e return map of $x$ up to the sign
$(-1)^n$. This is no longer true when $x$ is iterated. For instance,
$\chieq(x)=0$ when $x$ is non-degenerate and bad, while the Hopf index
is $\pm 1$. Moreover, let $x=z^k$ where $z$ is prime and, for the sake
of simplicity, $k$ is odd. Then $\chieq(x)$ counts the number (with
signs) of $k$-periodic orbits of a small non-degenerate perturbation
of the return map $\varphi$ of $z$, while the Hopf index of
$\varphi^k$ counts (with the same signs) the number of $k$-periodic
points.

\subsubsection{Behavior under iterations}
\label{sec:iterates}
Recall that a positive integer $k$ and the iterate $x^k$ are said to
be \emph{admissible} for $x$ if $k$ is not divisible by the degree of
any root of unity other than $1$ among the eigenvalues of the
linearized Poincar\'e return map of $x$. For instance, all $k\in\N$
are admissible when $x$ is \emph{totally degenerate}, i.e., all
eigenvalues are equal to 1.

We will need the following simple fact which is conceptually closely
related to the Shub--Sullivan Theorem and its proof; see \cite{SS}.

\begin{Lemma}
  \label{lemma:chieq-it}
  Let $x$ be an isolated (not necessarily prime) closed Reeb orbit.
  Assume that $k$ is admissible for $x$ and odd if the number of
  eigenvalues of the return map of $x$ in $(-1,0)$ is odd. Then
\begin{equation}
  \label{eq:chieq-it}
  \chieq\big(x^k\big)=\chieq(x).
\end{equation}
In particular, $x^k$ is $\Q$-visible when $\chieq(x)\neq 0$.
\end{Lemma}

\begin{proof}
  Fix $k$ and consider a $C^2$-small non-degenerate perturbation of
  the background contact form. Under this perturbation $x$ splits into
  strongly non-degenerate closed orbits $x_i$. The key fact is that
  $x_i^k$ are the only closed orbits of the perturbed flow in a
  neighborhood of $x$ with period close to $k\CA(x)$, provided that
  $k$ is admissible and the perturbation is sufficiently small
  (depending on $k$). This follows for instance from \cite[Prop.\
  1.1]{GG:gap}.

  We claim that $x_i^k$ is good (bad) if and only if so is $x_i$, and
  that $\mu(x_i)$ and $\mu\big(x^k_i\big)$ have the same parity. When
  $k$ is odd this is obvious. Indeed, the eigenvalues in $(-1,0)$
  remain in that range for all iterates and additional eigenvalues
  come in pairs, which implies the parity statement. Likewise, an odd
  iterate of a good (bad) closed orbit is automatically good (bad).
  If $-1$ is an eigenvalue of the return map of $x$, the iteration $k$
  is necessarily odd, since $k$ is admissible, and the claim follows.
  Assume that $-1$ is not an eigenvalue and $k$ is even.  Then for
  each $x_i$ the parity of the number of real eigenvalues in $(-1,0)$
  is the same as for $x$ and, by the conditions of the lemma, even.
  Hence, in this case, $\mu(x_i)$ and $\mu\big(x^k_i\big)$ also have
  the same parity. Furthermore, if $x_i=z^r$ where $z$ is prime, $x_i$
  is good (bad) if and only $\mu(x_i)$ and $\mu(z)$ have the same
  (opposite) parity. This is also true for $x_i^k=z^{kr}$. As a
  consequence $x_i$ and $x_i^k$ are good (bad) simultaneously.

  As described in Section \ref{sec:local-def}, $\COH(x;\Q)$ is the
  homology of a certain complex generated by good closed orbits $x_i$,
  and $\COH\big(x^k;\Q\big)$ is the homology of a complex generated by
  good $x_i^k$. In both cases the complex is graded by the
  Conley--Zehnder index. Therefore,
  $$
  \chieq\big(x^k\big)=\sum_{\text{good } x_i^k}
  (-1)^{\mu(x^k_i)}=\sum_{\text{good } x_i} (-1)^{\mu(x_i)}=\chieq(x).
  $$
  \end{proof}

  Next, for an isolating neighborhood $U$ of a closed Reeb orbit $x$,
  consider the $k$-fold covering $\pi\colon U_k\to U$ and the
  pull-back of the contact form from $U$ to $U_k$. This pull-back and
  hence the Reeb vector field are invariant under the action of $\Z_k$
  on $U_k$ by deck transformations. Then $\pi^{-1}(x(\R))$ is an
  isolated closed Reeb orbit of the lift, which we denote by $kx$,
  provided that $x^k$ is isolated. The period (action) of $kx$ is
  $k\CA(x)$. By Remark \ref{rmk:local}, we have the local symplectic
  homology $\SH(kx;\F)$ and $\COH(kx;\F)$ of $kx$ defined.

  \begin{Lemma}
    \label{lemma:iterates}
    There is a natural $\Z_k$-action on $\SH(kx;\F)$ and
    $\COH(kx;\F)$, and
    \begin{equation}
      \label{eq:iterates}
      \SH\big(x^k;\Q\big)=\SH(kx;\Q)^{\Z_k}\textrm{ and }
      \COH\big(x^k;\Q\big)=\COH(kx;\Q)^{\Z_k},  
    \end{equation}
    where the superscript refers to the $\Z_k$-invariant part of the
    homology.
\end{Lemma}

Note that here the orbit $x$ need not be prime and $k$ is not required
to be admissible. However, $x^k$ must be isolated. Furthermore, the
condition that $\F=\Q$ in \eqref{eq:iterates} can be relaxed. It would
be sufficient to assume that $\charr \F$ and $k$ are relatively prime.

\begin{proof}
  The argument is quite standard and we only briefly outline it.  For
  the sake of brevity, we focus on the non-equivariant case; for the
  equivariant version is handled in a similar fashion.

  In the notation from Section \ref{sec:local-def}, let
  $H=h(r)\colon \UU\to \R$ be a Hamiltonian from the construction of
  $\SH\big(x^k;\F\big):=\HF(H;\F)$. Fix a non-degenerate 1-periodic in
  time perturbation $F$ of $H$ and a regular 1-periodic almost complex
  structure $J$ on $\UU$. Thus, in particular,
  $\HF(H;\F)=\HF(F;\F)$. Denote by $H_k$, $F_k$ and $J_k$ the
  pull-backs of the Hamiltonians and $J$ to the $k$-fold covering
  $\pi\colon\UU_k\to\UU$. Then
  $\SH(kx;\F):=\HF(H_k;\F)=\HF(F_k;\F)$. These pull-backs are
  invariant under the $\Z_k$-action on $\UU_k$ by deck
  transformations. The 1-periodic orbits of $F_k$ project to the
  periodic orbits of $F$, and so do Floer cylinders.  Moreover, on the
  sets of 1-periodic orbits and Floer cylinders these projections are
  onto and $k$-to-1. To be more precise, since all 1-periodic orbits
  of $F$ are in the free homotopy class $k[S^1]$, the Hamiltonian
  $F_k$ has exactly $k$ distinct 1-periodic orbits lying above every
  1-periodic orbit of $F$, and deck transformations act freely on
  these orbits. It follows that the same is true for Floer cylinders;
  cf.\ \cite[Prop.\ 1.21]{Sa}.

  We claim that the pair $(F_k,J_k)$ is regular if and only if the
  pair $(F,J)$ is regular; cf.\ \cite[Prop.\ 5.13]{KS}. Indeed, let
  $u$ be a Floer cylinder in $\UU_k$. Then the pull-backs $u^*T\UU_k$
  and $(\pi u)^*T\UU$ are naturally isomorphic. (This is true whenever
  $\pi$ is a local diffeomorphism.) Hence, we obtain an isomorphism
  between the Sobolev spaces of sections of these bundles; see
  \cite[Sec.\ 2]{Sa} for the definitions. By construction, this
  isomorphism intertwines the linearized perturbed Cauchy--Riemann
  operators $D_u$ along $u$ and $D_{\pi u}$ along $\pi u$. Therefore,
  $D_u$ is onto if and only if $D_{\pi u}$ is onto. Finally, since
  every Floer cylinder in $\UU$ has a lift to $\UU_k$, we conclude
  that the pairs $(F,J)$ and $(F_k,J_k)$ are regular simultaneously.
  
  As a consequence, $\Z_k$ acts on the Floer complex $\CF(H_k,J_k)$ of
  $(H_k,J_k)$ and hence on $\SH(kx;\F)=\HF(F_k;\F)$. The standard
  continuation argument shows that this action is independent of the
  auxiliary choices made in the construction.

  When $\F=\Q$, the Floer complex $\CF(F,J)$ is naturally isomorphic
  to the invariant part $\CF(H_k,J_k)^{\Z_k}$ of
  $\CF(H_k,J_k)$. Indeed, in one direction we have (for any $\F)$ the
  homomorphism of complexes $\CF(H_k,J_k)\to \CF(F,J)$ induced by the
  projection $\pi\colon\UU_k\to\UU$. In the opposite direction, we
  have the averaging map. Namely, for $z\in\CF(F,J)$, pick an
  arbitrary lift $z'\in \CF(H_k,J_k)$ and set, in the obvious
  notation,
  $$
  z\mapsto \frac{1}{k}\sum_{g\in\Z_k} g z'\in \CF(H_k,J_k)^{\Z_k}.
  $$
  This map gives rise to an isomorphism
  $\CF(F,J)\cong \CF(H_k,J_k)^{\Z_k}$, which concludes the proof of
  \eqref{eq:iterates}.
\end{proof}

Local symplectic homology of $x$ over $\Q$ is intimately connected to
the local Floer homology $\HF(\varphi; \Q)$ of the return map
$\varphi$, although this connection is not straightforward. When $x$
is prime,
\begin{equation}
  \label{eq:Fe}
  \COH(x;\Q)=\HF(\varphi;\Q)\textrm{ and } 
  \SH(x;\Q)=\HF(\varphi;\Q)\oplus \HF(\varphi;\Q)[-1]
\end{equation}
as is shown in \cite{Fe,Fe:thesis}. (These two identifications are
equivalent due to \eqref{eq:SH-CH}.) Applying this to $kx$, which is
prime whenever $x$ is prime and using \eqref{eq:iterates}, we conclude
that
\begin{equation}
    \label{eq:McL}
    \dim\SH_m\big(x^k;\Q\big)\leq \dim\HF_m\big(\varphi^k;\Q\big)
    + \dim\HF_{m-1}\big(\varphi^k;\Q\big)
\end{equation}
for all $k\in\N$, provided that $x$ is prime and $x^k$ is isolated. A
closely related inequality is established in \cite[Lemma
3.4]{McL}. The inequalities in \eqref{eq:McL} can be strict. This is
the case, for instance, when $x$ is non-degenerate and alternating and
$k$ is even; see also Examples \ref{exam:iterates} and
\ref{exam:iterates2} below.

Conjecturally,
$$
\COH\big(x^k;\Q\big)=\HF\big(\varphi^k;\Q\big) ^{\Z_k}
$$
and
$$
\SH\big(x^k;\Q\big)=\HF\big(\varphi^k;\Q\big) ^{\Z_k}\oplus
\HF\big(\varphi^k;\Q\big) ^{\Z_k}[-1]
$$
for all $k\in\N$; see \cite{GHHM}. Lemma \ref{lemma:iterates} together
with \eqref{eq:Fe} almost proves these identifications. The missing
part is that the isomorphisms in \eqref{eq:Fe} commute with the
$\Z_k$-action when applied to $kx$.

\begin{Remark}  
  When the orbit $x$ is prime, the local cylindrical contact homology
  of $x$ is defined and isomorphic to $\HF(\varphi;\Q)$; see
  \cite{HM}. (Alternatively, one can use here a local version of
  \cite[Prop.\ 4.30]{EKP}.) Furthermore, localizing the isomorphism
  between the equivariant symplectic and cylindrical contact homology
  from \cite{BO} we see that the local cylindrical contact homology of
  $x$ is isomorphic to $\COH(x;\Q)$, assuming again that $x$ is
  prime. (Hence the notation.) This gives another proof of the
  isomorphisms in \eqref{eq:Fe}, albeit quite indirect.
\end{Remark}

  \begin{Remark}
    \label{rmk:iterates1}
    The behavior under iterations of local symplectic or Floer
    homology (equivariant or not) is complicated and overall poorly
    understood. For instance, it is not known if a closed orbit can go
    from invisible to visible under iterations, i.e., if a closed
    orbit $x^k$ can be $\F$-visible when $x$ is prime and
    $\F$-invisible, without additional assumptions on $x$, the ground
    field $\F$ or $k$. At the same time, $x^k$ can in general become
    $\F$-invisible when $x$ is prime and $\F$-visible. A simple
    situation when this happens is when $x$ is non-degenerate and
    alternating, $k$ is even and $\F=\Q$. In Examples
    \ref{exam:iterates} and \ref{exam:iterates2} we describe more
    interesting cases of this phenomenon.  Lemma \ref{lemma:chieq-it}
    provides a simple non-vanishing criterion, which is sufficient for
    our purposes. (The question for a variant of local symplectic
    homology and admissible iterations has been investigated in
    \cite{HHM}. The homology there is defined via generating functions
    and is hypothetically isomorphic to the local homology used here;
    cf.\ Remark \ref{rmk:iterates2}.)
\end{Remark}

\begin{Example}[Vanishing of iterated symplectic homology]
\label{exam:iterates}
Let $\psi\colon \R^2=\C\to\C=\R^2$ be the time-one flow of the
degree-three monkey saddle Hamiltonian $H=\Re (z^3)$ and let
$R_{2\pi/3}$ be the rotation of $\R^2$ counterclockwise by
$2\pi/3$. (We view $R_{2\pi/3}$ as an element of the universal
covering $\tSp(2)$.) Note that $R_{2\pi/3}$ and the flow of $H$
commute. Set $\varphi=R_{2\pi/3}\psi$ and let $x$ be a closed Reeb
orbit with the return map $\varphi$. The orbit $x$ is non-degenerate,
$\COH(x;\Q)=\Q$ and for a suitable trivialization $\COH(x;\Q)$ is
supported in degree one. The return map of $x^3$ is $\varphi^3$, which
is the time-three flow of $H$.  One can show that
$\COH(x^3;\Q)=0$. (This is not obvious.)  However, $\HF(\varphi^3;\Q)$
is $\Q^2$ and supported in degree zero. This gives another example
where the inequality in \eqref{eq:McL} is strict; see also Example
\ref{exam:iterates2} below.
\end{Example}

\begin{Example}
  \label{exam:iterates2}
  The notion of a (non-)alternating closed orbit is not particularly
  useful in the setting where $x$ is not strongly non-degenerate and
  the return map $\varphi$ has roots of unity of even degree among its
  eigenvalues. The reason is that in this case the behavior of the
  local symplectic homology is not determined by the eigenvalues and
  depends on higher order terms. For instance, in the spirit of
  Example \ref{exam:iterates}, let $R_\pi\colon \R^2\to \R^2$ be the
  counterclockwise rotation by $\pi$, i.e., central symmetry, viewed
  as an element of $\tSp(2)$. Consider the Hamiltonians $H_0=p^4-q^4$
  and $H_1=(p^2+q^2)^2$ on $\R^2$, where $(p,q)$ are the Darboux
  coordinates. The flows of $H_0$ and $H_1$ commute with $R_\pi$. Let
  $\varphi_i$, for $i=0, 1$, be obtained by composing $R_\pi$ with the
  time-one flow of $H_i$ and $x_i$ be a closed Reeb orbit with the
  return map $\varphi_i$. Clearly, $\COH(x_i;\Q)=\HF(\varphi_i)=\Q$ is
  supported in degree one. However, one can show that
  $\COH\big(x_0^2;\Q\big)=0$ and $\HF\big(\varphi_0^2;\Q\big)=\Q$
  (supported in degree two). At the same time,
  $\COH\big(x_1^2;\Q\big)=\HF\big(\varphi_1^2;\Q\big)=\Q$ (supported
  in degree three).
\end{Example}  

\subsubsection{From local to filtered homology}
In what follows, it is convenient to adopt the following convention;
cf.\ Section \ref{sec:SH-pers}. Namely, we will purely formally treat
the domain $W$ as a collection of ``closed orbits'' with action
zero. By definition, these orbits are in one-to-one correspondence
with the bars in $\SH(W;\F)$ beginning at zero; each of these orbits
has one-dimensional symplectic homology concentrated in the degree of
that bar and there are $\dim\H(W;\F)$ such bars. For instance, when
$W$ is star-shaped it contributes one orbit $[W]$ that has action zero
and degree $n$.

Local symplectic homology spaces are building blocks for the filtered
symplectic homology. This principle can be formalized in several ways.
For instance, assume that $a$ is the only point of $\CS(\alpha)$
in the interior of an interval $I=[t,s)$. Then
\begin{equation}
  \label{eq:local-filt}
\sum_{\CA(x)=a}\dim \SH_m(x;\F) =\dim \SH^I_m(W;\F).
\end{equation}
In fact,
  \begin{equation}
    \label{eq:sum}
  \bigoplus_{\CA(x)=a}\SH(x;\F) =\SH^I(W;\F)
\end{equation}
by the definition of local symplectic homology. Therefore, we have the
long exact sequence
  \begin{equation}
    \label{eq:long-exact}
   \ldots \to\SH^t_m(W;\F)\to \SH^s_m(W;\F)
   \to\bigoplus_{\CA(x)=a}\SH_m(x;\F)
   \to\SH^t_{m-1}(W;\F)\to\ldots .
 \end{equation}
 Moreover, for any field $\F$ and any interval $I$, there exists a
 spectral sequence
\begin{equation}
  \label{eq:spec-SH}
  E^r\Rightarrow \SH^I(W;\F) \textrm{ with }
  E^1=\bigoplus_{\CA(x)\in I}\SH(x;\F),
\end{equation}  
where, when $0\in I$, we treat $W$ as a collection of ``orbits'' with
action 0 as described above. This spectral sequence is simply
associated with the action filtration; see, e.g., \cite[Sec.\
2]{GG:LS}.

In a similar vein, we have analogues of \eqref{eq:sum} and
\eqref{eq:long-exact} for equivariant symplectic homology and
a variant of the Leray spectral sequence 
\begin{equation}
  \label{eq:spec-CH}
E^r \Rightarrow \COH^I(W;\F) \textrm{ with }
E^1=E^2=\SH^I(W;\F)\otimes \H_*(BS^1;\F),
\end{equation}
where for the sake of simplicity we have assumed that $0\not\in I$.
For instance, we have $\COH^I(W;\F)=0$ whenever $\SH^I(W;\F)=0$. In this
spectral sequence, $E^r_{q,p}=0$ when $q$ is odd. Hence, all odd
differentials vanish, and $E^1=E^2$, $E^3=E^4$, etc; see
\cite{CGG:Reeb-HZ2} for a further discussion.

As a consequence, when $W$ is a star-shaped domain, for every
$m\in n-1+2\N$ there must be a closed orbit $x$ with
$m\in\supp\COH(x;\F)$ due to \eqref{eq:CH}. Furthermore, if
$q\in n+2\N$ and $q\in\supp\COH(y;\F)$ for some closed orbit
  $y$, we must have
\begin{equation}
  \label{eq:odd}
  \sum\dim\COH_{q-1}(z;\F)\geq 2 \quad \textrm{or} \quad
  \sum\dim\COH_{q+1}(z;\F)\geq 2,
\end{equation}
where the sum is over all closed orbits $z$.
This can also be easily seen from \eqref{eq:CH} and the equivariant
version of \eqref{eq:long-exact}.

\begin{Remark}
\label{rmk:COH-complex}
When $\F=\Q$, the flow is non-degenerate and again
$I\subset (0,\infty)$, the spectral sequence \eqref{eq:spec-CH}, can
be reassembled into a complex. Namely, we have a complex generated
over $\Q$ by good closed orbits with action in $I$, graded by the
Conley--Zehnder index and filtered by the action, and such that its
homology is $\COH^I(W;\Q)$; see \cite[Sec.\ 2.5]{GG:LS}. The
differential in the complex is completely determined by the same
auxiliary data as the one from the construction of the equivariant
symplectic or Floer homology; it is functorial in a suitable sense,
etc. (Without non-degeneracy condition, the complex is the $E^1$-term
in \eqref{eq:spec-CH} but then we also need to fix more additional
data to define the differential.)  We omit the details of this
construction because it is not explicitly used in the paper. Here we
only note that this is no longer literally true when $\charr\F\neq 0$;
for then $\supp\COH(x;\F)$ can, for instance, be infinite.
\end{Remark}

Finally, let $\zeta_m^-(a)$ be the number of bars in the barcode of
$\SH(W;\F)$ of degree $m-1$ ending at $a$ and $\zeta^+_m(a)$ be the
number of bars of degree $m$ beginning at $a$. Set
$\zeta_m(a)=\zeta_m^-(a)+\zeta_m^+(a)$. The following observation is
nearly obvious.

\begin{Lemma}[``Morse Theory'']
  \label{lemma:bars-SH}
For all $a\in\R$, we have  
\begin{equation}
  \label{eq:bars-SH}
  \zeta_m(a)=\sum_{\CA(x)=a}\dim \SH_m(x; \F).
\end{equation}
\end{Lemma}

\begin{proof}
  Throughout the proof we suppress the ground field $\F$ in the
  notation. It is sufficient to prove the lemma when
  $a\in \CS(\alpha)$ or $a=0$; for otherwise both sides of
  \eqref{eq:bars-SH} are zero. Using the background assumption
  \ref{BA}, fix $\eps>0$ so small that $a$ is the only point of the
  intersection $I\cap(\CS(\alpha)\cup\{0\})$, where
  $I=(a-\eps,a+\eps)$. Clearly,
\[
  \zeta_m^-(a)=\dim\ker\left(\SH^{a-\eps}_{m-1}(W) \to
    \SH^{a+\eps}_{m-1}(W)\right)
\]
and
\[
  \zeta_m^+(a)=
  \dim\coker\left(\SH^{a-\eps}_{m}(W)\to\SH^{a+\eps}_{m}(W)\right).
\]
From the long exact sequence
$$
\cdots\to \SH^{a-\eps}_m(W)\to \SH^{a+\eps}_m(W)\to  \SH^{I}_m(W) \to
\SH^{a-\eps}_{m-1}(W) \to  \SH^{a+\eps}_{m-1}(W)\to \cdots 
$$
we see that also
$$
\zeta_m^-(a)=\dim\im\left(\SH^{I}_{m}(W)\to\SH^{a-\eps}_{m-1}(W)\right)
$$
and
$$
\zeta_m^+(a)=\dim\im\left(\SH^{a+\eps}_{m}(W)\to\SH^{I}_{m}(W)\right)
=\dim\ker\left(\SH^{I}_{m}(W)\to\SH^{a-\eps}_{m-1}(W)\right).
$$
Hence, again since the sequence is exact,
$$
\zeta_m^-(a)+\zeta_m^+(a)=\dim \SH^I_m(W),
$$
and \eqref{eq:bars-SH} follows from \eqref{eq:local-filt}.
\end{proof}

\subsection{The beginning and end maps}
\label{sec:beg-end}
For a fixed field $\F$, denote by $\CB$ or $\CB(\F)$ the collection of
bars of the graded persistence module $\SH(W;\F)$. Assume that
$\SH^\infty(W;\F)=0$. Then every bar in $\CB$ is finite by Theorem
\ref{thm:vanishing}. Furthermore, let $\PP=\PP(\F)$ be the set of
$\F$-visible closed Reeb orbits. As in the previous section, let us
also formally add to $\PP$ the set of the ``beginning'' of the bars of
the form $(0,a]$ for some $a>0$, denoting the resulting set by
$\PP'$. These additional elements $x$ of $\PP'$ are simply in
one-to-one correspondence with such bars.  By definition, each such
``orbit'' $x$ has symplectic homology equal to $\F$ and concentrated
in $\deg(I)$, and zero action.  (We also set
$\mu_-(x)=\mu_+(x):=\deg(x)$ for such $x$.) When $W$ is star-shaped,
we are adding just one extra element corresponding to the domain $W$,
which we denote by $[W]$. This element has action zero and is
supported in degree $n$.

The sets $\PP$ and $\PP'$ depend on $\F$ or, to be more precise, on
$\charr\F$ even when the Reeb flow is non-degenerate and $W$ is
star-shaped. For instance, when $\charr \F=2$ we have two generators
for each closed orbit of the flow regardless of whether the orbit is
good or bad. However, when $\charr \F\neq 2$, only good orbits
contribute.

\begin{Theorem}
  \label{thm:beg-end}
  There exist maps $\beg\colon \CB\to \PP'$ (the beginning of $I$) and
  $\en\colon \CB\to \PP$ (the end of $I$) such that for every bar
  $I=(a,b]$ and $x=\beg(I)$ and $y=\en(I)$ we have
\begin{itemize}
\item $\CA(x)=a$ and $\CA(y)=b$,
\item $\deg(I)\in \supp \SH(x;\F)$ and $\deg(I)+1\in \supp \SH(y;\F)$.
\end{itemize}
\end{Theorem}

As a consequence of the second of these requirements, we see
that
$$
\mu_-(x)\leq \deg(I)\leq \mu_+(x)+1
$$
and 
$$
\mu_-(y)-1\leq \deg(I)\leq \mu_+(y),
$$
and therefore
\begin{equation}
  \label{eq:mu(x)-mu(y)1}
\mu_-(y)-\mu_+(x)\leq 2.
\end{equation}
Moreover,
\begin{equation}
  \label{eq:mu(x)-mu(y)2}
\deg(I)=\mu_+(x)+1=\mu_-(y)-1,
\end{equation}
whenever the equality holds in \eqref{eq:mu(x)-mu(y)1}. Note also that
the second condition of the theorem necessitates that the orbits in
$\PP$ are $\F$-visible. We will see later that in the setting of
Theorem \ref{thm:SH} and as a consequence of that theorem, the maps
$\beg$ and $\en$ are onto. We do not know if this is true in
general. Furthermore, in the setting of Theorem \ref{thm:SH} and again
as a consequence of that theorem, both maps are one-to-one. This is
certainly not true in general. For instance, when $\dim\SH(x;\F)= 3$
and $x$ is the only closed orbit with action $\CA(x)$, at least one of
the pre-images $\beg^{-1}(x)$ or $\en^{-1}(x)$ must contain more than
one point.

In general, even when the Reeb flow is non-degenerate, the maps $\beg$
and $\en$ are not canonically defined and furthermore, just as the set
$\PP$, depend on $\charr\F$. However, the construction can be made
canonical if all closed orbits have distinct actions.

\begin{proof}[Proof of Theorem \ref{thm:beg-end}]
Let $\F$ be the underlying coefficient field and let
$$
e(m, a)_1, e(m, a)_2, \dots, e(m, a)_{\zeta_m(a)} 
$$
be an $\F$-basis for
$$
E:=\bigoplus_{\A(x)=a} \SH_m(x;\F),
$$
where $\zeta_m(a)$ is the bar counting function from Lemma
\ref{lemma:bars-SH}. To keep track of the orbits, we require in
addition the basis elements $e(m, a)_i$ to be adapted to the above
direct sum decomposition of $E$. Next, let
$\CB = \{ I_1, I_2, \ldots \}$ be an arbitrary ordering by $\N$ of the
bars in the symplectic homology barcode $\CB$. (Note that, by the
background assumption \ref{BA}, there are countably many bars in
$\CB$.) Let us denote by $\zeta^i_m(a)$ the restriction of
$\zeta_m(a)$ to the ``truncated'' barcodes
$\CB^i = \{ I_1, I_2, \dots, I_i\}$. In other words, $\zeta^i_m(a)$ is
the total number of bars of degree $m-1$ ending at $a$ and of degree
$m$ beginning at $a$ among the first $i$ bars.

Now we are in a position to define the maps $\beg$ and $\en$. To
simplify the notation, set $(c_i, b_i] :=I_i$ and $k_i :=
\deg(I_i)$. Let us denote by $\pi$ the natural projection map sending
$ e(m, a)_i$ to the corresponding element in $\mathcal{P}'$. (Here we
use the requirement that the basis respects the direct sum
decomposition). We define $\beg \colon \CB \to \mathcal{P}'$ by
$$
I_i \mapsto \pi \big (e(k_i, c_i)_{\zeta^i_{k_i}(c_i)} \big )
$$ 
and, similarly, $\en \colon \CB \to \mathcal{P}'$ by 
$$
I_i \mapsto \pi \big ( e(k_i+1, b_i)_{\zeta^i_{k_i+1}(b_i)} \big ).
$$ 
We leave it to the reader to check that the maps above satisfy the
requirements.
\end{proof}

\begin{Remark}
  \label{rmk:beg-end}  
  Our requirements on the maps $\beg$ and $\en$ are deliberately
  minimal. For instance, with more effort, we could construct these
  maps so that the sum of the number of bars of degree $m-1$ in
  $\en^{-1}(x)$ and the number of bars in $\beg^{-1}(x)$ of degree $m$
  is equal to $\dim\SH_m(x;\F)$. This would be a refinement of
  \eqref{eq:bars-SH} in Lemma \ref{lemma:bars-SH}. However, in the
  cases we are interested in, this condition is automatically
  satisfied as a consequence of \eqref{eq:bars-SH}, and we omit this
  construction.
\end{Remark}

\section{Conley-Zehnder index and index recurrence}
\label{sec:IR}
The goal of this section is two-fold. In the first part we
review several Conley--Zehnder index type invariants and their
properties, and the notion of dynamical convexity, starting with the
linear algebra setting and then turning to closed Reeb orbits. This
material is quite standard. The second part is devoted to index
recurrence and its illustrative applications, and this part is mainly
new.

\subsection{Background: the Conley--Zehnder index and dynamical
  convexity}
\label{sec:CZ+DC}

\subsubsection{Conley--Zehnder index: linear algebra}
\label{sec:CZ-LA}
In this section for, the reader's convenience we recall some basic
properties of the mean and Conley--Zehnder indices. We refer the
reader to, e.g., \cite{Lo} or \cite[Sec.\ 3]{SZ} or \cite[Sec.\
4]{GG:LS}, which we closely follow here, for a more thorough
treatment.

Let $\tSp(2m)$ be the universal covering of the group $\Sp(2m)$ of
symplectic linear transformations of $\R^{2m}$. We view elements of
$\tSp(2m)$ as paths starting at the identity, taken up to homotopy
with fixed end-points, and abusing terminology we will often refer to
elements of $\tSp(2m)$ as maps. Unless specifically stated otherwise,
we assume that such paths are parametrized by $[0, 1]$.

A map $\hmu$ from a Lie group to $\R$ is said to be a quasimorphism if
it fails to be a homomorphism only up to a constant, i.e.,
$$
\big|\hmu(\Phi\Psi)-\hmu(\Phi)-\hmu(\Psi)\big|<\const,
$$ 
where the constant is independent of $\Phi$ and $\Psi$. One can prove
that there is a unique quasimorphism
$\hmu\colon \widetilde{\Sp}(2m)\to\R$ which is continuous and
homogeneous (i.e., $\hmu(\Phi^k)=k\hmu(\Phi)$) and satisfies the
normalization condition:
$$
\hmu(\Phi_0)=2\quad \textrm{for}\quad \Phi_0(t)
=\exp\big(2\pi \sqrt{-1}
t\big)\oplus I_{2m-2}
$$ 
with $t\in [0,\,1]$, in the self-explanatory notation; see \cite{BG}.
The quasimorphism $\hmu$ is called the \emph{mean index}. The
continuity requirement holds automatically and is not necessary for
the characterization of $\hmu$. (This is not obvious). The mean index
is also automatically conjugation invariant, as a consequence of the
homogeneity, and $\hmu$ restricts to an isomorphism
$\pi_1(\Sp(2m))\to 2\Z$. Furthermore, $\hmu(\Phi)\in\Z$ whenever the
elliptic eigenvalues of $\Phi(1)$ are $\pm 1$ or when $\Phi(1)$ is
hyperbolic. (Elliptic eigenvalues are by definition unit eigenvalues
and the elliptic part of $A\in\Sp(2m)$ is the restriction of $A$ to
the sum of the generalized eigenspaces for all elliptic eigenvalues.)

Recall that $A\in\Sp(2m)$ is said to be \emph{non-degenerate} if all
eigenvalues of $A$ are different from one, and \emph{strongly
  non-degenerate} if all powers $A^k$, $k\in\N$, are
non-degenerate. (Similarly, an element $\Phi\in\tSp(2m)$ is
non-degenerate when $\Phi(1)$ is.) We denote the set of non-degenerate
symplectic linear maps by $\Sp^*(2m)$ and the part of $\tSp(2m)$ lying
over $\Sp^*(2m)$ by $\tSpn(2m)$. It is not hard to see that
$A=\Phi(1)$ can be connected to a symplectic map with the elliptic
part equal to $-I$ (if non-trivial) by a path $\Psi$ lying entirely in
$\Sp^*(2m)$. Concatenating this path with $\Phi$, we obtain a new path
$\Phi'$. By definition, the \emph{Conley--Zehnder index}
$\mu(\Phi)\in\Z$ of $\Phi$ is $\hmu(\Phi')$. One can show that
$\mu(\Phi)$ is well-defined, i.e., independent of $\Psi$. The function
$\mu\colon \tSpn(2m)\to\Z$ is locally constant, i.e., constant on
connected components of $\tSpn(2m)$.

The mean index $\hmu(\Phi)$ measures the total rotation angle of
certain unit eigenvalues of $\Phi(t)$. For instance,
for the path $\Phi(t)=\exp\big(2\pi\sqrt{-1}\lambda t\big)$,
$t\in [0,\,1]$, in $\Sp(2)$ we have
\begin{equation}
\label{eq:sp2}
\hmu(\Phi)=2\lambda \textrm{ and }
\mu(\Phi)=\sign(\lambda)\big(2\lfloor|\lambda|\rfloor +1\big)
\textrm{ when } \lambda\not\in\Z, 
\end{equation}
where $\sign(\lambda)$ stands, predictably, for the sign of $\lambda$.

The \emph{upper and lower
  Conley--Zehnder indices} are defined as
$$
\mu_+(\Phi):=\limsup_{\tPhi\to \Phi}\mu(\tPhi)
\quad\textrm{and}\quad
\mu_-(\Phi):=\liminf_{\tPhi\to \Phi}\mu(\tPhi),
$$
where in both cases the limit is taken over $\tPhi\in \tSpn(2m)$
converging to $\Phi\in \TSp(2m)$. In fact, $\mu_+(\Phi)$ is simply
$\max\mu(\tPhi)$, where $\tPhi\in \tSpn(2m)$ is sufficiently close to
$\Phi$ in $\tSp(2m)$; likewise, $\mu_-(\Phi)=\min \mu(\tPhi)$. (In
terms of actual paths, rather than their homotopy classes, $\tPhi$ can
be taken $C^r$-close to $\Phi$ for any $r\geq 0$; the resulting
definition of $\mu_\pm(\Phi)$ is independent of $r$ and equivalent to
the one above.)  Clearly, $\mu(\Phi)=\mu_\pm(\Phi)$ when $\Phi$ is
non-degenerate. The indices $\mu_\pm$ are the upper semi-continuous
and, respectively, lower semi-continuous extensions of $\mu$ from
$\tSpn(2m)$ to $\tSp(2m)$. These maps are of interest to us because
they bound the support of the local symplectic homology of an isolated
closed orbit; see \eqref{eq:support1} and \eqref{eq:support2}. To the
best of our knowledge, the indices $\mu_\pm$ were first introduced and
studied in \cite{Lo90, Lo97}.

The mean index and the upper and lower indices are related by the
inequalities
\begin{equation}
\label{eq:mu-del}
\hmu(\Phi)-m\leq\mu_-(\Phi)\leq \mu_+(\Phi)\leq \hmu(\Phi)+m.
\end{equation}
Moreover, at least one of the inequalities is strict when $\Phi(1)$ is
\emph{weakly non-degenerate}, i.e., at least one of the eigenvalues is
different from 1. As a consequence of \eqref{eq:mu-del}, $\mu_\pm$ are
quasimorphisms $\tSp(2m)\to\Z$, and
$$
\lim_{k\to\infty}\frac{\mu_\pm(\Phi^k)}{k}=\hmu(\Phi);
$$
hence the name ``mean index'' for $\hmu$.

Whenever $\Phi'\in\TSp(2m')$ and $\Phi''\in\TSp(2m'')$ we have the
\emph{direct sum} $\Phi=\Phi'\oplus \Phi''\in \TSp(2m)$, where
$m=m'+m''$, defined in a natural way. This is equivalent to that
 $$
 \Phi(1) = \Phi'(1)\oplus \Phi''(1) \textrm{ and }
 \hmu(\Phi)=\hmu(\Phi')+\hmu(\Phi'').
 $$
 It is easy to see that any decomposition $\Phi(1)=A\oplus B$ with
 $A\in\Sp(2m')$ and $B\in\Sp(2m'')$ extends (non-uniquely) to a direct
 sum decomposition of $\Phi$.

We call $\Psi\in\TSp(2m')$ the \emph{non-degenerate part} of
$\Phi\in\TSp(2m)$ when $\Psi(1)$ is non-degenerate,
$\Phi(1)=\Psi(1)\oplus A$, with $m'\leq m$ and
$A\in \Sp\big(2(m-m')\big)$ \emph{totally degenerate} (i.e., with all
eigenvalues equal to 1), and $\hmu(\Psi)=\hmu(\Phi)$. (The mean index
condition is imposed here to make $\Psi$ unique and ensure that
$\Psi=\Phi$ when $\Phi$ is non-degenerate.) When $\Phi$ is itself
totally degenerate we say that the non-degenerate part of $\Phi$ is
trivial. The non-degenerate part of $\Phi$ is a direct summand.

 The indices $\hmu$ and $\mu_\pm$ are additive under direct sum; see,
 e.g., \cite{Lo97} and \cite[Sec. 4.1]{GG:LS}. Namely, for
 $\Phi'\in\TSp(2m')$ and $\Phi\in\TSp(2m'')$, we have
\begin{equation}
\label{eq:add}
\hmu(\Phi'\oplus\Phi'')=\hmu(\Phi')+\hmu(\Phi'')
\quad\textrm{and}\quad
\mu_\pm(\Phi'\oplus\Phi'')=\mu_\pm(\Phi')+\mu_\pm(\Phi'').
\end{equation}

Let us now briefly recall the definitions of several invariants of
$\Phi$, which play a central role in the index recurrence theorem
(Theorem \ref{thm:IR-A}). We refer the reader to \cite[Sec.\
4.1.3]{GG:LS} for more details and proofs.

Consider first a totally degenerate symplectic linear map
$A\in\Sp(2m)$. Then $A=\exp(JQ)$ where $J$ is the matrix of the
standard linear symplectic form on $\R^{2n}$ and all eigenvalues of
$JQ$ are equal zero. The quadratic form $Q$ can be symplectically
decomposed into a sum of terms of four types:
\begin{itemize}
\item the identically zero quadratic form on $\R^{2\nu_0}$,
\item the quadratic form $Q_0=p_1q_2+p_2q_3+\cdots+p_{d-1}q_d$ in
  Darboux coordinates on $\R^{2d}$, where $d\geq 1$ is odd,
\item the quadratic forms $Q_\pm=\pm(Q_0+p^2_d/2)$ on $\R^{2d}$ for
  any $d$.
\end{itemize}
These normal forms are taken from \cite[Sec.\ 2.4]{AG}; see also
\cite{LM} and \cite[Sect.\ 2.6]{CLW} and Section
\ref{sec:isospectral}. We find them more convenient to work with than
the original Williamson normal forms; see \cite{Wi} and, e.g.,
\cite[App.\ 6]{Ar}.

Denoting by $\sgn$ the signature, we have $\dim\ker Q_0=2$ and
$\sgn Q_0=0$, and $\dim\ker Q_\pm=1$ and $\sgn Q_\pm=\pm 1$.  Let
$b_*(Q)$, where $*=0,\pm$, be the number of the $Q_0$ and $Q_\pm$
terms in the decomposition. Let us also set $b_*(A):=b_*(Q)$ and
$\nu_0(A):=\nu_0(Q)$, and $b_*(\Phi):=b_*(A)$ when $A=\Phi(1)$ is
totally degenerate.  These are symplectic invariants of $Q$ and $A$
and $\Phi$, and the geometric multiplicity of the eigenvalue $1$ is
$2(b_0+\nu_0)+b_++b_-$.

These definitions readily extend to all paths. Namely, every
$\Phi\in\TSp(2m)$ can be written (non-uniquely) as a direct sum
$\Phi_0\oplus \Psi$ where $\Phi_0(1)\in\Sp(2m_0)$ is totally
degenerate and $\Psi(1)\in\Sp(2m_1)$ is non-degenerate. In particular,
$m_0$ is the algebraic multiplicity of the eigenvalue 1 of $\Phi(1)$,
and $m_0+m_1=m$. Then we set
$$
b_*(\Phi):=b_*\big(\Phi_0(1)\big)\textrm{ for $*=0,\pm$ and }
\nu_0(\Phi):=\nu_0\big((\Phi_0(1)\big).
$$
These are symplectic invariants of $\Phi$. Furthermore, set
$$
\beta_\pm:=\nu_0+b_0+b_\pm.
$$
It readily follows from the definition that
\begin{equation}
  \label{eq:beta-it}
\beta_\pm\big(\Phi^k\big)=\beta_\pm(\Phi)
\end{equation}
for all $k\in\N$ when $\Phi$ is totally degenerate or, more generally,
for admissible $k$, where by definition
$\beta_\pm(\Phi):=\beta_\pm(\Phi_0)$. (This is also true for $b_*$ and
$\nu_0$.)

It is not hard to show that
\begin{equation}
\label{eq:bpm}
\mu_\pm(\Phi)=\mu(\Psi)+ \hmu(\Phi_0)\pm\beta_\pm(\Phi);
\end{equation}
see \cite[Sec.\ 4.1]{GG:LS}. In particular, when $\Psi$ is the
non-degenerate part of $\Phi$, i.e., in addition $\hmu(\Phi_0)=0$, we
have
\begin{equation}
\label{eq:bpm2}
\mu_\pm(\Phi)=\mu(\Psi)\pm\beta_\pm(\Phi),
\end{equation}

Furthermore, as readily follows from the definitions, the invariants
$b_*$ and $\nu_0$ and $\beta_\pm$ are additive in the sense of
\eqref{eq:add}. For instance, for $\Phi'\in\TSp(2m')$ and
$\Phi\in\TSp(2m'')$, we have
\begin{equation}
\label{eq:add2}
\beta_\pm(\Phi'\oplus\Phi'')=\beta_\pm(\Phi')+\beta_\pm(\Phi'').
\end{equation}
Likewise,
$$
\beta_\pm\big(\Phi^{-1}\big)=\beta_\mp(\Phi).
$$
Finally, we note that for any $\Phi\in\TSp(2m)$,
\begin{equation}
  \label{eq:beta+-beta-}
|\beta_+(\Phi)-\beta_-(\Phi)|=|b_+(\Phi)-b_-(\Phi)|\leq m,
\end{equation}
which is essential for what follows.

\subsubsection{Dynamical convexity}
\label{sec:index:DC}
The notion of dynamical convexity was originally introduced in
\cite{HWZ:convex} for Reeb flows on the standard contact $S^3$ and,
somewhat in passing, for higher-dimensional contact spheres. In
\cite{AM:elliptic, AM} the definition was extended to other contact
manifolds. In this section we mainly focus on the linear algebra
aspect of dynamical convexity and it is convenient for us to adopt the
following definition.

\begin{Definition}
  \label{def:DC} A path $\Phi\in \TSp(2m)$ is said to be
  \emph{dynamically convex} (DC) if $\mu_-(\Phi)\geq m+2$.
\end{Definition}

A word of warning is due: the flow generated by a positive definite
Hamiltonian need not be dynamically convex. This is the case for,
instance, for the time-one flow of a small Hamiltonian; see \cite[Ex.\
4.7]{GG:LS} for further discussion. The proof of the following result
is quite standard and can be found in, e.g., \cite{Lo, SZ} and also
\cite[Sec.\ 4.2]{GG:LS}.

\begin{Lemma}
\label{lemma:DC}
For any $\Phi \in \TSp(2m)$, we have
\begin{equation}
\label{eq:DC-it}
\mu_-\big(\Phi^{k+1}\big)\geq
\mu_-\big(\Phi^{k}\big)+\big(\mu_-(\Phi)-m\big)
\end{equation}
for all $k\in\N$. In particular,
$$
\mu_-\big(\Phi^k\big)\geq (\mu_-(\Phi)-m)k+m.
$$
Assume, furthermore, that $\Phi$ is dynamically convex.  Then the
function $\mu_-(\Phi^k)$ of $k\in\N$ is strictly increasing,
$$
\mu_-\big(\Phi^k\big)\geq 2k+m,
$$
and $\hmu(\Phi)\geq \mu_-(\Phi)-m\geq 2$. Thus
$\mu_-\big(\Phi^k\big)\geq m+2$ and all iterates $\Phi^k$ are also
dynamically convex.
\end{Lemma}

\subsubsection{Conley--Zehnder index: closed orbits}
\label{sec:CZ-orbits}
In this section we briefly recall how the definition of the
Conley--Zehnder index type invariants from Section \ref{sec:CZ-LA} for
elements of $\tSp(2m)$ translates to closed orbits of a Reeb
flow. Below all symplectic vector bundles are treated as Hermitian
vector bundles.

Now let $M=\p W$ be the boundary of a Liouville domain
$(W^{2n\geq 4},d\alpha)$ such that $c_1(TW)=0$ as in Section
\ref{sec:Liouville} and let $\xi$ be the contact structure
$\ker\alpha|_M$. Then $c_1(\xi)=0$ and hence the top complex wedge
power $L_M$ of $\xi$ is a trivial complex line bundle over
$M$. Likewise, the top complex wedge of $TW$ is a trivial line bundle
$L$. Fix a unit section $\fs$ of $L$. Note that the symplectic
orthogonal to $\xi$ in $T_MW$ is trivial and using, say, the Reeb
vector field $R$ as a section of this bundle, we have an isomorphism
$T_MW=\xi\oplus \C$ as bundles over $M$. Let $\fs_M$ be the unit
section of $L_M$ such that $\fs=\fs_M\otimes R$ over $M$. This section
is completely determined by $\fs$, up to homotopy.

Next, let $x\colon S_T^1:=\R/T\Z\to M$ be a closed orbit of the Reeb
flow, which we do not require to be prime. The pull-back
$\fs_M\circ x$ of $\fs_M$ by $x$ is a non-vanishing section of the
pull-back $x^*L_M$ over $S_T^1$ and $x^*L_M$ is in turn the top
complex wedge of $x^*\xi$. Pick any K\"ahler trivialization of
$x^*\xi$ which, up to homotopy, gives rise to the trivialization
$\fs_M\circ x$ of $x^*L$, i.e., such that its top complex wedge is
$\fs_M\circ x$. (Such a trivialization is unique up to homotopy.) Then
we use this trivialization to turn the linearized Reeb flow
$D\varphi^t$, $t\in [0,T]$, along $x$ into a path $\Phi(t)\in \Sp(2m)$
with $m=n-1$ starting at the identity, and hence an element
$\Phi\in\tSp(2m)$. By definition, $\mu_\pm(x):=\mu_\pm(\Phi)$ and
$\hmu(x):=\hmu(\Phi)$. In particular, when $x$ is non-degenerate,
$\mu(x):=\mu(\Phi)$. Furthermore, $\hmu$ is homogeneous, i.e.,
$$
\hmu\big( x^k\big)=k\hmu(x),
$$
and
$$
\hmu(x)-(n-1)\leq \mu_-(x)\leq \mu_+(x)\leq \hmu(x)+(n-1).
$$
by \eqref{eq:mu-del}. In a similar vein, we set
$\beta_\pm(x):=\beta_\pm(\Phi)$, etc.

A Reeb flow is said to be \emph{dynamically convex} if
$\mu_-(x)\geq n+1$ for all closed Reeb orbits $x$. By Lemma
\ref{lemma:DC}, this is equivalent to that $\mu_-(x)\geq n+1$ for all
prime orbits. This notion was introduced in \cite{HWZ:convex} where it
was shown that the Reeb flow on a convex hypersurface in $\R^{2n}$ is
dynamically convex.  We refer the reader to, e.g., \cite[Sec.\
4.2]{GG:LS} for a simple proof of this fact and to \cite{ALM,
  AM:elliptic, AM, DLLW1, GM, GGMa} for other relevant notions and
results.

\subsection{Index recurrence and applications}
\label{sec:IR-Thm}
The index recurrence theorem from \cite{GG:LS} asserts that
arbitrarily long intervals in the sequence of the Conley--Zehnder
indices of the iterates of one or a finite collection of symplectic
linear maps repeat themselves infinitely many times up to a
shift. That theorem is closely related and can be derived
  from various versions of the common index jump theorem \cite{DLW,
    Lo, LZ}, which is done in detail in \cite{DLWZ}. A more general
variant of the index recurrence theorem which we
prove here (using the approach from \cite{GG:LS})
allows for control of not only the indices but also additional real
parameters assigned to the maps, which should be thought of as periods
of closed orbits. In addition, the new variant shows that
the sequence of ``recurrence events'' can be made to have a bounded
gap which is crucial for the main results of this paper.

Thus we consider pairs $(\Phi,a)$ where the ``action'' $a$ is real and
positive, $\Phi\in \TSp(2m)$ and set $\Delta:=\hmu(\Phi)>0$.  Assume
that we have a finite collection of such pairs $(\Phi_i,a_i)$,
$i=1,\ldots, r$. Let us relabel and regroup these pairs putting two
pairs in the same group, which we will call a \emph{cluster}, if they
have the same ratio $a/\Delta$. To be more specific, we will denote
the pairs as $(\Phi_{ij},a_{ij})$, where $i=1,\ldots, i_0$ and
$j=1,\ldots, j_0 (i)$ with the index $i$ indicating the cluster, i.e.,
$i=i'$ if and only if $a_{ij}/\Delta_{ij}=a_{i'j'}/\Delta_{i'j'}$. We
will extend this division into clusters to ``iterated'' pairs
$(\Phi^k_{ij},ka_{ij})$, $k\in \N$, by the same rule. (In particular,
all iterates of a fixed pair are in the same cluster.) We call the
union $\CS$ of the sequences $a_{ij}\N$ the \emph{spectrum} of the
collection $(\Phi_{ij},a_{ij})$. In what follows, we denote by $[t]$
the closest integer to $t\in \R$ assuming that $t\not \in 1/2+\Z$.

The statement of our next result is admittedly notation heavy. It
might be helpful for the reader to keep in mind that in what follows
$l$ indexes recurrence events, $k$ with various subscripts and $\ell$
stand for the order of iteration, and $d_{il}$ is the index shift,
i.e., an integer approximating the mean indices of the iterates
$\Phi^{k_{ijl}}_{ij}$.

\begin{Theorem}[Index recurrence with action matching and bounded gap]
\label{thm:IR-A}
Let $(\Phi_{ij},a_{ij})$ be as above. Then for any $\eta>0$ and any
$\ell_0\in\N$, there exists a sequence $C_l\to\infty$ of real numbers
and integer sequences $d_{il}\to\infty$ and $k_{ijl}\to \infty$ (as
$l\to\infty$) such that for all $i$, $j$ and $l$, and all $\ell\in\Z$
in the range $1\leq |\ell|\leq \ell_0$, we have 
\begin{itemize}
\item[\reflb{IRA0}{\rm{(IR1)}}]
  $\big|\hmu\big(\Phi^{k_{ijl}}_{ij}\big)-d_{il}\big|<\eta$ and, in
  particular, $d_{il}=\big[\hmu\big(\Phi^{k_{ijl}}_{ij}\big)\big]$,
  and
  $$
  d_{il}-m
  \leq \mu_-\big(\Phi^{k_{ijl}}_{ij}\big)
  \leq \mu_+\big(\Phi^{k_{ijl}}_{ij}\big)
  \leq d_{il}+ m;
  $$
\item[\reflb{IRA+}{\rm{(IR2)}}]
  $\mu_\pm\big(\Phi^{k_{ijl}+\ell}_{ij}\big)= d_{il} +
  \mu_\pm\big(\Phi^\ell_{ij}\big)$  when $0<\ell\leq \ell_0$;

\item[\reflb{IRA-}{\rm{(IR3)}}]
  $\mu_+\big(\Phi^{k_{ijl}-\ell}_{ij}\big)= d_{il} -
  \mu_-\big(\Phi^\ell_{ij}\big)+ \big(\beta_+\big(\Phi^\ell_{ij}\big)-
  \beta_-\big(\Phi^\ell_{ij}\big)\big)$, where
  $0<\ell\leq \ell_0<k_{ijl}$ and $\beta_+-\beta_-=b_+-b_-\leq m$ with
  $\beta_\pm=0$ when $\Phi_{ij}^\ell$ is non-degenerate;

\item[\reflb{IRAsummand}{\rm{(IR4)}}] Assertions
  \ref{IRA0}--\ref{IRA-} continue to hold with the same values of
  $d_{il}$ and $k_{ijl}$ when $\Phi_{ij}$ is replaced by its
  non-degenerate part $\Psi$ with all invariants of $\Phi_{ij}$ and
  its iterates replaced by their counterparts for $\Psi$, i.e., $m$
  replaced by $m'$, $\mu_\pm\big(\Phi^{k_{ijl}}_{ij}\big)$ replaced by
  $\mu\big(\Psi^{k_{ijl}}\big)$, $\beta_\pm$ set to be 0, etc.;
  
\item[\reflb{IRAa}{\rm{(IR5)}}] $C_l-\eta< k_{ijl}a_{ij}< C_l$, and
  $ka_{ij}<C_l-\eta$ when $k<k_{ijl}$ and $ka_{ij}>C_l$ when
  $k>k_{ijl}$.
\end{itemize}
Moreover, we can ensure that the sequences $C_l$, $d_{il}$ and
$k_{ijl}$ are strictly increasing and have bounded gap.  In addition,
we can make all $k_{ijl}$ and $d_{il}$ divisible by any pre-assigned
integer and all $C_l$ lie outside the spectrum $\CS$.
\end{Theorem}

In what follows we will refer to the collection
$\big\{k_{ijl},d_{il},C_l\big\}$ for a fixed $l$ as a \emph{recurrence
  event}.

\begin{Remark}
  Condition \ref{IRAsummand} in Theorem \ref {thm:IR-A} and its
  corollaries can be generalized a bit further. Namely,
  \ref{IRA0}--\ref{IRA-} continue to hold with the same values of
  $d_{il}$ and $k_{ijl}$ when $\Phi_{ij}$ is replaced by any
  $\Psi\in\TSp(2m')$ such that $\Phi(1)=\Psi(1)\oplus A$, where
  $m'\leq m$ and $A\in \Sp\big(2(m-m')\big)$ and
  $\hmu(\Psi)=\hmu(\Phi)$. As in \ref{IRAsummand}, we also need to
  replace $m$ by $m'$, $\mu_\pm\big(\Phi^{k_{ijl}}_{ij}\big)$ by
  $\mu_\pm\big(\Psi^{k_{ijl}}\big)$,
  $\beta_\pm\big(\Phi^{k_{ijl}}_{ij}\big)$ by
  $\beta_\pm\big(\Psi^{k_{ijl}}\big)$, etc. This fact readily follows
  from the proof of the theorem.
\end{Remark}  

Let us now state some of the corollaries of Theorem \ref{thm:IR-A} and
show how this theorem is related to several previous results of the
same type.

\begin{Corollary}
  \label{cor:IR-A-DC}
  In the setting of Theorem \ref{thm:IR-A}, assume that all maps
  $\Phi_{ij}$ are dynamically convex and $\ell_0$ is taken large
  enough. (For instance, $2\ell_0\geq 3(m+1)$.) Then as a formal
  consequence of \ref{IRA+}, \ref{IRA-} and \ref{IRAa}, we have
  
\begin{itemize}

\item[\reflb{IRA+DC}{\rm{(IR2')}}]
  $\mu_-\big(\Phi^{k}_{ij}\big)\geq d_{il} + m+2$ and $ka_{ij}>C_l$
  when $k>k_{ijl}$;

\item[\reflb{IRA-DC}{\rm{(IR3')}}]
  $\mu_+\big(\Phi^{k}_{ij}\big)\leq d_{il} - 2$ and $ka_{ij}<C_l-\eta$
  when $k<k_{ijl}$, and $\mu\big(\Phi^{k}_{ij}\big)\leq d_{il} - m-2$
  when in addition $\Phi_{ij}$ is strongly non-degenerate.

\end{itemize}
In particular, for any $i$, none of the intervals
$\big[ \mu_-\big(\Phi^k_{ij}\big), \mu_+\big(\Phi^k_{ij}\big)\big]$,
$k\in\N$, contains the point $d_{il}+m+1$. When $\Phi_{ij}$ is in
addition strongly non-degenerate this is also true for
$d_{il}-m-1$. Furthermore, $k a_{ij}\in (C_l-\eta, C_l)$ only when
$k=k_{ij}$.
\end{Corollary}

\begin{proof}
  Only the index part in \ref{IRA+DC} and \ref{IRA-DC} requires a
  proof. The ``action'' part follows directly from \ref{IRAa}. By
  \ref{IRA0}, the ``In particular'' part of the corollary is
  essentially a consequence of \ref{IRA+DC} and \ref{IRA-DC}.

  Recall from Lemma \ref{lemma:DC} that $\hmu(\Phi_{ij})\geq 2>0$ due
  to the dynamical convexity condition. Fix $\ell_0$ so that
$$
\ell_0\min_{ij}\hmu(\Phi_{ij})\geq 2\ell_0\geq 3(m+1).
$$
Then, when $k> k_{ijl}+\ell_0$, by the triangle inequality and
\eqref{eq:mu-del},
\begin{equation*}
  \begin{split}
    \mu_-(\Phi^{k}_{ij})-d_{il} & \geq
    \big(\hmu(\Phi^{k}_{ij})-\hmu(\Phi^{k_{ijl}}_{ij})\big)
    -\big|\mu_-(\Phi^{k}_{ij})- \hmu(\Phi^{k}_{ij}) \big|
    - \big|\hmu(\Phi^{k_{ijl}}_{ij})-d_{il}\big|\\
    &\geq
    (k-k_{ijl})\hmu(\Phi_{ij})-m-\eta\\
    &\geq
    2(k-k_{ijl})-m-1\\
    &\geq 2m+2.
\end{split}
\end{equation*}
When $k=k_{ijl}+\ell$ with $1\leq \ell\leq \ell_0$, we have 
$$
\mu_-\big(\Phi^{k}_{ij}\big)= d_{il} +
\mu_-\big(\Phi^{\ell}_{ij}\big)\geq d_{il} + m+2
$$
by \ref{IRA+} and the dynamical convexity assumption. This proves
\ref{IRA+DC}.

Condition \ref{IRA-DC} is established by a similar
calculation. Namely, for $k< k_{ijl}-\ell_0$,
\begin{equation*}
  \begin{split}
    \mu_+(\Phi^{k}_{ij})- d_{il} & \leq
    \big(\hmu(\Phi^{k}_{ij})-\hmu(\Phi^{k_{ijl}}_{ij})\big)
    +\big|\mu_+(\Phi^{k}_{ij})-\hmu(\Phi^{k}_{ij}) \big|
    +\big|\hmu(\Phi^{k_{ijl}}_{ij})-d_{il}\big|\\
    &\leq
    (k-k_{ijl})\hmu(\Phi_{ij})+m+\eta\\
    &\leq
    2(k-k_{ijl})+m+1\\
    &\leq -2m-2.
\end{split}
\end{equation*}
When $k=k_{ijl}-\ell$ with $1\leq \ell\leq \ell_0$, we have
\begin{equation*}
  \begin{split}
    \mu_+\big(\Phi^{k}_{ij}\big) & = d_{il}
    -\mu_-\big(\Phi^{\ell}_{ij}\big) +
    \big(\beta_+\big(\Phi^\ell_{ij}\big)
    -\beta_-\big(\Phi^\ell_{ij}\big) \big) \\
    & \le
    d_{il} -m-2+m \\ & = d_{il} -2    
\end{split}
\end{equation*}
by \ref{IRA-} and again the dynamical convexity assumption, proving
\ref{IRA-DC}. When $\Phi_{ij}$ is strongly non-degenerate, the term
$\beta_+-\beta_-$ disappears and we obtain the upper bound
$d_{il}-m-2$.
\end{proof}

Setting $a_{ij}:=\hmu(\Phi_{ij})$, we force all maps $\Phi_{ij}$ to be
in one cluster, and then we can suppress $i$ in the notation, labeling
the maps as $\Phi_j$. In particular, $d_{il}$ becomes independent of
$i$. Thus, the above results immediately yield the following.

\begin{Corollary}
  \label{cor:IR}
  In the setting of Theorem \ref{thm:IR-A} or Corollary
  \ref{cor:IR-A-DC}, for a collection of maps $\Phi_j\in\TSp(2m)$, we
  can find bounded gap sequences $C_l$ and $d_l$ independent of $j$,
  and bounded gap sequences $k_{jl}$ such that all other conditions
  are satisfied. In particular, the sequences $k_{jl}$ and $d_l$ can
  be made divisible by any fixed integer. Furthermore, with
  $\Delta_j:=\hmu(\Phi_j)$ condition \ref{IRAa} takes the following
  form: $C_l-\eta< k_{jl}\Delta_{j}< C_l$, and $k\Delta_{j}<C_l-\eta$
  when $k<k_{jl}$ and $k\Delta_{j}>C_l$ when $k>k_{jl}$.
\end{Corollary}

Thus, in this case, a recurrence event takes a simpler form
$\big\{k_{jl},d_{l},C_l\big\}$. In fact, the sequence $C_l$ is
somewhat redundant now and carries essentially the same information as
the sequence $d_l:= [ k_{jl}\Delta_{j}]$.

The new point of Corollary \ref{cor:IR} compared to the index
  recurrence theorem from \cite[Thm.\ 5.2]{GG:LS} is the assertion
that all sequences involved have bounded gap. A simple proof from
\cite{GG:LS} of the following multiplicity result, quite standard now,
is an illustrative application of the index recurrence theorem in the
form of Corollaries \ref{cor:IR-A-DC} and \ref{cor:IR}. Interestingly,
the argument relies on the existence of just one recurrence event.

\begin{Corollary}[Multiplicity; \cite{LZ} and \cite{GK}]
    \label{cor:multiplicity}
    Assume that the Reeb flow on the boundary $M$ of a star-shaped
    domain $W\subset \R^{2n}$ is non-degenerate and dynamically
    convex. Then the flow has at least $n$ prime closed orbits. If, in
    addition, the flow is a pseudo-rotation, there are $n$
    non-alternating prime closed orbits with index of the same parity
    as $n+1$.
\end{Corollary}

\begin{proof}
  Without loss of generality, we may assume that there are only
  finitely many prime closed orbits $x_i$ in $M$. Denote the
  linearized flows along these orbits by $\Phi_1,\ldots,
  \Phi_r$. Consider one recurrence event from Corollary \ref{cor:IR}
  for these elements of $\Sp(2(n-1))$. Suppressing $l$ in the notation
  and ignoring $C_l$, we denote this event by $\{d,k_1,\ldots,k_r\}$
  and require $d$ and all $k_j$ to be divisible by two. Since
  $\COH_{n-1+2i}(W;\Q)\neq 0$, $i\in \N$, there must be at least one
  closed orbit with index $n-1+2i$. There are exactly $n$ such integer
  points in the range $\CI:=[d-n+1,d+n-1]$. By Corollary
  \ref{cor:IR-A-DC}, the only closed Reeb orbits with index in $\CI$
  are $x_j^{k_j}$. Moreover, since $k_j$ is even, to contribute to the
  symplectic homology, $x_j$ must be non-alternating and hence, the
  parity of $\mu(x_j)$ must be the same as that of
  $\mu\big(x_j^{k_j}\big)$, i.e., $n+1$. Therefore, there are at least
  $n$ non-alternating prime closed orbits with parity $n+1$.
  \end{proof}

A more subtle application allows one to remove the non-degeneracy
condition when $W$ is centrally symmetric and convex.

\begin{Corollary}[Multiplicity in the convex, symmetric case;
  \cite{LLZ}]
    \label{cor:multiplicity-sym}
    Assume that $W$ is centrally symmetric and strictly convex. Then
    the Reeb flow on $M=\p W$ has at least $n$ prime closed orbits.
\end{Corollary}

\begin{proof}[Outline of the proof following \cite{GM}] As in the
  proof of Corollary \ref{cor:multiplicity}, assume that the flow has
  only finitely many prime closed orbits. Keeping the notation from
  that proof, denote these orbits by $x_j$. The key observation is
  that with our conventions
  $\mu_-\big(x_i^k\big)-\beta_+\big(x_i^k\big)+\beta_-\big(x_i^k\big)\geq
  n+1$ for all $k\in\N$ as a consequence of symmetry and convexity;
  see \cite[Sec.\ 4]{GM}, \cite[Lemma 4.1]{LLZ} and \cite[Lemma
  15.6.3]{Lo}. (This fact is formalized in \cite{GM} as a
  symplectically invariant notion of \emph{strong dynamical
    convexity}.) Considering an index recurrence event, we now see
  that $\supp\COH\big(x_j^k;\Q\big)$ does not overlap with the
  interval $\CI$ unless $k=k_j$ which we can assume to be even. Using
  the $S^1$-equivariant spectral invariants over $\Q$ as in
  \cite{GG:LS}, we conclude that among the prime orbits $x_i$ there
  must be at least $n$ orbits to fill in all $n$ spots in
  $\CI$. Alternatively, we can use the fact that, by Corollary
  \ref{cor:SH}, $\supp\COH\big(x_j^k;\Q\big)$ contains at most one
  point for all $k\in\N$.
  \end{proof}

\begin{Remark}  
  As we have already mentioned in the introduction, Corollary
  \ref{cor:multiplicity} had been generalized and strengthened in
  several directions. When the non-degeneracy condition is removed, a
  proof similar to the above one but relying on equivariant spectral
  invariants gives a lower bound $\lfloor n/2\rfloor +1$. The argument
  can be further pushed (in a non-trivial way) to improve this lower
  bound to $\lceil n/2\rceil +1$; see \cite{Wa2} and also \cite{DL,
    GG:LS}. In dimensions 6 and 8, the lower bound is 3 and,
  respectively 4; see \cite{DL, Wa1, WHL}.) Furthermore, for
  non-degenerate flows, the dynamical convexity condition has been
  replaced by a much less restrictive index requirement, \cite{DLLW1,
    DLLW2, GGMa}. Corollary \ref{cor:multiplicity-sym} has also been
  refined and generalized in a variety of ways; see, e.g., \cite{AM,
    ALM, GM} and references therein.
 \end{Remark} 

Furthermore, arguing as in the proof of Corollary
 \ref{cor:multiplicity} and using Corollary \ref{cor:SH} and Theorem
 \ref{thm:Orbits}, we obtain the following.

\begin{Corollary}[HZ-conjecture]
    \label{cor:mult+HZ}
    Assume that the Reeb flow on the boundary of a star-shaped domain
    in $\R^{2n}$ is a dynamically convex pseudo-rotation and all
    closed orbits are $\Q$-visible.  Then the flow has exactly $n$
    prime closed orbits.
\end{Corollary}

\section{Proofs of the main results}
\label{sec:pfs}

Here we first derive Theorems \ref{thm:mult} and \ref{thm:HZ} from
Corollary \ref{cor:SH} and Theorem \ref{thm:Orbits} in Section
\ref{sec:pf-main1}. Then, in Section \ref{sec:pf-main2}, we establish
the latter as a consequence of Theorem \ref{thm:SH}. In Section
\ref{sec:pf-main3}, we turn to the proofs of Theorem \ref{thm:SH},
which is the central part of the argument, and Corollary \ref{cor:SH}.
Finally, Corollary \ref{cor:ellipsoids} is proved in Section
\ref{sec:ellipsoids-pf}

\subsection{Proofs of Theorems \ref{thm:mult} and \ref{thm:HZ}}
\label{sec:pf-main1} 
The proof of Theorem \ref{thm:mult} comprises two unequal parts with
most of the work done in the proof of the first part. Theorem
\ref{thm:HZ} is proved in Section \ref{sec:HZ-pf}.

\subsubsection{Proof of Theorem \ref{thm:mult}}
Our proof of the theorem actually establishes a more precise result
asserting the existence of exactly $n$ prime orbits satisfying certain
visibility conditions. Although this result is more technical, we
state and prove it here, for it sheds more light on the proof and is
of independent interest.

Let $y$ be a prime orbit and let $p(y)$ be the least common multiple
of the degrees of the roots of unity among the elliptic eigenvalues of
the return map of $y$. Thus $y^{p(y)}$ is a ``maximally degenerate''
iterate of $y$. Note that $y^{p(y)}$ can fail to be $\Q$-visible
even when $y$ is elliptic and non-degenerate; see Example
\ref{exam:iterates}.  Set $N(y):=2p(y)$.  We will call $y$
\emph{virtually $\Q$-visible} when $y^{N(y)}$ is $\Q$-visible. For
instance, a strongly non-degenerate prime orbit is virtually
$\Q$-visible if and only if it is non-alternating.  A totally
degenerate $\Q$-visible prime orbit is virtually $\Q$-visible. The
orbit from Example \ref{exam:iterates} is not. As of writing, it is
not known if a virtually $\Q$-visible prime orbit must be $\Q$-visible
even in the setting of Theorem \ref{thm:mult}.

\begin{Theorem}
  \label{thm:mult-refined}
  Assume that the Reeb flow on the boundary $M^{2n-1}\subset \R^{2n}$
  of a star-shaped domain is dynamically convex and has finitely many
  prime closed orbits. Then the flow has exactly $n$ virtually
  $\Q$-visible prime orbits. Moreover, these orbits $y$ have the same
  ratio $\CA(y)/\hmu(y)$.
\end{Theorem}

\begin{Remark}
  The $n$ virtually $\Q$-visible closed orbits from Theorem
  \ref{thm:mult-refined} have an additional property that
  $\hmu(y)/\hmu(y')\not\in \Q$ for any two such distinct orbits $y$
  and $y'$. Indeed, assume that this ratio is rational:
  $\hmu(y)/\hmu(y')=p/q$ for some positive integers $p$ and $q$. Let
  $m$ be a common multiple of $N(y)$ and $N(y')$. Then, by
  \eqref{eq:virst-vis} below, $y^{qm}$ and $(y')^{pm}$ are both
  $\Q$-visible. The iterates $y^{qm}$ and $(y')^{pm}$ have the same
  action (aka period), which contradicts Corollary \ref{cor:SH}.  This
  observation, added after we have learned about \cite{Wa:new},
  sharpens the main result from that paper, which in turn is a
  refinement of a theorem from \cite{HO}.
\end{Remark}

Part \ref{Mult} of Theorem \ref{thm:mult} immediately follows from
Theorem \ref{thm:mult-refined}. To prove \ref{Mult-nd}, recall that in
the non-degenerate case a prime orbit is virtually $\Q$-visible if and
only if it is non-alternating. Hence there are exactly $n$ such
orbits.  The parity statement is a consequence Part \ref{O-index} from
Theorem \ref{thm:Orbits} since now the degree is equal to the
Conley--Zehnder index. As we have already mentioned, the lower bound
in \ref{Mult-nd} is not new; cf.\ Corollary \ref{cor:multiplicity} and
\cite{GG:LS, GK, LZ}.

\begin {proof}[Proof of Theorem \ref{thm:mult-refined}]
  The ``Moreover'' part follows immediately from
    \ref{O-clusters} in Theorem \ref{thm:Orbits}.  We begin with some
  general remarks on virtually $\Q$-visible orbits.  Then we turn to
  the multiplicity statement which comprises the main part and the
  bulk of the proof.

  Let $y$ be a closed orbit. Then, in the setting of the theorem,
  \begin{equation}
    \label{eq:virst-vis}
    y^{kN(y)} \textrm{
      is $\Q$-visible for some $k\in \N$ iff } y^{kN(y)} \textrm{ is $\Q$-visible for all $k\in \N$}.
  \end{equation}
  In other words, $y$ is virtually $\Q$-visible if and only if there
  exists $k\in\N$ such that $y^{kN(y)}$ is $\Q$-visible, and if and
  only if $y^{kN(y)}$ is $\Q$-visible for all $k\in\N$.  (In
particular, again within the theorem's framework, for any orbit $z$
which is not virtually $\Q$-visible, no iterate $z^{kN(y)}$,
$k\in \N$, is $\Q$-visible.)

  Indeed, $y^{kN(y)}$ is an admissible iterate of $y^{N(y)}$ and
  the return map of $y^{N(y)}$ has no real negative eigenvalues by the
  definition of $N(y)$. Hence,
  $$
  \chieq\big(y^{kN(y)}\big)=\chieq\big(y^{N(y)}\big)
  $$
  for all $k\in\N$ by Lemma \ref{lemma:chieq-it}. Under the conditions
  of the theorem, $\chieq(x)$ is either $(-1)^{n+1}$ or 0 depending on whether
  $x$ is $\Q$-visible or not, as follows from Corollary
  \ref{cor:SH}. Hence, $y^{N(y)}$ and $y^{kN(y)}$ are both either
  $\Q$-visible or $\Q$-invisible.

  The strategy of the proof of the multiplicity part is to use the
  index recurrence theorem to identify an interval
  $\CI:=[d-n+1, d+n-1]\subset \N$ such that for each eventually
  $\Q$-visible prime orbit $x_j$ only one iterate $x_j^{k_j}$ can have
  degree in $\CI$. Moreover, we ensure that $x_j^{k_j}$ is
  $\Q$-visible if and only if $x_j$ is virtually $\Q$-visible and thus
  there are exactly $n$ such orbits. The main difficulty lies in
  showing that $\deg(x_j^{k<k_j})\not\in\CI$ and this is where we
  crucially rely on Corollary \ref{cor:SH} and Theorem
  \ref{thm:Orbits}.  We prove this first for the least action orbit
  $x$ by analyzing the behavior of the sequence $\deg(x^k)$. Dealing
  with the iterates of the remaining prime orbits $x_j$, we show that
  $x_j^{k}$ must be $\Q$-invisible whenever $\deg(x_j^k)\in \CI$ and
  $k<k_j$, by comparing the degree and the action of $x_j^k$ with the
  degree and the action of the iterates of $x$. It is convenient to
  break down the proof into three steps.

\emph{Step 1: Setup.} As in Theorem \ref{thm:Orbits}, denote by $\PP$
the set of all $\Q$-visible closed orbits. Let $x\in\PP$ be the orbit
with the least action. Then $\deg(x)=n+1$ since the ordering of $\PP$
by the action agrees with the ordering by the degree due to Part
\ref{O-index} of that theorem. As a consequence, $\mu_-(x)\leq
n+1$. We claim that $x$ is prime. Indeed, if $x=z^k$ with $k\geq 2$,
by Lemma \ref{lemma:DC} and dynamical convexity we have
  $$
  \mu_-(x)=\mu_-\big( z^k\big)>n+1,
  $$
  which is impossible.

  \begin{Remark}
    When $M$ is convex, $x$ is, in fact, a least period closed orbit
    on $M$, among all (not necessarily $\Q$-visible) closed orbits. In
    other words, at least one of the least period closed orbits is
    necessarily $\Q$-visible. This is a consequence of the fact that
    for convex hypersurfaces the (non-equivariant) homological
    capacity is equal to the least period, \cite{AbK, Ir}.
  \end{Remark}
  
  Next, let $x_0:=x,\,x_1,\ldots, x_r$ be all prime closed orbits
  which are \emph{eventually $\Q$-visible}, i.e., $x_j^k\in\PP$ for
  some $k\in\N$.  Note that the orbits $x_{j\geq 1}$ might be
  $\Q$-invisible and/or have action below $\CA(x)$. In fact, $x_j^s$
  with $j\geq 1$ is necessarily $\Q$-invisible when
  $\CA\big(x_j^s\big)\leq \CA(x)$. For each $j\geq 1$, we denote by
  $s_j$ the largest such $s$, setting $s_j=0$ if $\CA(x_j)>\CA(x)$.
  Our goal is to show that there are exactly $n$ virtually
  $\Q$-visible orbits among the orbits $x_j$.

  By \ref{O-clusters} and homogeneity of $\CA$ and $\hmu$, the
  orbits $x_j$, $j\geq 0$, form one cluster, i.e., have the same ratio
  $\CA(x_j)/\hmu(x_j)$. Without loss of generality, we may assume that
  this ratio is 1, i.e.,
  \begin{equation}
    \label{eq:action=hmu}
    \CA(y)=\hmu(y)
  \end{equation}
  for all $y\in \PP$.
  
  Consider an index recurrence event as in Corollary \ref{cor:IR},
  where we require $\eta>0$ to be sufficiently small, all $d_{l}$ to
  be even and $k_{jl}$ divisible by $N\in\N$ specified as follows. Let
  $p$ be the least common multiple of the degrees of the roots of
  unity among the elliptic eigenvalues of all closed orbits
  $x_j$. Then
  \begin{equation}
    \label{eq:ThmA-N}
  N=2p\prod_{j=1}^r s_j!
\end{equation}
where, as usual, we set $0!=1$. (Since there is only one cluster we
can suppress $i$ in the notation.) Furthermore, we will also drop $l$
from the notation, concentrating on one event. Thus $d=d_{l}$ and
$k_j=k_{jl}$, etc.

We have three groups (not as algebra objects) of $\Q$-visible closed
orbits $x_j^k$:
  \begin{itemize}
  \item \emph{Group $\Gamma_+$} formed by the $\Q$-visible orbits
    $x_j^{k>k_j}$. Then $\mu_-\big(x_j^k\big)\geq d+n+1$, and hence
    $\deg \big(x_j^k\big)\geq d+n+1$.

  \item \emph{Group $\Gamma_0$} comprising the $\Q$-visible orbits
    $x_j^{k_j}$. Then
  $$
  d-n+1\leq \mu_-\big(x_j^{k_j}\big)\leq \mu_+\big(x_j^{k_j}\big)\leq
  d+n-1,
  $$
  and hence
  $$
  \deg\big(x_j^{k_j}\big)\in \CI:=[d-n+1, d+n-1]\subset \N.
  $$

\item \emph{Group $\Gamma_-$} formed by the $\Q$-visible orbits
  $x_j^{k<k_j}$. Then $\mu_+\big(x_j^k\big)\leq d-2$, and hence
  $\deg \big(x_j^k\big)\leq d-2$. 

\end{itemize}

It is enough to show that the orbits from $\Gamma_-$ cannot have
degree in $\CI$, i.e.,
\begin{equation}
\label{eq:spot}
\deg \big(x_j^k\big)\leq d-n\textrm{ when } k<k_j \textrm{ and
  $x_j^k\in\PP$}.
\end{equation}
Indeed, then the only closed orbits with degree in $\CI$ must be from
$\Gamma_0$. By \eqref{eq:virst-vis}, since $N$ is divisible
by all $N(x_j)$, an orbit $x_j^{k_j}$ is $\Q$-visible if and only if
$x_j$ is virtually $\Q$-visible.  There are exactly $n$ integer spots
in $\CI$ of parity $n+1$, which must all be filled. Different
$\Q$-visible orbits have different degrees by Theorem
\ref{thm:Orbits}. Therefore, there are exactly $n$ virtually
$\Q$-visible prime orbits.

Up to this point our reasoning has followed closely the proof of
multiplicity results in \cite[Sec.\ 6]{GG:LS}. The rest of the
argument is new and different and critically relies on Theorem
\ref{thm:Orbits}, Corollary \ref{cor:SH} and Lemma
\ref{lemma:chieq-it} precisely describing the behavior of the
equivariant Euler characteristic $\chieq$ under admissible iterations.

\emph{Step 2: The orbit $x_0=x$.}  To establish \eqref{eq:spot}, we
start with $j=0$, i.e., $x_0=x$, and $k=k_0-1$. Then $k$ is admissible
and odd due to our choice of $N$. Thus
$$
\chieq(x^k)=\chi(x)=(-1)^{n+1}\neq 0.
$$
Here in the first equality we used Lemma \ref{lemma:chieq-it}. The
second equality follows from the fact that $\COH(x;\Q)$ is
concentrated in degree $n+1$ by Corollary \ref{cor:SH}. As a
consequence, $x^k\in\PP$. We will need the following observation.

\begin{Lemma}
  \label{lemma:SH-it}
Let $x$ and $k$ be as above. 
\begin{itemize}
\item[\reflb{LIE1}{\rm{(D1)}}] Then
  $$
  \deg(x^k)=\deg(x)+ (k-1)\hmu(x)
  $$
  if $x$ is totally degenerate. (In fact, this is true for all
  $k\in\N$.)
  
\item[\reflb{LIE2}{\rm{(D2)}}] When $x$ is not totally degenerate, let
  $\Psi$ be the non-degenerate part of $x$. Then
    $$
  \deg(x^k)=\deg(x)+ \big[\mu\big(\Psi^k\big)-\mu(\Psi)\big].
  $$
  \end{itemize}
\end{Lemma}

\begin{proof} We already know that $\COH(x;\Q)$ and
  $\COH\big(x^k;\Q\big)$ are supported in one degree and
  one-dimensional. Hence we only need to determine the degree shift
  $L$.  Let $\varphi$ be the Poincar\'e return map of $x$. Then the
  return map of $x^l$ is $\varphi^l$ for all $l\in\N$. By
  \eqref{eq:SH-CH} and \eqref{eq:McL},
  $$
  \dim \COH\big(x^l;\Q\big)\leq \dim \HF\big(\varphi^l;\Q\big),
  $$
  again for all $l\in\N$, and 
  $$
  \dim \COH\big(x^k;\Q\big)=1
  $$
  when $k$ is admissible. Moreover, $\COH(x;\Q)=\HF(\varphi;\Q)$ due
to \cite[Thm.\ 1.2]{Fe}; see also \cite{Fe:thesis}. Furthermore, as
readily follows from \cite[Thm.\ 1.1]{GG:gap} (and its proof),
  $$
  \HF\big(\varphi^k;\Q\big)=\HF(\varphi;\Q)[-L]=\COH(x;\Q)[-L]
  $$
  since $k$ is admissible. Here the shift $L$ is exactly as in
  \ref{LIE1} and \ref{LIE2} depending on whether $x$ is totally
  degenerate or not. (Indeed, in the former case
  $\HF\big(\varphi^k;\Q\big)=\HF(\psi\varphi;\Q)$ for a loop $\psi$ of
  local Hamiltonian diffeomorphisms with $\hmu(\psi)=(k-1)\hmu(x)$ by
  \cite[Claim 4.1]{GG:gap}. Hence,
  $\HF\big(\varphi^k;\Q\big)=\HF(\varphi;\Q)[-(k-1)\hmu(x)]$ by
  \cite[(LF6)]{GG:gap}.  The latter case follows from \cite[Sect.\
  4.5]{GG:gap}.) Combining these facts we conclude that
  $\HF\big(\varphi^k;\Q\big)$ is one-dimensional and supported in
  $\deg(x^k)$, and thus $\deg(x^k)=\deg(x)+L$.
  \end{proof}

  \begin{Remark}
    \label{rmk:iterates2}
    Lemma \ref{lemma:SH-it} is a particular case of a much more
    general hypothetical statement. Namely, we conjecture that for any
    isolated closed Reeb orbit $x$,
  $$
  \COH\big(x^k;\Q\big)=\COH(x;\Q)[(1-k)\hmu(x)]
  $$
  if $x$ is totally degenerate.  When $x$ is not totally degenerate,
  let as above $\Psi$ be the non-degenerate part of $x$. Then, at
  least if $x$ is prime,
    $$
    \COH\big(x^k;\Q\big)=
    \COH(x;\Q)\big[\mu(\Psi)-\mu\big(\Psi^k\big)\big]
  $$
  where $k$ is admissible and odd. For non-equivariant symplectic
  homology we expect similar statements to be true for any field. A
  variant of this conjecture is proved in \cite[Thm.s\ 1.10 and
  1.11]{HHM} with a different definition of the local homology, which
  is hypothetically equivalent to the one used here.
  \end{Remark}

  We are now in a position to return to the proof of \eqref{eq:spot}
  for $x_0=x$ and $k=k_0-1$. We will prove the strict inequality
\begin{equation}
\label{eq:spot2}
  \deg \big(x^{k_0-1}\big) < d-n.
\end{equation}
There are two cases to consider depending on whether $x$ is totally
degenerate or not.

Assume first that $x$ is totally degenerate. In this case $x^k\in\PP$
for all $k\in\N$.  Let us define $q$ by the condition
$$
n+1=\deg(x)=\hmu(x)+q.
$$
Note that by dynamical convexity and \eqref{eq:bpm} with $\Psi$ absent
we have
$$
n+1\leq \mu_-(x)= \hmu(x)-\beta_-(x), \textrm{ where }
\beta_-(x)\geq 0.
$$
Thus,
\begin{equation}
  \label{eq:q-and-hmu0}
q\leq 0 \quad \textrm{and} \quad \hmu(x)\geq n+1.
\end{equation}

It follows from \ref{LIE1} of Lemma \ref{lemma:SH-it} that
$$
\deg\big(x^k\big)=k\hmu(x)+q.
$$
(In fact this is true for all $k\in\N$ since all $k$ are admissible in
this case.) Applying this to $k=k_0-1$ and using \eqref{eq:q-and-hmu0}
in the first inequality and then \ref{IRA0} from Theorem
\ref{thm:IR-A} in the penultimate one, we have
\begin{equation*}
  \begin{split}
    \deg\big(x^{k_0-1}\big)&=(k_0-1)\hmu(x)+q\\
    &= \hmu\big(x^{k_0}\big)-\hmu(x)+q\\
    &\leq \hmu\big(x^{k_0}\big)-(n+1)\\
    &\leq d+\eta-(n+1)\\
    &< d-n.
\end{split}
\end{equation*}
In the last inequality, we use the fact that $\eta$ is small and, in
particular, $\eta<1$.  (In fact, in this case
$d=\hmu\big(x^{k_0}\big)$ because $x$ is totally degenerate, and hence
$\hmu\big(x^{k_0}\big)\in\Z$.) Thus $\deg\big(x^{k}\big)< d-n$ proving
\eqref{eq:spot2} when $x$ is totally degenerate.

Next, assume that $x$ has a non-trivial non-degenerate part which we
denote by $\Psi$ as above. The argument is quite similar to the
totally degenerate case with $\hmu(x)$ replaced by
$\mu(\Psi)$. Namely, note that $k:=k_0-1$ is admissible and odd, due
to our divisibility assumptions on $k_0$; see
\eqref{eq:ThmA-N}. Define $q$ by the condition
$$
n+1=\deg(x)=\mu(\Psi)+q.
$$
By dynamical convexity and \eqref{eq:bpm} or more
specifically \eqref{eq:bpm2},
$$
n+1\leq \mu_-(x)=\mu(\Psi)-\beta_-(x) \textrm{ where } \beta_-(x)\geq
0.
$$
It follows that
$$
n+1=\mu(\Psi)+q\leq \mu(\Psi)-\beta_-(x) \textrm{ with }
\beta_-(x)\geq 0,
$$
and hence
\begin{equation}
  \label{eq:q-and-hmu}
q\leq 0 \quad \text{and} \quad \mu(\Psi)\geq n+1.
\end{equation}
Now by \ref{LIE2} with $k=k_0-1$, we have
$$
\deg\big(x^{k}\big)=\mu(\Psi)+q+\mu\big(\Psi^{k}\big)-\mu(\Psi)
=\mu\big(\Psi^{k}\big)+q.
$$
Applying Parts \ref{IRA-} and \ref{IRAsummand} of Theorem
\ref{thm:IR-A} to $\Psi^{k_0-1}=\Psi^k$, we have
$$
\mu\big(\Psi^{k}\big)=d-\mu(\Psi),
$$
Putting these two facts together and using \eqref{eq:q-and-hmu} we
obtain that
$$
\deg\big(x^{k}\big)=d-\mu(\Psi)+q\leq d-n-1.
$$
Therefore, $\deg\big(x^{k}\big)< d-n$, i.e., \eqref{eq:spot2} holds
when the non-degenerate part of $x$ is non-trivial.

When $j=0$, i.e., $x_0=x$, and
$k=k_0-\ell$ with $\ell>1$, condition \eqref{eq:spot} is an easy
consequence of \eqref{eq:spot2}. Indeed, then
$$
\CA\big(x^k\big)<\CA\big(x^{k_0-1}\big).
$$
When in addition $x^k\in\PP$, i.e., $x^k$ is $\Q$-visible, we must
have
$$
\deg\big(x^k\big)<\deg\big(x^{k_0-1}\big)<d-n
$$
by Part \ref{O-index} of Theorem \ref{thm:Orbits} and
\eqref{eq:spot2}. This proves \eqref{eq:spot} for $j=0$.

\emph{Step 3: The orbits $x_{j\geq 1}$.}  Next, let us focus on
\eqref{eq:spot} for $j\geq 1$. Fixing $j$, assume that \eqref{eq:spot}
fails:
\[
  x_j^{k_j-s}\in \PP\textrm{ and }
  \deg\big(x_j^{k_j-s}\big)>d-n\textrm{ for some } s>0.
\]
Then we also have
$$
\deg\big(x_j^{k_j-s}\big) > \deg\big(x^{k_0-1}\big),
$$
and hence
\begin{equation}
  \label{eq:action-xj-x0}
\CA\big(x_j^{k_j-s}\big) > \CA\big(x^{k_0-1}\big)
\end{equation}
as a consequence again of \ref{O-index}.

By \eqref{eq:action=hmu}, the action is equal to the mean index. By
Part \ref{IRA0} of Theorem \ref{thm:IR-A},
$$
\big|\CA\big(x_j^{k_j}\big) - \CA\big(x^{k_0}\big)\big|<2\eta.
$$
Therefore, provided that $\eta>0$ is sufficiently small,
\eqref{eq:action-xj-x0} is equivalent to the condition that
\begin{equation}
  \label{eq:action-xj-x0-2}
\CA\big(x_j^{s}\big) < \CA(x).
\end{equation}
To summarize, to complete the proof we need to show that
$x_j^{k_j-s}\not\in\PP$, i.e., $x_j^{k_j-s}$ is $\Q$-invisible,
whenever \eqref{eq:action-xj-x0-2} is satisfied.

Note that then $s\leq s_j$ and hence $k_j$ is divisible by $s$ due to
the divisibility requirement, \eqref{eq:ThmA-N}. Furthermore, $x_j^s$
is $\Q$-invisible and $\chieq\big(x_j^s\big)=0$. Set $k'_j=k_j/s$. We
have
$$
x_j^{k_j-s}=\big(x_j^s\big)^{k'_j-1}.
$$
This guarantees, again by \eqref{eq:ThmA-N}, that $x_j^{k_j-s}$ is an
admissible odd iterate of $x_j^s$. Finally, by Lemma
\ref{lemma:chieq-it},
$$
\chieq\big(x_j^{k_j-s}\big)=\chieq\big(x_j^{s}\big)=0,
$$
and $x_j^{k_j-s}$ is $\Q$-invisible due to Corollary \ref{cor:SH}.

This completes the proof of \eqref{eq:spot} and Theorem
\ref{thm:mult-refined}.
\end{proof}

\subsubsection{Proof of Theorem \ref{thm:HZ}}
\label{sec:HZ-pf}
The idea of the proof is that symmetry and Corollary \ref{cor:SH} rule
out the existence of asymmetric orbits of the flow. Furthermore, prime
symmetric orbits are necessarily non-alternating and the result
follows from Part \ref{Mult-nd}.

Here is a detailed argument. Recall that when $M$ is centrally
symmetric, there are two types of closed Reeb orbits in $M$: symmetric
and asymmetric. The former are the closed orbits $x$ such that
$-x(t)=x(t+T/2)$ where $T=\CA(x)$ is the period of $x$, and the latter
are the orbits $x$ such that $-x$ and $x$ are geometrically
distinct. Clearly, asymmetric orbits come in pairs, $(x,-x)$, with
both $x$ and $-x$ having the same action and the same index. Thus, by
Corollary \ref{cor:SH}, the flow has no asymmetric orbits.

To finish the proof it suffices to show that all symmetric closed
orbits are non-alternating. This is a standard observation going back
to \cite{LLZ}; below we closely follow \cite{GM}.

Consider the homogeneous of degree two Hamiltonian $H$ on $\R^{2n}$
with $\{ H=1\}=M$. The flow of $H$ on $M$ agrees with the Reeb flow on
$M$. Let $\Psi^t$ be the linearized flow of $H$ in $\R^{2n}$ along a
closed symmetric orbit $x$ in $M$ with period $T$. Using the constant
trivialization $T\R^{2n}\cong \R^{2n}\times \R^{2n}$ we can view
$\Psi^t$ as a path in $\Sp(2n)$. Then $\Psi^T=(\Psi^{T/2})^2$, since
the trivialization and $H$ are both centrally symmetric, and
therefore, $\Psi^T$ is non-alternating. To be more precise, $\Psi^T$
has an even number of eigenvalues in the range $(-1,0)$ and $-1$ is
not an eigenvalue since the Reeb flow is non-degenerate. Fix a point
$z$ on $x$ and consider the decomposition
$T_z\R^{2n}=\xi_z\oplus \R^2$, where $\xi_z\subset T_zM$ is the
contact plane and the second factor is the span of the Reeb vector
field vector at $z$ and the radial direction. The decomposition is
invariant under $\Psi^T\colon T_z\R^{2n}\to T_z\R^{2n}$, which then
decomposes as $\Phi\oplus \id$, where $\Phi\colon \xi_z\to\xi_z$ is
the linearized Reeb flow in time $T$. It follows that $\Phi$ and
$\Psi^T$ have the same eigenvalues other than 1. Therefore, $\Phi$ and
hence, by definition, $x$ are also non-alternating.

This completes the proof of the theorem. \hfill\qed

\begin{Remark}[Role of symmetry and non-degeneracy]
  It is illuminating to examine the role of symmetry, non-degeneracy
  and our other main results in the proof of Theorem \ref{thm:HZ}. The
  non-degeneracy condition is used to have non-alternating closed
  orbits defined and then, in some sense even more crucially, to
  guarantee that all iterates except the bad ones are $\Q$-visible;
  cf.\ Examples \ref{exam:iterates} and \ref{exam:iterates2}. The
  lower bound $n$ on the number of non-alternating prime orbits
  follows in essence from \cite{GK, LZ}; cf.\ Corollary
  \ref{cor:multiplicity}. The genuinely new part of Theorem
  \ref{thm:mult} is not needed here.  However, we would have the same
  lower bound on the total number of prime closed orbits even without
  non-degeneracy by \ref{Mult}. Likewise, if we replaced dynamical
  convexity by convexity or strong dynamical convexity in Theorem
  \ref{thm:HZ}, which are admittedly more restrictive conditions, by
  \cite{LLZ} or \cite{GM}; cf.\ Corollary \ref{cor:multiplicity-sym}.

  The symmetry condition together with the fact that distinct orbits
  have distinct actions (Corollary \ref{cor:SH}) is central to the
  proof of Theorem \ref{thm:HZ} to rule out the existence of
  asymmetric closed orbits. This step is independent of the
  non-degeneracy requirement.
\end{Remark} 

\subsection{Proof of Theorem \ref{thm:Orbits}}
\label{sec:pf-main2} 
The proof relies on Theorem \ref{thm:SH} and Corollary \ref{cor:SH},
which are the central parts of the argument. Recall that $\PP$ is the
set of $\Q$-visible closed orbits and $\PP'$ stands for $\PP$ with the
extra element $[W]$ of degree $n$ and action $0$. We start with the
proof of Part \ref{O-index}, noting first that the existence of
$x\in \PP$ with $\deg(x)=m$ for any $m\in n-1+2\N$ is an immediate
consequence of \eqref{eq:CH} and \eqref{eq:spec-CH}.

For every $x\in \PP$, there exists exactly one bar $I_-$ ending at $x$
and one bar $I_+$ beginning at $x$ by Lemma \ref{lemma:bars-SH} and
Theorem \ref{thm:SH}. Furthermore, again due to Lemma
\ref{lemma:bars-SH}, there are two possibilities:
  \begin{equation}
  \label{eq:P1}
  \deg(I_-)=\deg(x)-1\textrm{ and then } \deg(I_+)=\deg (x) + 1
  \end{equation}
  or
  \begin{equation}
  \label{eq:P2}
  \deg(I_-)=\deg(x) \textrm{ and then } \deg(I_+)=\deg (x).
  \end{equation}

  Next, denote by $x_1, x_2, \ldots$ the elements of $\PP$ arranged in
  the order of (strictly) increasing action. We also set $x_0$ to be
  the extra element $[W]$ of $\PP'$. By definition, this element has
  action 0 and degree $n$. Let $I_i$ be the bar beginning at $x_i$.

  It immediately follows from \eqref{eq:P1} and \eqref{eq:P2} that the
  sequence $\deg(I_i)$ has constant parity. Since $\deg(I_0)=n$, all
  $\deg(I_i)$ have parity $n$.

  Again by \eqref{eq:P1} and \eqref{eq:P2}, the sequences $\deg(x_i)$
  and $\deg(I_i)$ are (not necessarily strictly) increasing. Moreover,
  \eqref{eq:P2} cannot happen when $\deg(x_i)$ has parity $n+1$; for
  then $\deg(I_i)$ would have parity $n+1$. Hence, $\deg(x_i)$ is
  strictly increasing at the points of $n-1+2\N$. It follows that for
  every $m$ in this set there is exactly one $x\in \PP$ with
  $\deg(x)=m$.

  Furthermore, if we had $y\in \PP$ with $q:=\deg(y)\in n+2\N$ we
  would have at least two orbits in $\PP$ of degree $q-1$ or $q+1$ due
  to \eqref{eq:spec-CH} or, more specifically, \eqref{eq:odd} and
  Theorem \ref{thm:SH}. This is impossible since, as we have shown,
  for every $m\in n-1+2\N$ an orbit $x\in \PP$ with $\deg(x)=m$ is
  unique. This completes the proof of the statement that the degree
  map $\deg$ on $\PP$ takes values in $n-1+2\N$ and is a bijection.

  Together with the fact that all $\deg(I_i)$ have parity $n$, this
  also completely rules out \eqref{eq:P2}. Therefore, the degree map
  $\deg$ on $\CB$ takes values in $n-2+2\N$ and is also a bijection.
  
  The statement that the ordering of the orbits in $\PP$ (or the bars
  in $\CB$) by the action agrees with the ordering by the index is
  simply a reformulation of the fact that the sequence $\deg(x_i)$ is
  strictly increasing. (Recall that the sequence $x_i$ is originally
  ordered by the action.) Alternatively, this statement is also a
  consequence of \cite[Thm.\ 1.3]{GG:LS}. This concludes the proof of
  Part \ref{O-index}.

  Turning to Part \ref{O-clusters}, assume that $x\in\PP$. Then
  $\chieq(x)=(-1)^{\deg(x)}\neq 0$ by \eqref{eq:chieq}. Then by Lemma
  \ref{lemma:chieq-it} and, in particular, \eqref{eq:chieq-it},
  $x^k\in \PP$ for all odd admissible $k\in\N$. Thus, using the
  terminology from \cite{GG:LS}, we see that infinitely many iterates
  $x^k$ occur as carriers of equivariant action selectors (aka
  equivariant capacities). Therefore, all $\Q$-visible closed orbits
  have the same ratio $\CA(x)/\hmu(x)$ by \cite[Thm.\
  6.4]{GG:LS}. \hfill\qed
 
\begin{Remark}
  When the flow is non-degenerate, Part \ref{O-index} has an analogue
  for any field $\F$ with $\deg(x)$ replaced by $\mu(x)$, although the
  statement is then less precise. Namely, the sequences $\mu(x_i)$ and
  $\deg(I_{i})$ are still increasing although now not necessarily
  strictly. However, for every $m\in n-1+2\N$ we still have exactly
  one $\F$-visible closed orbit $x$ with $\mu(x)=m$ and all
  $\deg(I_i)$ have parity $n$. The key difference is that in this case
  we cannot rule out the existence of closed $\F$-visible orbits $x$
  with $\mu(x)$ of parity $n$ and intervals in the sequences
  $\deg(x_i)$ and $\deg(I_i)$ where these sequences are constant. The
  proof is a subset of the proof of Part \ref{O-index} combined with
  the fact that $\supp\SH(x;\F)=\{\mu(x),\mu(x)+1\}$ for any field
  $\F$ when $x$ is non-degenerate, i.e., the second part of Corollary
  \ref{cor:SH} holds automatically.  We do not know to what degree
  Part \ref{O-clusters} extends to this case.
\end{Remark}

\subsection{Proofs of Theorem \ref{thm:SH} and Corollary \ref{cor:SH}}
\label{sec:pf-main3}
The strategy of the proof of the theorem is as follows. Due to the
universal coefficient theorem, it suffices to prove the theorem for
$\F=\F_p$. By the Smith inequality (see Theorem \ref{thm:Smith} or
\cite{Se, ShZ}), $D_t:=\dim\SH^t(W;\F_p)\geq 1$ for all
$t>0$. Moreover, for every $t\in\R\setminus\CS(\alpha)$, there are
arbitrarily long intervals $I\subset\R$ such that $D_\tau\geq D_t$ for
all $\tau\in I$. Next, we infer from the index recurrence theorem
(Theorem \ref{thm:IR-A}) and an upper bound on the boundary depth (see
Theorem \ref{thm:vanishing} or \cite{GS}) that there exists $C\in I$
such that $\SH^C(W;\F_p)$ is supported only in degrees of parity
$n$. This is the central point of the proof, and this is where
dynamical convexity and index recurrence in the form of Theorem
\ref{thm:IR-A} become crucial. Then
$$
D_C=(-1)^n\chi(W)=1
$$
by \eqref{eq:Euler}. Hence, $D_t=1$ for all $t>0$.

\subsubsection{Generalities} We start by stating several facts used in
the proof of the theorem. In this section, we will also reduce the
theorem to an auxiliary result (Proposition \ref{prop:key_prop}) which
is then proved in Section \ref{sec:pf-key_prop}.

Recall that by Theorem \ref{thm:Smith} (the Smith inequality)
proved in \cite{Se, ShZ}, for every Liouville domain $W$, a
characteristic $p$ field $\F_p$ and an interval $I$,
\begin{equation}
  \label{eq:Smith2}
\dim\SH^{pI}(W;\F_p)\geq\dim \SH^I(W;\F_p).
\end{equation}
(The reader may simply set here $\F_p$ to be the standard field
$\Z/p\Z$ of characteristic $p$.) In particular,
$\dim \SH^{pt}(W;\F_p)\geq\dim \SH^t(W;\F_p)$ for all $t>0$.
Therefore, for all $t>0$ and any field $\F$, we have
\begin{equation}
  \label{eq:Smith1}
\dim \SH^t(W;\F)\geq\dim\H(W;\F).
\end{equation}
When $W$ is star-shaped, this also follows from \eqref{eq:Euler}
  for all fields $\F$. However, the Smith inequality is needed here
  for more general Liouville domains.

As a consequence, $\dim \SH^t(W;\F)\geq 1$. Indeed, to prove
\eqref{eq:Smith1}, we first note that for all $k\in \N$ and $t>0$,
$$
\dim \SH^t(W;\F_p)\geq \dim \SH^{t/p^k}(W;\F_p)
$$
due to \eqref{eq:Smith01}.
On the other hand, when $k$ is large, by
\eqref{eq:SH-small-delta} we have
$$
\SH^{t/p^k}(W;\F_p)=\H(W;\F_p)[-n].
$$
Thus, \eqref{eq:Smith1} holds for every positive characteristic
field. Then the case of zero characteristic follows from the universal
coefficient theorem and, in particular, \eqref{eq:universal}.

With the Smith inequality, \eqref{eq:Smith1}, in mind, the proof of
the theorem is based on the following key result giving the opposite
bound for a certain sequence of action values.

\begin{Proposition}
\label{prop:key_prop}
In the setting of Theorem \ref{thm:SH}, for every field $\F$ there
exists a positive real sequence $C_l\to \infty$ with bounded gap such
that
\begin{equation*}
\label{eq:key_prop}
\dim \SH^{C_l} (W; \F) = 1
\end{equation*}
for all $C_l$.
\end{Proposition}

This proposition is proved in Section \ref{sec:pf-key_prop}. Here we
only point out that this is the only part of the proof of Theorem
\ref{thm:SH} where the finiteness and dynamical convexity conditions
are used. The rest of the argument is independent of these
requirements. Note also that the sequence $C_l$ can be taken
independent of the field $\F$. However, we do not need this fact.

\begin{proof}[Proof of Theorem \ref{thm:SH}]
  As above, by the universal coefficient theorem and, in particular
  \eqref{eq:universal}, it is enough to prove \eqref{eq:dim-key} for
  all positive characteristic fields $\F_p$.  Setting
  $$D_t:=\dim\SH^t(W;\F_p),
  $$
  our goal is to show that $D_t =1$ for all $t>0$. By
  \eqref{eq:Smith1}, we have $D_t\geq 1$, and hence we only need to
  prove the opposite inequality.  Clearly, it suffices to show this
  for $t>0$ outside the action (period) spectrum $\CS(\alpha)$ of the
  contact form $\alpha$.

  Thus, fix $t \in \R_+\setminus\CS(\alpha)$ and let $t'>t$ be such
  that $[t,t'] \cap\CS(\alpha) = \emptyset$.  Since
  $ [t,t'] \cap\CS(\alpha) = \emptyset$, we have $D_t=D_\tau$ for all
  $\tau \in [t,t']$. Then, by the Smith inequality, \eqref{eq:Smith2},
  again, applied $k$ times to $I=[0,\tau]$, we have
$$
D_\tau\leq D_{p^k\tau}
$$
for all $k\in \N$. Here $p^k\tau$ ranges through the interval
$I_k=[p^kt, p^kt']$ as $\tau$ ranges through $[t,t']$. Now let
$C_l\to \infty$ be as in Proposition \ref{prop:key_prop}. Since the
sequence $C_l$ has bounded gap, for all sufficiently large $k$, the
intersection of the sequence $\{C_l\}$ and $I_k$ is non-empty:
$\{ C_l \} \cap [p^kt, p^kt'] \neq \emptyset$. Therefore, there exists
$\tau\in [t,t']$ and $k\in\N$ and $l\in\N$ such that $p^k\tau=C_l$.
By Proposition \ref{prop:key_prop}, $D_{C_l} \leq 1$. Combining these
facts, we arrive at
$$
D_t=D_\tau\leq D_{p^k\tau=C_l}=1.
$$
Hence, $D_t\leq 1$, which concludes the proof of the theorem.
\end{proof}

\subsubsection{Proof of Proposition \ref{prop:key_prop}}
\label{sec:pf-key_prop}
Fix a field $\F$, which we suppress in the notation. As in Section
\ref{sec:IR-Thm}, let us break down the collection $\PP=\PP(\F)$ of
$\F$-visible closed orbits into several groups, called
\emph{clusters}. We say that two such orbits $x$ and $y$ belong to the
same cluster if $\CA(x)/\hmu(x)=\CA(y)/\hmu(y)$. (Thus, for instance,
\ref{O-clusters} asserts that there is only one cluster.)

Consider now the maps $\beg\colon \CB\to\PP'$ and $\en\colon\CB\to\PP$
from Theorem \ref{thm:beg-end}, where $\CB$ is the collection of bars
of the graded persistence module $\SH(W)$.  Let $I\in \CB$, and
$x=\beg(I)$ and $y=\en(I)$. Then

 \begin{itemize}

\item[\reflb{bar-actions}{(B1)}] $a=\CA(x)$ and $b=\CA(y)$;

\item[\reflb{bar-indices}{(B2)}] $\mu_-(y) -
  \mu_+(x)\leq 2$; 

\item[\reflb{bar-degrees}{(B3)}] 
  $\deg(I)=\mu_+(x)+1=\mu_-(y)-1$ if $\mu_-(y) - \mu_+(x) = 2$.

\end{itemize}
Here \ref{bar-actions} is the first item in Theorem \ref{thm:beg-end},
\ref{bar-indices} is \eqref{eq:mu(x)-mu(y)1}, and \ref{bar-degrees} is
\eqref{eq:mu(x)-mu(y)2}.

Our first observation is that beyond some action threshold there are
no inter-cluster bars.

\begin{Lemma}
  \label{lemma:inter-cluster}
  There exists a constant $\Kbar>0$ such that for every bar $(a, b]$
  with $a>\Kbar$ beginning at an orbit $x$ and ending at $y$, the
  orbits $x$ and $y$ are in the same cluster.
\end{Lemma}

\begin{Remark}
  The constant $\Kbar$ can be taken independent of the background
  field $\F$, but we do not need this fact.
\end{Remark}

\begin{proof}
  Recall that by Theorem \ref{thm:vanishing}, for any field $\F$ all
  bars in $\SH(W)$ are finite and bounded from above by a constant
  $\Cbar$. (This constant can also be taken independent of $\F$, but
  this is not essential for our purposes.) In other words, there
  exists $\Cbar >0$ such that any bar $(a,b]$ has length
  $b-a < \Cbar$.

  Observe that by \ref{bar-indices} the difference
  $\big|\hmu(y) - \hmu(x)\big|$ is a priori bounded from above by
  $2n$. Here we use the fact that $|\mu_\pm - \hmu | \leq n-1$ by
  \eqref{eq:mu-del}. Set $\rho=\CA(x)/\hmu(x)$ and
  $\rho'=\CA(y)/\hmu(y)$ and recall that $b=\CA(y)$ and
  $a=\CA(x)$. Then a direct calculation shows that 
  $$
  b-a\geq |\rho/\rho'-1|\cdot b-2n\rho.
  $$
  We also have $b-a < \Cbar$. Therefore, since there are only finitely
  many clusters, we must have $\rho=\rho'$ when $a>\Kbar$ for some
  constant $\Kbar$.
\end{proof}

Next, let us label the prime orbits as $x_{ij}$, where
$i=1,\ldots, i_0$ and $j=1,\ldots, j_0 (i)$, with the index $i$
indicating the cluster. In other words, $x_{ij}^k$ and $x_{i'j'}^{k'}$
are in the same cluster if and only if $i=i'$. (Note that $x_{ij}^k$
may fail to be $\F$-visible, i.e., $x_{ij}^k$ is not necessarily in
$\PP$, for some $k\in\N$.)

By Theorem \ref{thm:IR-A} and Corollary \ref{cor:IR-A-DC}, there
exists a positive real sequence $C_l \to \infty$ and for each
$i=1,\ldots, i_0$ there is a sequence of positive even integers
$d_{il} \to \infty$ as $l\to \infty$ such that for every iterated
orbit $x_{ij}^k$ we have
\begin{equation}
  \label{eq:IR-A1}
  \mu_+\big(x_{ij}^k\big) \leq d_{il} + n-1 \textrm{ when }
  \CA\big(x_{ij}^k\big) < C_l
\end{equation}
and 
\begin{equation}
  \label{eq:IR-A2}
  \mu_-\big(x_{ij}^k\big) \geq d_{il} + n+1 \textrm{ when }
  \CA\big(x_{ij}^k\big) > C_l.
\end{equation}
Indeed, in the setting of Theorem \ref{thm:IR-A} and Corollary
\ref{cor:IR-A-DC} for a recurrence event $k_{ijl}$ we have
$k\leq k_{ijl}$ in \eqref{eq:IR-A1} and $k>k_{ijl}$ in
\eqref{eq:IR-A2}. Here both action inequalities are strict, since the
sequence $C_l$ can be chosen outside the action spectrum
$\CS(\alpha)$.

The next lemma asserts that for sufficiently large $C_l$, the graded
vector space $\SH^{C_l} (W)$ is supported in degrees
$\{d_{il}+n\mid i=1,\ldots, i_0\}$.

\begin{Lemma}
\label{lemma:IRT-support}
Let $C_l \to \infty$ and $d_{il} \to \infty$ be as above, i.e., these
sequences satisfy \eqref{eq:IR-A1} and \eqref{eq:IR-A2} and all
$d_{il}$ are even. Then there exists a constant $K>0$ such that once
$C_l > K$ we have
\begin{equation}
  \label{eq:IRT-support}
\dim\SH^{C_l} (W) = \sum_i \dim \SH_{d_{il}+n} ^{C_l} (W).
\end{equation}
In particular, $\SH^{C_l} (W)$ is supported in degrees of the same
parity as $n$.
\end{Lemma}

\begin{proof}
  Note that $\dim\SH^{C_l} (W)$ is equal to the number of bars
  containing $C_l$. Since $b-a<\Cbar$ by Theorem \ref{thm:vanishing}
  for any bar $(a,b]$, the bars $(a,b]$ with $a < C_l -\Cbar$ do not
  contribute to $\dim\SH^{C_l} (W)$. Next, fix a bar $(a,b]$
  containing $C_l$, beginning at $x_{ij}^k$ and ending at
  $x_{i'j'}^{k'}$. Let us also assume from now on that
  $C_l>K:=\Kbar+\Cbar$. Then $a>\Kbar$, and hence, by Lemma
  \ref{lemma:inter-cluster}, $i=i'$, i.e., these orbits are in the
  same cluster.

  We claim that when $C_l>K$ the degree of $I=(a,b]$ is in the set
  $\{d_{il}+n\}$. This will prove \eqref{eq:IRT-support} since the
  equality simply asserts that $\dim\SH_{m} ^{C_l} (W)=0$ when
  $m\neq d_{il}+n$.

  To see that $\SH^{C_l} (W)$ is supported in the degrees $d_{il}+n$,
  first note that
$$
a=\CA\big(x_{ij}^k\big)< C_l < \CA\big(x_{ij'}^{k'}\big)=b.
$$
Hence, by \eqref{eq:IR-A1} and \eqref{eq:IR-A2},
\[
  \mu_+\big(x_{ij}^k\big) \leq d_{il} + n-1 \quad \textrm{and} \quad
  \mu_-\big(x_{ij'}^{k'}\big) \geq d_{il} + n+1.
\]
(Here we are using the fact that $i=i'$, i.e., the orbits are in the
same cluster.) Due to \ref{bar-indices}, we have equality in both of
these inequalities and the index difference is exactly $2$. Therefore,
by \ref{bar-degrees},
$\deg(I)=\mu_{+}\big(x_{ij}^{k}\big)+1=d_{il} + n$.
\end{proof}

It remains to show that $D:=\dim\SH^{C_l} (W)=1$. By
\eqref{eq:Euler},
$$
\sum_m(-1)^m\dim\SH^{C_l}_m(W)=(-1)^n\chi(W)=(-1)^n.
$$
At the same time, the left-hand side is equal to $(-1)^nD$ by Lemma
\ref{lemma:IRT-support}. Indeed, each term on the left is either $0$
when $m\neq d_{il}+n$ or $(-1)^n\dim\SH^{C_l}_m(W)$ when $m=d_{il}+n$
since $d_{il}$ is even. Thus
$$
\sum_m(-1)^m\dim\SH^{C_l}_m(W)=(-1)^n\dim\SH^{C_l}(W)=(-1)^nD.
$$
As a consequence, $(-1)^n D=(-1)^n$, i.e., $D=1$. This completes the
proof of the proposition and hence of the theorem. \hfill\qed

\subsubsection{Proof of Corollary \ref{cor:SH}}
\label{sec:corSH-pf}
The first assertion is a consequence of \eqref{eq:dimgeq2},
\eqref{eq:bars-SH} and Theorem \ref{thm:SH}. Indeed, if we had two
$\F$-visible orbits $x$ and $x'$ with action $a$, by
\eqref{eq:dimgeq2} and \eqref{eq:bars-SH}, we would have at least four
bars with end-point $a$ and hence $\dim\SH^t(W)\geq 2$ for some $t$
close to $a$.  Likewise, $\dim \SH(x;\F)=2$ for all $\F$-visible
orbits by \eqref{eq:bars-SH} and Theorem \ref{thm:SH}.  Now the second
assertion follows immediately from \eqref{eq:SH-CH} and the first
one. \hfill\qed

\subsection{Proof of Corollary \ref{cor:ellipsoids}}
\label{sec:ellipsoids-pf}
As was already mentioned in Section \ref{sec:ellipsoids}, by Part
\ref{O-clusters} of Theorem \ref{thm:Orbits}, we can scale $M$ so that
$\hmu(x)=\CA(x)$ for all closed orbits in $M$. From now on we assume
that this condition is satisfied. Denote by $x_1,\ldots, x_n$ the
prime closed orbits in $M$.  Then the action spectrum of $M$ is the
same as the mean index spectrum and is simply the union of the
sequences $\Delta_j\N$.  Note also that we have
$\Delta_j/\Delta_i\not\in \Q$, when $i\neq j$, again by
\ref{O-clusters}, since all closed orbits have distinct actions due to
Corollary \ref{cor:SH}.
  
The prime closed orbits in $M_{\vDelta}$ have the form
$$
y_j(t)=(0,\ldots,0, z_j(t),0,\ldots, 0),
$$
where $z_j(t)=\sqrt{\Delta_j}\exp\big(2\sqrt{-1}t/\Delta_j \big)$ with
$t\in [0, \pi\Delta_j)$. Therefore, in suitable coordinates, the
Poincar\'e return map $\Phi_j$ along $y_j$ is the $(n-1)\times (n-1)$
diagonal matrix with entries $\exp(2\pi\sqrt{-1}\lambda_{ij})$, where
$\lambda_{ij}=\Delta_j/\Delta_i$ and $i\neq j$. Since
$\Delta_j/\Delta_i\not\in \Q$, all orbits $y_j$ are strongly
non-degenerate and all eigenvalues of $\Phi_j$ are elliptic. Note,
however, that the ``logarithmic'' eigenvalues $\lambda_{ij}$ are not
normalized in any sense. For instance, the eigenvalues $\lambda_{ij}$
can be arbitrarily large or small and need not be of first Krein type.

Recall that the non-resonance condition is that for every $j$ the
ratios $\pm \Delta_j/\Delta_i=\pm\lambda_{ij}$ are distinct modulo
$\Z$ for every fixed $j$ and all $i\neq j$. This guarantees that the
first Krein type eigenvalues are distinct for every $j$ and are not
complex conjugate to each other.

A straightforward calculation shows that 
$$
\hmu(y_j)=2\Delta_j\sum_{i=1}^n\frac{1}{\Delta_i}.
$$
As a consequence of Theorem \ref{thm:mult}, the prime orbits
$x_1,\ldots, x_n$ are non-alternating. By \cite{Vi:R}, the sum in this
formula is $1/2$; see also \cite{GiK}. Therefore,
\begin{equation}
  \label{eq:hmu-xj-yj}
\hmu(y_j)=\Delta_j=\hmu(x_j).
\end{equation}

Clearly, $\CA(x_j^k)=k\Delta_j=\CA(y_j^k)/\pi$.  By Theorem
\ref{thm:Orbits}, for each $m\in\N$, the $m$th entry in the action
spectrum, for $M$ and $M_{\vDelta}$, is the value of the $m$th
$S^1$-equivariant spectral invariant over $\Q$, which is attained on
an orbit of index $n+m$; see \cite{GG:LS}.  It follows that
$\mu\big(x_j^k\big)=\mu\big(y_j^k\big)$. By \cite[Cor.\
4.3]{GG:RvsPR}, for every $j$ the linearized return maps along $y_j$
and $x_j$ have the same first Krein type eigenvalues, which are all
distinct. Furthermore, the linearized flows also have the same mean
index by \eqref{eq:hmu-xj-yj}. As a consequence, the linearized flows
along $x_j$ and $y_j$ agree as elements of $\TSp(2(n-1))$ up to
conjugation.  \hfill\qed

\begin{Remark}
  Without the non-resonance condition, the argument still provides
  some comparison information about the linearized flows along $x_j$
  and $y_j$. Namely, it shows that the return maps have the same first
  Krein type eigenvalues, counted with multiplicity which is defined
  as follows. Assume that a linear symplectic map has both $w$ and
  $\bar{w}$ as first Krein type eigenvalues and their algebraic
  multiplicities are $m$ and, respectively, $l$. Then we set, by
  definition, the multiplicity of $w$ to be $m-l$ and the multiplicity
  of $\bar{w}$ to be $l-m$; see \cite[Sec.\ 4]{GG:RvsPR}. (Note that
  with this way of counting multiplicity an elliptic orbit such that
  its first Krein type eigenvalues come in complex conjugate pairs is
  indistinguishable from a hyperbolic orbit.)  Furthermore, as in the
  non-resonance case, $x_j$ and $y_j$ have the same mean index (see
  \eqref{eq:hmu-xj-yj}) and also, up to a factor, the same action.
\end{Remark}  

\section{Proof of the index recurrence theorem}
\label{sec:IRT-pf}
Our goal in this section is to prove the index recurrence theorem,
Theorem \ref{thm:IR-A}. The section is broken down into two parts: we
first review some standard facts from geometry of numbers and then
turn to the actual proof.

\subsection{Input from geometry of numbers}
\label{sec:numbers}
The existence of the sequences $C_l$, $d_{il}$ and $k_{ijl}$ with the
required properties hinges on two standard facts along the lines of
Kronecker's and Minkowski's theorems; see, e.g., \cite{Ca}. These
facts are completely standard and obvious to the experts.

Let $U\subset \T^n=\R^n/\Z^n$ be a neighborhood of $0$ and
$\theta\in \T^n$. Consider the sequence
$k_1\leq k_2\leq k_3\leq \ldots$ of all positive integers such that
$k_l\theta\in U$. Clearly, $k_l\to\infty$. Moreover, it is well known
that the orbit $\{k\theta\mid k\in \N\}$ is equidistributed in the
closure $G$ of $\theta\Z$ in $\T^n$. In particular, the density of
this sequence in $\N$ is equal to the $\dim G$-dimensional volume of
$U\cap G$.

\begin{Lemma}
  \label{lemma:gap}
The sequence $k_l$ has bounded gap. 
\end{Lemma}

For the sake of completeness we include a short proof of the lemma.

\begin{proof}
  Since $G$ is compact, there exists a finite collection of integers
  $S=\{s_1,\ldots, s_N\}$ such that $G$ is covered by the shifted
  neighborhoods $s_i\theta+ U$. Thus, for every $\ell\in\N$, we have
  $\ell\theta\in s_i\theta+U$ for some $i$, which might not be
  unique. We pick one such $i$ and write $s_i=: s(\ell)$.

  Set $q_\ell=\ell-s(\ell)$. Clearly, $q_\ell\theta\in U$. Thus the
  sequence $q_\ell$ takes values in the set $M:=\{k_l\mid l\in\N\}$
  and $q_\ell\to\infty$ since the set $S$ is finite. The maximal gap
  in the sequence $k_l$ is equal to the maximal gap in the set $M$
  which is bounded from above by the maximal gap in the sequence
  $q_\ell$. The latter sequence has bounded gap, for
$$
|q_{\ell+1}-q_\ell|\leq 1 + |s(\ell+1)-s(\ell)|\leq 1+\max_{i,j\in S}|
s_i-s_j|.
$$
\end{proof}

We will also need the following consequence of Lemma \ref{lemma:gap}
along the lines of Minkowski's theorem on linear forms; see, e.g.,
\cite{Ca}.  Consider a collection of $p<n $ linear maps
$$
f_s\colon \R^n\to \R, \quad s=1,\ldots, p
$$
and positive constants $\delta_1>0, \ldots, \delta_p>0$. We are
interested in integer solutions $K\in\Z^n$ of the system of
inequalities
\begin{equation}
  \label{eq:system}
|f_s(K)|<\delta_s, \quad s=1,\ldots, p.
\end{equation}

\begin{Lemma}
  \label{lemma:Mink}
  The system of inequalities \eqref{eq:system} has a sequence of
  non-zero solutions $K_l\to \infty$ in $\Z^n$ with bounded gap, i.e.,
  such that the norm of $K_l$ goes to $\infty$, but the norm of the
  difference $K_{l+1}-K_l$ is bounded by a constant independent of
  $l$. Moreover, we can make all components of $K_l$ divisible by any
  fixed integer.
\end{Lemma}

Since $p<n$, this system has infinitely many distinct solutions by the
Minkowski theorem; see \cite{Ca}. However, the standard proof of this
theorem does not immediately imply the bounded gap part. Hence we
include a proof of the lemma for the sake of completeness.

\begin{proof}[Proof of Lemma \ref{lemma:Mink}] Let $\tU\subset \R^n$
  be a small ball centered at the origin such that
  $|f_s|_{\tU}|<\delta_s$ for all $s$ and the balls $K+\tU$ for all
  $K\in\Z^n$ do not overlap. Denote by $U$ the projection of $\tU$ to
  $\T^n=\R^n/\Z^n$.

  Since $p<n$, the hyperplanes $\ker f_s$ have a non-zero
  intersection. Pick a non-zero vector
  $$
  \ttheta\in \bigcap_{s=1}^p\ker f_s.
  $$
  Without loss of generality, we may assume that $\ttheta\not\in \Z^n$
  for otherwise we can simply set $K_l=l\ttheta$. Thus the projection
  $\theta$ of $\ttheta$ to $\T^n$ is also non-zero. By Lemma
  \ref{lemma:gap}, there exists a bounded gap sequence $k_l\to \infty$
  such that $k_l\theta\in U$ or equivalently
  $$
  k_l\ttheta\in K_l+ \tU
  $$
  for some sequence $K_l\in\Z^n$. Clearly, $K_l$ has bounded gap and
  $K_l\to\infty$.

  We claim that $K_l$ satisfies the system of inequalities,
  \eqref{eq:system}. Indeed, let us write
  $$
  k_l\ttheta= K_l+v_l
  $$
  where $v_l\in \tU$. Then $f_s(k_l\ttheta)=0$, and hence
  $$
  |f_s(K_l)|=|f_s(v_l)|<\delta_s.
  $$
  Replacing $\ttheta$ by $N\ttheta$, we can make all components of
  $K_l$ divisible by $N\in\N$ if needed.
\end{proof}

\subsection{Proof of Theorem \ref{thm:IR-A}}
The strategy of the proof is as follows. After settling some
preliminaries in Sections \ref{sec:isospectral} and
\ref{sec:prelim-obs}, we show in Section \ref{sec:IRT-pf-r=1} that for
any element $\Phi:=\Phi_{ij}$ assertions \ref{IRA0}--\ref{IRAa} are
formal consequences of the requirement that $k=k_{ijl}$ satisfies a
certain system of inequalities completely determined by $\Phi$. Here
$d:=d_{il}$ is determined by $k$ and, of course, $\Phi$ by the first
part of \ref{IRA0}. Furthermore, \ref{IRAa} can be taken as the
defining condition for $C_l$. (The existence of an infinite bounded
gap sequence of solutions $k$ of this system of inequalities is a
consequence of Lemma \ref{lemma:gap}.)

As the second step of the proof, we prove in Section
\ref{sec:CGT-pf-gen} that for a finite collection of maps $\Phi_{ij}$
as in the theorem the resulting systems of inequalities have bounded
gap sequences of solutions $k_{ijl}$ such that $d_{il}$ is independent
of $j$ within every cluster and $C_l$ is independent of $i$ and
$j$. This is done by using Lemma \ref{lemma:Mink}. This independence
is the key new component of the second step. For simply applying the
first step to each $\Phi_{ij}$ individually would result in $d_{il}$
depending on $j$ and $C_l$ depending in addition on $i$ and $j$.

\subsubsection{Isospectral families}
\label{sec:isospectral}
In the proof, we will use a general fact from symplectic linear
algebra, which allows one to simplify certain index
calculations. Recall that a linear map is said to be \emph{semisimple}
if for every eigenvalue its algebraic multiplicity is equal to the
geometric multiplicity. It is essential for what follows that
  $P\in\Sp(2m)$ is semisimple if and only if it is conjugate in
  $\Sp(2m)$ to a direct sum of the following linear symplectic maps:
  rotations of $\R^2$ including possibly the identity map, hyperbolic
  maps of $\R^2$, hyperbolic maps of $\R^4$ with complex
  coefficients. (This fact is not obvious, but follows immediately
  from the Williamson normal form theorem; \cite{Wi}.) Furthermore, a
family of linear maps $P_t$, $t\in [0,1]$, is said to be
\emph{isospectral} if the eigenvalues of $P_t$ are constant. (Then, as
a consequence, they are constant as a multiset.)

For such a family, the mean index $\hmu(P_t)$ and, when $P$ is
non-degenerate, the Conley--Zehnder index $\mu(P_t)$ are constant as
follows immediately from the definitions; see, e.g., \cite{Lo,SZ}.
(The non-degeneracy condition is essential: more sensitive
invariants such as $\mu_\pm$ or $\beta_\pm$ need not be constant.)

\begin{Lemma}
\label{lemma:isospectral}
  Every $P\in\Sp(2m)$ can be connected in $\Sp(2m)$ to a semisimple
  linear map by an isospectral family.
  \end{Lemma}

  As a consequence, in many index calculations we can assume that the
  map is semisimple by replacing $P_0$ by $P_1$. Lemma
  \ref{lemma:isospectral} has certainly been known in some contexts,
  but we are not aware of any reference and we include a proof for the
  sake of completeness. In fact, we need the lemma only when $P$ is
  elliptic and non-degenerate, but the general case might be of
  independent interest and the proof under these conditions is not
  much shorter. In the proof we employ the variant of
    Williamson's normal forms, \cite{Wi}, from \cite{CK}; see also
    \cite{LM} and \cite[Sect.\ 2.6]{CLW}. (There is a typo in
    \cite[(4), p.\ 120--121]{CLW}: $I_2$ should be the $2\times 2$
    identity matrix.) Note also that various sources use different
    sign conventions in the Hamilton equation. Clearly, this has no
    effect on the normal forms, Lemma \ref{lemma:isospectral} and
    other proofs in the paper. In the proof of the lemma we adopt the
    conventions from \cite{CK}.

  \begin{proof} By decomposing $P$ as a direct sum depending on the
    type of eigenvalues, we only need to prove the lemma when $P$ is
    one of the following three types:
    \begin{itemize}

    \item All eigenvalues of $P$ are elliptic (i.e., unit) and
      different from $-1$.

        \item The only eigenvalue is $-1$.

    \item All eigenvalues are hyperbolic, i.e., not unit.

\end{itemize}

We will start with the first case. Since $P$ has no real negative
eigenvalues, it lies in the image of the exponential map
$\exp\colon \spl(2m)\to \Sp(2m)$, \cite{Wi:exp}. Thus
$P=e^A$, where $A$ is a linear Hamiltonian operator, aka a linear
Hamiltonian vector field, and it is sufficient to connect $A$ to a
semisimple operator in $\spl(2m)$. When we order the linear Darboux
coordinates as $(p_1,\ldots, p_m, q_1, \ldots, q_m)$, the operator $A$
takes the form
$$
A=\left[
  \begin{array}{cc}
    A_{11} & A_{12}\\
    A_{21} & A_{22}
  \end{array}
\right],
$$
where $A_{11}\in \gl(n)$ is an arbitrary matrix,
$A_{22}=-A_{11}^\top$, and $A_{12}$ and $A_{21}$ are symmetric:
$A_{12}=A_{12}^\top$ and $A_{21}=A_{21}^\top$. Furthermore, since $P$
is elliptic, $A$ has pure imaginary eigenvalues. Let us denote the
spectrum of $A$ by $\sigma(A)$ with eigenvalues taken without
multiplicity. Without loss of generality, we may assume that $A$ does
not have semisimple eigenvalues.  Then, depending on $\sigma(A)$, the
operator $A$ breaks down into a direct sum of the following four
normal forms:
\begin{itemize}
\item[\reflb{NF1}{(NF1)}] In this case $\sigma(A)=\{0\}$, $m$ is odd,
  and $A_{11}$ is the standard Jordan block of size $m$ with
  eigenvalue 0, i.e.,
\[
A_{11}=\left[
  \begin{array}{cccc}
    0 & 1 &  &  \\
    & \ddots  &\ddots  & \\
    & &\ddots &1\\
    & & &  0\\
  \end{array}
\right]
\]
with the remaining empty spots standing for zeros, and
$A_{12}=0=A_{21}$. For $m=1$, we set $A=0$.

\item[\reflb{NF2}{(NF2)}] As above, $\sigma(A)=\{0\}$ and
  $A_{11}$ is the standard Jordan block of size $m$ with
    eigenvalue 0 and $A_{21}=0$, but $m$ is arbitrary and all entries
    of $A_{12}$ are zero except the bottom right corner which is
  $\pm 1$. When $m=1$, this is the $2\times 2$ matrix with $\pm 1$ in
    the top right corner and the remaining entries zero.

\item[\reflb{NF3}{(NF3)}] Here $\sigma(A)=\{\sqrt{-1}b\}$ with $b> 0$,
  $m$ is even, and
\[
A_{11}=\left[
  \begin{array}{cccc}
    B & I &  &  \\
    & \ddots  &\ddots  & \\
    & &\ddots & I\\
    & & &  B\\
  \end{array}
\right],
\]
where
\[
B=\left[
  \begin{array}{cc}
    0 & b\\
    -b & 0
  \end{array}
\right]
\textrm{ and }
I=\left[
  \begin{array}{cc}
    1 & 0\\
    0 & 1
  \end{array}
\right]
\]
and the empty spots are again zeros.  Furthermore, $A_{21}=0$ and all
entries of $A_{12}$ are zero except the $2\times 2$ block $\pm I$ at
the bottom right corner. When $m=2$,
$$
A=\left[
  \begin{array}{cc}
    B & \pm I\\
    0 & B
  \end{array}
\right].
$$

\item[\reflb{NF4}{(NF4)}] In this case $\sigma(A)=\{\sqrt{-1}b\}$ with
  $b\neq 0$, $m$ is odd, and $A_{11}$ is the standard Jordan
  block with eigenvalue 0 as in \ref{NF1}. However, $A_{12}$ and
  $A_{21}$ are anti-diagonal matrices:
\[
A_{12}=\left[
  \begin{array}{ccccc}
    &  &  & & b \\
    &  &  & -b & \\
    &  &\iddots  & & \\
    & -b &   &  & \\
   b &  &   &  & \\
  \end{array}
\right]=-A_{21}.
\]
For $m=1$, we have
$$
A=\left[
  \begin{array}{cc}
    0 & b\\
    -b & 0
  \end{array}
\right].
$$
\end{itemize}
These normal forms, originating \cite{Wi}, are taken from \cite[p.\
63: (aii), (ai), (di), (dii)]{CK}.

Next, let us replace $\pm 1$ (but not $\pm b$) by $\pm t$ in all
occurrences in these matrices. For instance, in \ref{NF3} $I$
becomes $t I$ in $A_{11}$, etc. In each case, we obtain a family of
Hamiltonian operators $A(t)$, $t\in [0,1]$, with $A(1)=A$.

It is easy to see that $A(0)$ is semisimple in all four cases. Indeed,
$A(0)=0$ in \ref{NF1} and \ref{NF2}. In \ref{NF3} and \ref{NF4},
$A(0)$ is equivalent via a linear symplectic change of variables to a
direct sum of infinitesimal rotations with eigenvalues
$\pm \sqrt{-1}b$ in $m$ symplectic planes.

We claim that the four families $A(t)$ are isospectral. Let
$\rho_t(z)=\det\big(zI-A(t)\big)$ be the characteristic polynomial. It
is sufficient to show that $\rho_t(z)$ is independent of $t$ and this
can be done by a direct calculation. In \ref{NF1} or \ref{NF2}, we
have the determinant of a block-diagonal or block-triangular matrix
with triangular matrices on the diagonal, and $\rho_t(z)=z^{2m}$ for
all $t$. Likewise, in \ref{NF3} we have the determinant of a
block-triangular matrix with block-triangular matrices on the
diagonal, and $\rho_t(z)=(z^2+b^2)^{m}$. To calculate $\rho_t(z)$ in
\ref{NF4}, we expand the determinant in the first row and then expand
each of the resulting two non-zero terms in the last row. Then
$\rho_t(z)$ factors as the product of $z^2+b^2$ and the characteristic
polynomial of a $(2m-2)\times(2m-2)$ matrix of exactly the same form,
but with $b$ replaced by $-b$. (The resulting matrix is not a normal
form because $m-1$ is even. However, this does not matter since we are
just calculating the characteristic polynomial.)  Proceeding
inductively, we again see that $\rho_t(z)=(z^2+b^2)^{m}$. This
concludes the proof for an elliptic map $P$ without eigenvalue $-1$.

The second case is where $-1$ is the only eigenvalue of $P$. Then $-P$
is covered by the previous case. Thus there exists an isospectral
family connecting $-P$ to a semisimple symplectic map. (In fact, this
map is $I$ because $1$ is the only eigenvalue of $-P$.)  Multiplying
this family by $-I$, we obtain the required isospectral family
for~$P$.

The remaining case is where $P$ is hyperbolic. This case can be dealt
with in a fashion similar to the first two, but a much simpler proof
is available. Namely, $\R^{2m}=W_-\oplus W_+$ be the decomposition of
$\R^{2m}$ into subspaces corresponding to the eigenvalues $\lambda$ of
$P$ with $|\lambda|<1$ and, respectively, $|\lambda|>1$. These are
$P$-invariant Lagrangian subspaces. As a consequence, the symplectic
form gives rise to an isomorphism $W_+=W_{-}^{*}$ and $P=(P_-,\, P_+)$
with $P_+=(P_{-}^{-1})^\top$ in this decomposition. (In other words,
identifying $\R^{2m}=W_-\oplus W_{-}^{*}$ with $T^*W_-$ we identify
$P\colon T^*W_-\to T^*W_-$ with the map induced by
$P_{-}^{-1}\colon W_-\to W_-$.) Let us connect $P_-$ to an semisimple
linear map by a family $P_-(t)\colon W_-\to W_-$ of isospectral linear
maps. The existence of such a family immediately follows from the
Jordan normal form theorem. Then
$P_t:=\big( P_-(t),\, (P_{-}^{-1}(t))^\top \big)$ is a family of
isospectral maps in $\Sp(2m)$ connecting $P$ to a semisimple map.
This completes the proof of Lemma \ref{lemma:isospectral}.
  \end{proof}

  \begin{Remark}
    Throughout the proof of the lemma, one could have used the
      normal forms for symplectic matrices from \cite{Wi}, but the
      proof would be more involved. The last step of the proof can
      also be used to directly obtain the normal forms of hyperbolic
      symplectic or Hamiltonian operators from the Jordan normal form
      theorem.
  \end{Remark}    

  \begin{Remark}
    Alternatively, one can prove the lemma as follows. The
      symplectic Jordan decomposition of $P$ has the form $P=P_0U$,
      where $P_0$ is symplectic and semisimple, $U$ is symplectic and
      unipotent, i.e., $(U-I)^{2m}=0$, and $P_0$ and $U$
      commute. Furthermore, we have $U=\exp(N)$, where we specifically
      take $N\in\spl(2m)$ to be $\log U$. Then $N$ also commutes with
      $P_0$ and is nilpotent. Finally, the required isospectral family
      is $P_t=P_0\exp(tN)$, $t\in[0,1]$.
  \end{Remark}

  \subsubsection{Preliminary observations}
  \label{sec:prelim-obs}
  Before turning to the actual proof of Theorem \ref{thm:IR-A} it is
  useful to make a couple of simple observations.

  The claim that the sequences $C_l$, $d_{il}$ and $k_{ijl}$ can be
  taken strictly increasing is a formal consequence of the rest of the
  theorem. Indeed, observe that every bounded gap sequence
  $d_{il}\to\infty$ has a bounded gap strictly increasing refinement
  $s\mapsto d_{il_s}$. Moreover, for any $c>0$ we may ensure by
  further thinning the sequence that $d_{il_{s+1}}-d_{l_s}>c$. Hence,
  by passing to subsequences of $d_{il}$ and $k_{ijl}$, we can
  guarantee that $d_{il}$ is strictly increasing, $d_{il+1}-d_{il}>c$
  and these sequences still have bounded gap. Then, since
  $\hmu\big(\Phi^{k_{ijl}}_{ij}\big)=k_{ijl}\hmu(\Phi_{ij})$ and
  $\hmu(\Phi_{ij})>0$, we conclude from \ref{IRA0} that the sequences
  $k_{ijl}$ are automatically strictly increasing when $c$ is
  sufficiently large. Likewise, the sequence $C_l$ is forced to be
  strictly increasing by \ref{IRAa} once $\eta>0$ is sufficiently
  small.

  In a similar vein, again by \ref{IRA0}, the sequences $d_{il}$ and
  $k_{ijl}$ have bounded gap if and only if one of them does. Thus it
  suffices to prove this for the sequences $k_{ijl}$.

  To summarize, we only need to show that there exist sequences $C_l$,
  $d_{il}\to\infty$ and $k_{ijl}\to \infty$ satisfying the
  requirements of the theorem and such that one of the sequences
  $k_{ijl}$ has bounded gap, and that we can make $d_{il}$ and
  $k_{ijl}$ divisible by any pre-assigned integer.

  Finally, note that throughout the proof we can assume without loss
  of generality that $\eta>0$ is sufficiently small, e.g.,
  $\eta<\min a_{ij}$.

\subsubsection{The case of one map: $r=1$.}
\label{sec:IRT-pf-r=1} 
Let $\Phi=\Phi_{1j}\in\TSp(2m)$. Throughout the argument we suppress
$i$ and $j$ in the notation, i.e., we will write $k=k_l$ for
$k_{1jl}$, etc.

Let us decompose $\Phi$ as the product of a loop $\varphi$ and a
direct sum $\Phideg\oplus\Phi_h\oplus\Phi_e$ with the following
properties. The end-point $\Phideg(1)$ is totally degenerate and
$\hmu(\Phideg)=0$. The path $\Phi_h$ has a hyperbolic end-point
$\Phi_h(1)$, while $\Phi_e$ is the elliptic part of $\Phi$, i.e., all
eigenvalues of $\Phi_e(1)$ are on the unit circle and different from 1
due to the presence of $\Phideg$. We write the eigenvalues of
$\Phi_e(1)$ as $\exp(\pm 2\pi\sqrt{-1}\lambda_q)$, where
$0<|\lambda_q|\leq 1/2$ and $q=1,\ldots, m'$. Here for each pair of
complex conjugate eigenvalues it is convenient to pick the sign of
$\lambda_q$ so that the resulting eigenvalue is of the first Krein
type; see \cite{SZ}. (We will elaborate on this choice later.)
Furthermore, we then require that
\begin{equation}
  \label{eq:hmu-elliptic}
  \hmu(\Phi_e)=2\sum_q\lambda_q.
\end{equation}  
Note that one or more terms in this decomposition might be absent. For
instance, $\Phideg$ is not present when $\Phi(1)$ is non-degenerate.

It is easy to see that such a decomposition exists; see, e.g.,
\cite[Sec.\ 3]{SZ}. (In general, the decomposition is not unique in
$\TSp(2m)$ without extra assumptions on the index of $\Phi_h$. For
instance, the loop $\varphi$ could have been absorbed into $\Phi_h$;
it is needed only when $\Phi$ has no hyperbolic eigenvalues.)
Throughout the rest of the proof, we will assume that $\Phi_e$ is
present in the decomposition, i.e., $m'\geq 1$. It is not hard to see
from the proof that the argument goes through with minor modifications
when $\Phi_e$ is absent.

Set
\begin{equation}
\label{eq:eps1}
\eps_0=
\min\{\|\lambda_q\ell\|\mid
\lambda_q\ell\not\in\Z, 1\leq|\ell|\leq\ell_0, 1\leq q\leq m'\} >0,
\end{equation}
where $\|\cdot\|$ stands for the distance to the nearest integer. We
will require $\eps>0$ (and $\eta>0$) to be so small that
\begin{equation}
  \label{eq:pick-eps}
\eps\leq\eps_0\quad\textrm{and}\quad 2m'\eps<\eta<1/2.
\end{equation}
Furthermore, denote by $U\subset \T^{m'}$ the neighborhood of the unit
given by
$$
U=\big\{(e^{2\pi\sqrt{-1}\tau_1}, \ldots,
e^{2\pi\sqrt{-1}\tau_{m'}})\,\big|\, |\tau_1|<\eps, \ldots,
|\tau_{m'}|<\eps\big\},
$$
and set
$$
\theta=\big(e^{2\pi\sqrt{-1}\lambda_1}, \ldots,
e^{2\pi\sqrt{-1}\lambda_{m'}}\big)\in\T^{m'}.
$$

In what follows, to be consistent with Section \ref{sec:numbers}, we
will use additive notation for the group operations in $\T^{m'}$.

\begin{Lemma}
  \label{lemma:inequalities1}
  Assume that $k$ is such that $k\theta\in U$, or equivalently
\begin{equation}
\label{eq:eps2}
\|k\lambda_q\|<\eps \textrm{ for $q=1,\ldots, m'$},
\end{equation}
and that $k$ is divisible by the least common multiple $D$ of the
degrees of all roots of unity among the elliptic eigenvalues of
$\Phi$.  Let $d$ be the nearest integer to $\hmu\big(\Phi^k\big)$ or,
more explicitly,
\begin{equation}
  \label{eq:d}
  d=\big[k\hmu(\Phi)\big]=   k\big(\hmu(\varphi)+
  \hmu(\Phi_h)\big)+[k \hmu(\Phi_e)].
\end{equation}
Then conditions \ref{IRA0}--\ref{IRAsummand} are satisfied with
$\Phi:=\Phi_{ij}$, $k=k_{ijl}$, etc.
\end{Lemma}

Before proving the lemma, we note that \eqref{eq:eps2} has a bounded
gap sequence of solutions $k_l\to\infty$ by Lemma \ref{lemma:gap} and
then, as we have already pointed out, the corresponding sequence
$d=d_l\to\infty$ has bounded gap automatically. For any $a:=a_{ij}>0$
to find a bounded gap sequence $C_l\to\infty$ meeting condition
\ref{IRAa}, we can just pick any $C_l\not\in a\N$ in the range
$(k_la, k_la+\eta)$ assuming that $\eta<a$. Furthermore, to make sure
that $k$ is divisible by $D$ or any fixed integer $N>0$, it suffices
to replace $\eps$ by $\eps/N$ in the definition of $U$. Then, after
applying Lemma \ref{lemma:gap} for this new neighborhood $U$, we
replace sequence of solutions $k$ by $Nk$ and have a bounded gap
sequence of solutions satisfying \eqref{eq:eps2} and divisible by
$N$. As the next step, to meet the requirement of \eqref {eq:d} we
have $d$ replaced by $Nd$, and hence also divisible by $N$, provided
that $\eps>0$ is sufficiently small.  Thus, to complete this step of
the proof, it remains to prove Lemma \ref{lemma:inequalities1}. The
proof has a substantial overlap with \cite[Sec.\ 5.2.1]{GG:LS} and is
implicitly contained in that section.

\begin{proof}[Proof of Lemma \ref{lemma:inequalities1}]
  We start with conditions \ref{IRA0}--\ref{IRA-}, which we restate
  now for the reader's convenience:
\begin{itemize}
\item[\reflb{IRA0'}{\rm{(IR1'')}}]
  $\big|\hmu\big(\Phi^{k}\big)-d\big|<\eta$ and
  $$
  d-m
  \leq \mu_-\big(\Phi^{k}\big)
  \leq \mu_+\big(\Phi^{k}\big)
  \leq d+ m;
  $$
\item[\reflb{IRA+'}{\rm{(IR2'')}}]
  $\mu_\pm\big(\Phi^{k+\ell}\big)= d +
  \mu_\pm\big(\Phi^\ell\big)$  when $0<\ell\leq \ell_0$;

\item[\reflb{IRA-'}{\rm{(IR3'')}}]
  $\mu_+\big(\Phi^{k-\ell}\big)= d-
  \mu_-\big(\Phi^\ell\big)+\big(\beta_+\big(\Phi^\ell\big)-
  \beta_-\big(\Phi^\ell\big)\big)$, where $0<\ell\leq \ell_0<k$ and
  $\beta_+-\beta_-=b_+-b_-\leq m$ with $\beta_\pm=0$ when $\Phi^\ell$
  is non-degenerate.
\end{itemize}
In spite of a minor change of notations, these are exactly conditions
\ref{IRA0}--\ref{IRA-} stated for one map $\Phi=\Phi_{ij}$ with
$k=k_{ijl}$, etc.

Note that concatenation with the loop $\varphi$ simply results in the
shift of $d$ and the indices in each of these formulas by
$k\hmu(\varphi)\in\Z$. Hence, without loss of generality, we can
assume that $\varphi=\id$. Likewise, the effect of $\Phi_h$ is the
shift by $k\hmu(\Phi_h)$ and again we can assume that $\Phi_h$ is
trivial, i.e., $\Phi=\Phideg\oplus\Phi_e$.

To prove the first part of \ref{IRA0'}, observe that by
\eqref{eq:hmu-elliptic} and homogeneity of the mean index we have
\begin{equation}
\label{eq:hmu-Phik}
\hmu(\Phi^k) =2k\sum_q\lambda_q,
\end{equation}
since $\hmu(\Phideg)=0$.  By \eqref{eq:eps1} for $\ell=1$,
\eqref{eq:pick-eps} and \eqref{eq:eps2}, we have
\begin{equation}
  \label{eq:dj}
  d =\sum_q \sign(\lambda_q)\big[2k|\lambda_q|\big],
\end{equation}
where $[\,\cdot\,]$ stands for the nearest integer. Then
\begin{equation}
  \label{eq:hmu-dj}
\big|d-\hmu\big(\Phi^{k}\big)\big|< 2m'\eps<\eta,
\end{equation}
again by \eqref{eq:eps2}. This also shows that $d$ is
unambiguously defined.

To prove the second part of \ref{IRA0'}, we combine \eqref{eq:mu-del}
and \eqref{eq:hmu-dj} and use the triangle inequality. Thus we have
$$
\big|d-\mu_\pm\big(\Phi^{k}\big)\big|\leq\big|d-
\hmu\big(\Phi^k\big)\big|
+\big|\mu_\pm\big(\Phi^{k}\big)-\hmu\big(\Phi^k\big)\big| \leq
\eta+m'.
$$
Then, since the left-hand side is an integer,
$$
\big|d- \mu_\pm(\Phi^{k})\big|\leq m'\leq m,
$$
establishing the second inequality in \ref{IRA0'}.

Let us now turn to the proof of \ref{IRA+'} and \ref{IRA-'}. The fact
that
$$
\beta_+(\Phi)-\beta_-(\Phi)=b_+(\Phi)-b_-(\Phi)\leq m
$$
in \ref{IRA-'} follows directly from the definitions of these
invariants; cf.\ \eqref{eq:beta+-beta-}.

Next, observe that these formulas are additive in $\Phi$ in the
following sense. Assume that $\Phi=\Phi'\oplus\Phi''$, i.e.,
$$
\Phi(1)=\Phi'(1)\oplus\Phi''(1) \textrm{ and }
\hmu(\Phi)=\hmu(\Phi')+\hmu(\Phi'').
$$
Then
$$
\mu_\pm\big(\Phi^k\big)=
\mu_\pm\big((\Phi')^k\big)+\mu_\pm\big((\Phi'')^k\big)
$$
and 
$$
\beta_\pm\big(\Phi^k\big)=\beta_\pm\big((\Phi')^k\big)+
\beta_\pm\big((\Phi'')^k\big)
$$
for all $k\in\N$ by \eqref{eq:add} and \eqref{eq:add2}. Moreover, when
$k$ satisfies \eqref{eq:eps2} for $\Phi$, it also satisfies this
condition for $\Phi'$ and $\Phi''$. Then setting, as in Lemma
\ref{lemma:inequalities1} and \eqref{eq:d}, $d\big((\Phi')^k\big)$ to
be the nearest integer to $\hmu\big((\Phi')^k\big)$ and
$d\big((\Phi'')^k\big)$ to be the nearest integer to
$\hmu\big((\Phi'')^k\big)$ and using self-explanatory notation, we
have
$$
d\big(\Phi^k\big)= d\big((\Phi')^k\big)+d\big((\Phi'')^k\big).
$$
It follows that \ref{IRA+'} and \ref{IRA-'} hold for $\Phi$ whenever
they hold for $\Phi'$ and $\Phi''$ assuming that $k$ satisfies
\eqref{eq:eps2} for $\Phi$.

We will use this observation to reduce the proof to several particular
cases. Throughout the argument we will assume without specifically
stating that $k\in\N$ and $\ell\in\N$ and $k>\ell$ in \ref{IRA-'}.

When $\Phi$ is totally degenerate and $\hmu(\Phi)=0$, i.e.,
$\Phi=\Phideg$, we have $d=0$ and
$\mu_\pm\big(\Phi^k\big)=\mu_\pm(\Phi)=\pm \beta_\pm(\Phi)$ for all
$k\in \N$ by \eqref{eq:bpm}. Then \ref{IRA+'} reduces to the identity
$$
\beta_\pm\big(\Phi^{k+\ell}\big) =\beta_\pm(\Phi)=
\beta_\pm\big(\Phi^{\ell}\big)
$$
and, in a similar fashion, \ref{IRA-'} turns into
$$
\beta_+\big(\Phi^{k-\ell}\big) = -\big(-\beta_-
\big(\Phi^\ell\big)\big) +
\big(\beta_+\big(\Phi^\ell\big)-\beta_-\big(\Phi^\ell\big)\big).
$$
Therefore, \ref{IRA+'} and \ref{IRA-'} hold again for all $k$ and
$\ell$.

The remaining case is when $\Phi=\Phi_e$ is elliptic. We break it down
into two subcases.

\emph{Subcase 1: $\Phi$ is elliptic and strongly non-degenerate.}
Then, by definition, no eigenvalue of $\Phi(1)$ is a root of unity and
$\Phi^k$ is non-degenerate for all $k\in\N$. As a consequence,
$$
\mu_-\big(\Phi^k\big)=\mu\big(\Phi^k\big)=\mu_+\big(\Phi^k\big).
$$
By Lemma \ref{lemma:isospectral}, without loss of generality, we may
assume that map $\Phi(1)$ is semisimple. Indeed, consider an
isospectral (i.e., with constant eigenvalues) family $\Psi_s$,
$s\in [0,1]$, connecting $\Phi(1)$ to a semisimple linear symplectic
map with the same eigenvalues as $\Phi(1)$. Let $\tPhi_s\in \TSp(2m)$
be obtained by concatenating $\Phi$ and $\Psi$. Then
$\mu\big(\tPhi_s^k\big)=\mu\big(\Phi^k\big)$ and
$\hmu\big(\tPhi_s^k\big)=\hmu\big(\Phi_s^k\big)$ for all $k$ and $s$,
and hence in the proof of the theorem we can replace $\Phi$ by
$\tPhi$. We keep the notation $\Phi$.

Then, since $\Phi$ is semisimple, it decomposes as a direct sum of $m$
``short'' rotations
\begin{equation}
  \label{eq:rot}
R_q(t):=\exp(2\pi\sqrt{-1}\lambda_q t)\colon \C\to \C,
\end{equation}
where $0<|\lambda_q|\leq 1/2$, $q=1,\ldots, m'$, and $t$ ranges from 0
to $1$. Since no eigenvalue is a root of unity all $\lambda_q$ are
irrational.  Furthermore, $\hmu(R_q)=2\lambda_q$ and
$\mu(R_q)=\sign(\lambda_q)$ and $\hmu(\Phi)=\sum\hmu(R_q)$. Therefore,
by additivity, we only need to establish \ref{IRA+'} and \ref{IRA-'}
for one rotation $R_q$ to complete the proof of this subcase. For the
sake of brevity, set $R:=R_q$ and $\lambda:=\lambda_q$. Then these
identities amount to
\begin{equation}
  \label{eq:IR-SP2}
  \mu\big(R^{k\pm\ell}\big)=
  \big[\hmu \big(R^{k}\big)\big]\pm \mu\big(R^{\ell}\big),
\end{equation}
where $1\leq\ell\leq\ell_0$ and
\begin{equation}
  \label{eq:k-eps-3}
  \|k\lambda\|< \min\{\|\lambda\ell\|\mid
  1\leq \ell \leq \ell_0\}<1/2,
\end{equation}
as a combination of \eqref{eq:eps1}, \eqref{eq:pick-eps} and
\eqref{eq:eps2}. By \eqref{eq:sp2},
\[
  \hmu\big(R^k\big)=2k\lambda \quad \textrm{and} \quad
  \mu\big(R^k\big)=\sign(\lambda) \big(2\lfloor k|\lambda|\rfloor
  +1\big).
\]
Thus, \eqref{eq:IR-SP2} turns into the identity
\[
\sign(\lambda) \big(2\lfloor (k\pm\ell)|\lambda|\rfloor
+1\big)=[2k\lambda]\pm \sign(\lambda) \big(2\lfloor
\ell|\lambda|\rfloor +1\big).
\]
It is routine to see that this condition is indeed satisfied whenever
\eqref{eq:k-eps-3} holds; see, e.g., \cite[Sec.\ 5.2.1]{GG:LS} for a
detailed argument. This finishes the proof of the first subcase.

\emph{Subcase 2: Roots of unity.} In this subcase all eigenvalues of
$\Phi$ are roots of unity. This is the only case where it matters that
$k$ is divisible by $D$, the common multiple of their degrees, which
we will assume from now on. Hence, $\Phi^k$ is totally degenerate and
$d=\hmu\big(\Phi^k\big)$. Fixing $\ell$, let us decompose $\Phi$ as
$\Phi_0\oplus\Phi_1$, where the degrees of the eigenvalues of
$\Phi_0(1)$ divide $\ell$ and the degrees of the eigenvalues of
$\Phi_1(1)$ do not. In other words, $\Phi_1^\ell$ and
$\Phi_1^{k\pm\ell}$ are non-degenerate, and $\Phi_0^\ell$ and
$\Phi_0^{k\pm\ell}$ are totally degenerate.

We claim that 
\begin{equation}
  \label{eq:mu-Phi1}
  \mu\big(\Phi_1^{k\pm\ell}\big)=
  \hmu\big(\Phi_1^{k}\big)\pm\mu\big(\Phi_1^{\ell}\big).
\end{equation}
In spite of its appearance, this statement is not entirely obvious and
requires a proof. An easy way to prove \eqref{eq:mu-Phi1} is as
follows. By Lemma \ref{lemma:isospectral}, there exists an isospectral
family $\Psi_s$, $s\in [0,1]$, of transformations with $\Psi_1=\Phi_1$
and $\Psi_0(1)$ semisimple. Then $\Psi_0^{k}(1)=\id$, i.e.,
$\Psi_0^{k}$ is a loop, and
$\hmu\big(\Phi_1^{k}\big)=\hmu\big(\Psi_0^{k}\big)$ because all
eigenvalues of $\Psi_s^{k}(1)$ are equal to 1 for all $s$.
Furthermore, $\Psi_s^{k\pm\ell}$ and $\Psi_s^{\ell}$ remain
non-degenerate for all $s$. As a consequence,
$\mu\big(\Phi_1^{\ell}\big)=\mu\big(\Psi_0^{\ell}\big)$ and
$\mu\big(\Phi_1^{k\pm\ell}\big)=\mu\big(\Psi_0^{k\pm\ell}\big)$. It
follows that \eqref{eq:mu-Phi1} holds if and only if it holds for
$\Psi_0$. The latter, however, is obvious:
$$
\mu\big(\Psi_0^{k\pm\ell}\big)
=\mu\big(\Psi_0^{k}\Psi_0^{\pm\ell}\big)
=\hmu\big(\Psi_0^{k}\big)\pm\mu\big(\Psi_0^{\ell}\big)
$$
since $\Psi_0^k$ is a loop. Let us note that one can derive
  both \eqref{eq:IR-SP2} and \eqref{eq:mu-Phi1} directly from
  \cite{Lo} (see Thm. 3.1 and Cor. 3.1 in Chap. 9) bypassing Lemma
  \ref{lemma:isospectral}.

Let $q$ be the least common multiple of the degrees of eigenvalues of
$\Phi_0$. Then $\Phi_0^q$ is totally degenerate. Furthermore, $q$
divides both $k$ and $\ell$, and
$\Phi_0^\ell=\big(\Phi_0^q\big)^{\ell/q}$ and
$\Phi_0^{k+\ell}=\big(\Phi_0^q\big)^{(k+\ell)/q}$ are admissible
iterates of $\Phi_0^q$.  Hence
$$
\beta_\pm\big(\Phi_0^{k+\ell}\big)=\beta_\pm(\Phi_0^q)
=\beta_\pm\big(\Phi_0^\ell\big)
$$
by \eqref{eq:beta-it} applied to $\Phi_0^q$. Therefore, using
\eqref{eq:mu-Phi1} in the third line below, we have
\begin{equation*}
  \begin{split}
    \mu_\pm\big(\Phi^{k+\ell}\big) &=
    \mu_\pm\big(\Phi_1^{k+\ell}\big)+\mu_\pm\big(\Phi_0^{k+\ell}\big) \\
    &= \mu\big(\Phi_1^{k+\ell}\big)+
    \hmu\big(\Phi_0^{k+\ell}\big)\pm
    \beta_\pm\big(\Phi_0^{k+\ell}\big) \\
    &= \hmu\big(\Phi_1^{k}\big)+\mu\big(\Phi_1^{\ell}\big)+
    \hmu\big(\Phi_0^{k+\ell}\big)\pm
    \beta_\pm(\Phi_0^{\ell}) \\
    &= \big(\hmu\big(\Phi_1^{k}\big)+\hmu\big(\Phi_0^{k}\big)\big)
    +\mu\big(\Phi_1^{\ell}\big)+\big(\hmu\big(\Phi_0^{\ell}\big)\pm
    \beta_\pm\big(\Phi_0^\ell\big)\big) \\
    &= d+ \mu_\pm\big(\Phi^\ell\big),
\end{split}
\end{equation*}
proving \ref{IRA+'}. Condition \ref{IRA-'} is established by a similar
calculation:
\begin{equation*}
  \begin{split}
    \mu_+\big(\Phi^{k-\ell}\big) &=
    \mu_+\big(\Phi_1^{k-\ell}\big)+\mu_+\big(\Phi_0^{k-\ell}\big) \\
    &= \hmu\big(\Phi_1^{k}\big)-\mu\big(\Phi_1^{\ell}\big)+
    \hmu\big(\Phi_0^{k-\ell}\big)+
    \beta_+\big(\Phi_0^{k-\ell}\big) \\
    &= \big(\hmu\big(\Phi_1^{k}\big)+\hmu\big(\Phi_0^{k}\big)\big)
    -\mu\big(\Phi_1^{\ell}\big)-\big(\hmu\big(\Phi_0^{\ell}\big)-
    \beta_-\big(\Phi_0^\ell\big)\big)\\
    &\quad +\big(\beta_+\big(\Phi_0^\ell\big) -
    \beta_-\big(\Phi_0^\ell\big)\big)\\
    &= d- \mu_-\big(\Phi^\ell\big)+
    \big(\beta_+\big(\Phi^\ell\big) -\beta_-\big(\Phi^\ell\big)\big).
\end{split}
\end{equation*}
In the last two lines we have used the fact that $\Phi_1^\ell$ is
non-degenerate, and hence we can replace $\Phi_0^\ell$ by $\Phi^\ell$
when calculating $\beta_\pm$. This completes the proof of Subcase 2.

To finish the proof of the lemma, it remains to verify condition
\ref{IRAsummand} asserting in this case that \ref{IRA0'}, \ref{IRA+'}
and \ref{IRA-'} with the same $k$ and $d$ continue to hold with $\Phi$
replaced by its non-degenerate part $\Psi$. To this end, we note that
the elliptic eigenvalues of $\Psi(1)$ still satisfy
\eqref{eq:eps2}. Thus, since $\hmu(\Psi)=\hmu(\Phi)$, using the
self-explanatory notation as above, we have
$$
d\big(\Phi^k\big)= d\big(\Psi^k\big).
$$
Then a simple examination of the rest of the proof shows that
\ref{IRA0'}--\ref{IRA-'} still hold for $\Psi$. This completes the
proof of Lemma \ref{lemma:inequalities1}.
\end{proof}

\subsubsection{The general case: $r\geq 1$.}
\label{sec:CGT-pf-gen} 
Let $(\Phi_{ij}, a_{ij})$ be as in the statement of the theorem. We
denote the total number of maps $\Phi_{ij}$ by $r$. As we have already
pointed out, if we apply the argument from Section
\ref{sec:IRT-pf-r=1} to each of these maps individually, we obtain
integer sequences $l\mapsto k_{ijl}$ and integer sequences $d_{ijl}$
such that conditions \ref{IRA0}--\ref{IRAa} are satisfied with
$d_{ijl}$ depending on $j$ and $C_{ijl}$ depending on $i$ and
$j$. Thus our goal in this section is to modify the process to remove
this dependence.  This is where Lemma \ref{lemma:Mink} enters the
proof.

As in Section \ref{sec:IRT-pf-r=1}, denote the elliptic eigenvalues of
$\Phi_{ij}$ by $\exp(\pm 2\pi\sqrt{-1}\lambda_{ijq})$, where
$0<|\lambda_{ijq}|\leq 1/2$ and $q=1,\ldots, m_{ij}'$, and for each
pair of complex conjugate eigenvalues we pick the sign of
$\lambda_{ijq}$ so that the resulting eigenvalue is of the first Krein
type; see \cite{SZ}. Set $\Delta_{ij}:=\hmu(\Phi_{ij})$ and let $D$ be
the least common multiple of the degrees of the roots of unity among
the eigenvalues of all $\Phi_{ij}$.

Fix $\eps>0$ and $\sigma>0$, which we will eventually need to make
sufficiently small, and consider the system of inequalities
\begin{equation}
\label{eq:k}
\begin{aligned}
  \|k_{ij}\lambda_{ijq}\|& <\eps
  \quad\textrm{ for all $i$, $j$ and $q$,}\\
  |k_{11}a_{11}-k_{ij}a_{ij}| &< \sigma
  \quad\textrm{ for $i+j\geq 2$},
    \end{aligned}
\end{equation}
where we treat the integer vector $\vec{k}\in\Z^r$ with components
$k_{ij}$ as a variable. Introducing additional integer variables
$c_{ijq}$, we can rewrite the first group of inequalities in the form
\begin{equation}
  \label{eq:ciq}
|k_{ij}\lambda_{ijq}-c_{ijq}|<\eps.
\end{equation}
With this in mind, system \eqref{eq:k} has one fewer equation than the
number of variables.

This reformulation puts the question within the framework of Lemma
\ref{lemma:Mink}. Namely, thinking for a moment of $k_{ij}$ and
$c_{ijq}$ as real coordinates on $\R^n$ with $n=r+\sum m'_{ij}$, we
can view \eqref{eq:k} as a system in the form of \eqref{eq:system},
where $f_s=k_{11}a_{11}-k_{ij}a_{ij}$ for $s=1,\ldots, r-1$ and the
remaining linear maps are $k_{ij}\lambda_{ijq}-c_{ijq}$. By Lemma
\ref{lemma:Mink}, there exists a bounded gap sequence $K_l\to\infty$
of non-zero solutions of \eqref{eq:k} in $\Z^n$. Projecting to the
first $r$-coordinates, we obtain a bounded gap sequence
$\vec{k}_l\in\Z^r$ solving \eqref{eq:k}.  Moreover, we can make all
$k_{ijl}$ divisible by any fixed integer $N$ as in the lemma. In
particular, we can have all $k_{ijl}$ divisible by $D$.

Our next goal is to ensure that $k_{ijl}\to\infty$ as $l\to\infty$ for
all $i$ and $j$. To this end, we need to require that $\eps<1$ and
$\sigma<\min a_{ij}$. (Recall that by the assumptions of the theorem
$a_{ij}>0$ for all $i$ and $j$.) First, note that the sequence
$\vec{k}_l$ is necessarily unbounded; for otherwise the entire
sequence $K_l$ would be bounded due to \eqref{eq:ciq}. Furthermore,
$\vec{k}_l\neq 0$ since otherwise we would have $K_l=0$ by again
\eqref{eq:ciq}. (Here we use the fact that $0<\eps<1$.) In particular,
for every $l$ at least one of the integers $k_{ijl}\neq 0$.  Then by
the second inequality in \eqref{eq:k}, for a fixed $l$, all $k_{ijl}$
are non-zero and have the same sign when $\sigma<\min
a_{ij}$. Therefore, replacing $\vec{k}_l$ by $-\vec{k}_l$ when
necessary, we can guarantee that $k_{ijl}>0$ for all $i$, $j$ and
$l$. Clearly, after this sign change the sequence still has bounded
gap. Finally, since the norm of $\vec{k}_l$ goes to infinity, we must
have $k_{ijl}\to\infty$ for all $i$ and $j$ again by the second
inequality in \eqref{eq:k}.

It follows that when $\eps>0$ is sufficiently small the requirements
of Lemma \ref{lemma:inequalities1} are met for each $\Phi_{ij}$, and
hence conditions \ref{IRA0}--\ref{IRAsummand} hold with
$$
d_{ijl}:=[ k_{ijl}\Delta_{ij}].
$$
More specifically, fix $\ell_0$ and $\eta>0$ which we assume to be
sufficiently small (e.g., $\eta<1/4$) and similarly to
\eqref{eq:eps1}, set
$$
\eps_0=
\min_{i,j,q}\{\|\lambda_{ijq}\ell\|\mid
\lambda_{ijq}\ell\not\in\Z, 1\leq|\ell|\leq\ell_0\} >0.
$$
Furthermore, we let $\eps>0$ be so small that again
$$
\eps\leq\eps_0\quad\text{and}\quad 2\eps\max_{i,j} m'_{ij} <\eta
$$
as in \eqref{eq:pick-eps}. Then \eqref{eq:eps2} is satisfied for each
$\Phi_{ij}$.

Next, we claim that $d_{ijl}$ is in fact independent of $j$ when
$\sigma>0$ is sufficiently small. To see this, let us focus on one
particular cluster, i.e., fixing $i$ we have the ratio
$\rho_i=\Delta_{ij}/a_{ij}>0$ independent of $j$. First, note that as
a consequence of the second condition in \eqref{eq:k},
$$
|k_{ij'l}a_{ij'}-k_{ijl}a_{ij}| < 2\sigma.
$$
Therefore, multiplying this inequality by $\rho_i>0$, we have
$$
|k_{ij'l}\Delta_{ij'}-k_{ijl}\Delta_{ij}| < 2\sigma\rho_i.
$$
It follows that $d_{il}:=d_{ijl}$ is independent of $j$ for all $i$
when $\sigma\max_i\rho_i<1/8$.

Turning to \ref{IRAa},  we can pick any 
$$
C_l\in(\max_{ij} k_{ijl}a_{ij}, \, \min_{ij} k_{ijl}a_{ij}+\eta)
$$
lying outside the spectrum (i.e., the union of the sequences
$a_{ij}\N$) assuming again that $\sigma$ is so small that
$2\sigma<\eta$ and $\eta<\min a_{ij}$.

We have already shown that the sequences $k_{ijl}$ can be made
divisible by any given integer $N$. This, however, does not
automatically guarantee that $d_{il}$ is divisible by $N$. To achieve
this, we argue exactly as in the case of one map. Namely, we first
replace the upper bounds $\eps$ and $\sigma$ in \eqref{eq:k} by
$\eps/N$ and $\sigma/N$ and then multiply the resulting sequences
$k_{ijl}$ by $N$. (This also makes $k_{ijl}$ divisible by $N$,
although we have already guaranteed this by other means.) Then the
sequence $d_{il}$ gets multiplied by~$N$.

\end{document}